%% file: ndpg.tex
\pdfoutput=1
\documentclass[12pt, oneside, reqno]{amsart}

\title[Nonlinear residual loss functions for neural networks]{
  Variationally correct loss functions for nonlinear
  parametric problems and their neural approximations}

\author{Pablo Cort{\'e}s Castillo}
\address{Portland State University, PO Box 751, Portland OR 97201}
\email{pcortes@pdx.edu}

\author{Wolfgang Dahmen}
\address{University of South Carolina, 1523 Greene Street, 29208 Columbia, South
  Carolina}
\email{dahmen@math.sc.edu}

\author{Jay Gopalakrishnan}
\address{Portland State University, PO Box 751, Portland OR 97201}
\email{gjay@pdx.edu}

\subjclass[2020]{65N30, 65N15, 35J62, 68T07}
\keywords{nonlinear parametric PDEs, quasilinear fluxes, neural networks, Kirchhoff model, a posteriori estimator, DPG, FOSLS}

\input{preamble}
\makeatletter
\renewcommand{\@captionfont}{\normalfont\footnotesize}
\makeatother
\renewcommand{\fp}{\varPi}

\begin{document}
\begin{abstract}
  We construct computable {\em a posteriori} error estimators for
  parametric families of nonlinear elliptic problems by measuring the
  residual of a variational formulation in the dual norm of a closed
  subspace of the test space.  A single abstract theorem, requiring
  G\^ateaux differentiability of the nonlinear operator, a
  uniform bound on its derivative, a local nonlinear inf-sup
  condition, and a Fortin compatibility condition, shows that the
  resulting estimator is bounded above and below by the error in the
  trial norm, i.e., that it is a variationally correct loss function.
  Two instances are covered. The first
  requires the test space to be $L_2$-like, yielding a nonlinear
  first-order system least-squares loss function whose two-sided bound needs no
  mesh and therefore certifies arbitrary, possibly meshless,
  approximations.  The second breaks the test space, so that the dual
  norm is computed by independent element-local problems, yielding a
  nonlinear discontinuous Petrov--Galerkin loss function. All assumptions are verified in detail for several classes of examples. A class of
  quasilinear diffusion problems covers both a monotone
  nonlinear flux model and a semilinear problem whose nonlinearity
  sits in lower-order terms.  A nonlocal Kirchhoff model offers a
  qualitatively different example with stronger differentiability properties.
  Numerical studies train neural network approximations of the
  parameter-to-solution maps with the discontinuous Petrov--Galerkin
  losses and measure the ratio of loss to squared error.
\end{abstract}
\maketitle

\section{Introduction}
\label{sec:intro}
The twin advances made in this paper concern (i) new {\em a
  posteriori} error estimators for parametric families of nonlinear
elliptic problems, obtained by measuring a residual in a dual norm,
and (ii) variationally correct loss functions for neural network (NN)
approximations of the associated parameter-to-solution maps.  The two
are one and the same: the global error estimators are variationally correct loss functions, so we shall use the two terms interchangeably. The difference lies in
applications. In automatic adaptivity, one localizes the error estimators and uses them in an adaptive loop (consisting of the standard solve, estimate, mark, and refine steps). In
machine learning applications, one uses the loss function both for training and as an error certificate for NN predictions.
In this paper we are driven more by the modern NN applications, but the new theory we develop applies to both.

\subsection{Learning parameter-to-solution maps}

A common approach to modeling uncertain physical processes is through
systems of {\em parameter-dependent} partial differential equations
(PDEs).  Inference on system states, such as calibration of uncertain
coefficients or identification of  constitutive laws from observational data,
is greatly facilitated by having an approximation to the {\em
  parameter-to-solution map} in hand.  This motivates the vibrant area
of {\em operator learning}.
We shall see many examples later where a parametric PDE is reformulated into a nonlinear operator equation
$\Aa(\clu) = \ell$ on appropriate spaces,
where $\alpha$ is a high-dimensional parameter varying in some compact set $\cl Q$ and $\Aa$ involves partial derivatives with respect to the components of a spatial variable $x$ in a low-dimensional spatial domain $\om$. The map $\alpha \mapsto \clu$ is what we refer to as the parameter-to-solution map.
Since the parameter domain is typically
high- or even infinite-dimensional, operator learning extends established methods for solving a
single PDE  to learn the entire map.

Two paradigms currently dominate operator learning.  The first, purely
{\em data-driven}, requires computing high-fidelity solutions for a
large number of parameter samples and then applying regression in
function space over a suitable hypothesis class of operator
representations, see, e.g.,~\cite{Stu1,Stu2}.  The second---{\em
  residual-based}, or {\em physics-informed}---avoids high-fidelity
solution computations altogether by minimizing a residual loss over
the hypothesis class, see, e.g.,~\cite{BDO,CDG,Kar2,QCD,Kar1}.  The
latter approach aligns with this paper, and its computational
advantages are most pronounced when the PDE is {\em nonlinear}, since
each high-fidelity solve is then particularly expensive.  On the other
hand, hypothesis classes best suited to high-dimensional parameter
domains are themselves nonlinear, rendering the corresponding
optimization non-convex.  Consequently, it is generally uncertain
whether optimization can fully exploit the approximation power of the
hypothesis class, and the only available accuracy assessment at
termination is the loss itself.

\subsection{Variational correctness}
\label{ssec:vc-definition}

Since the error cannot be computed and the loss is all one has at termination, it is worth stating
precisely what one wants of it.  Fix a parameter $\alpha$, let
$u(\alpha)$ denote the exact solution of the PDE at that parameter,
let $\| \cdot \|$ denote a norm in which one wishes to measure the
error---a {\em model-compliant} norm, such as an energy norm for a
diffusion problem---and let $\cl H$ denote a hypothesis class of trial
functions over which the loss is minimized.  We call a loss function
$w \mapsto \cl L_\alpha( w)$ {\em variationally correct} \cite{BDO, CDG} if there are
constants $0 < c \le C < \infty$, independent of $\alpha$ and of any
discretization parameter, such that
\begin{equation}
  \label{ideally}
  c\, \| u(\alpha) - w \|^2
  \;\le\; \cl L_\alpha( w)
  \;\le\; C\, \| u(\alpha) - w \|^2
  \qquad \text{ for all } w \in \cl H .
\end{equation}
The value of the loss is then a certificate of accuracy: it is small
if and only if the error is small.  Two features
of~\eqref{ideally} deserve emphasis.  First, it is required to hold at {\em
  every} member of~$\cl H$, not merely at a minimizer, because during
training the network output is an arbitrary member of the class and
satisfies no Galerkin or Euler--Lagrange relation.  Second,  its constants
must not degrade as the discretization is refined or as $\alpha$
ranges over the parameter set, since otherwise the certificate says
nothing uniform.

A fundamental difficulty is that in infinite dimensions a small
residual, measured in a convenient norm, need not imply a small error
in a model-compliant norm.  For {\em linear} PDEs this has led to the
formulation of variationally correct residual
losses~\cite{BDO,CDG,QCD}, resting on the clean conceptual platform of
Babu\v{s}ka's theorem, which supplies stable
variational formulations and hence~\eqref{ideally}.  For {\em
  nonlinear} problems, which are the focus of this paper, matters are
substantially harder.  Existence and uniqueness of solutions may
require strong structural hypotheses (monotonicity, coercivity), and
bifurcations can give rise to multiple solution branches.  When
stability holds, it is typically only {\em local}, i.e., valid only in
a neighborhood of a {\em regular solution}, one at which the derivative of
the operator is well behaved; then~\eqref{ideally} can be expected
only for $w$ near $u(\alpha)$.  Beyond stability, a challenge
inherited from the linear case remains: variationally correct residual
losses involve {\em dual norms}, which are generally intractable to
evaluate exactly.

\subsection{Which losses are variationally correct?}
\label{sec:vc-verification}

Before describing what we do, we survey the losses in common use and
ask which of them satisfy~\eqref{ideally}.
We shall use  the standard
spaces on the low-dimensional spatial domain, such as
$H^1 = \{ w \in L_2(\om): \partial_i w \in L_2(\om)\}$,
its dual $H^{-1}$,  and
$\Hdiv = \{ r \in L_2(\om)^d: \dive r \in L_2(\om)\}$.
For any space $X$ and its dual $X'$, the action of a functional $f \in X'$ on $v \in X$ is written either as $f(v)$ or as $\ip{f, v}_X$, where the subscript is dropped when no confusion can arise.
We also omit $\om$ from the notation for spaces if no confusion can arise.
To keep the discussion
concrete we use one model problem: on a bounded domain
$\om \subset \bb R^d$, given $f \in L_2(\om)$ and a
parameter-dependent flux function $\sigma_\alpha : \bb R^d \to \bb
R^d$, find $u$ with
\begin{equation}
  \label{eq:intro-model}
  \dive (\sigma_\alpha( \nabla u)) = f \;\text{ in } \om,
  \qquad
  u = 0 \;\text{ on } \partial \om,
\end{equation}
i.e., $u \in \Ho^1$ with
$-(\sigma_\alpha(\nabla u), \nabla v)_\om = (f, v)_\om$ for all
$v \in \Ho^1$, where $(\cdot,\cdot)_\om$ is the $L_2(\om)$ inner
product and $\Ho^1$ is the Sobolev space of $L_2$ functions
with $L_2$ gradient and vanishing trace, normed by the
model-compliant norm $u \mapsto
\| \nabla u \|_{L_2}$.
We assume that $\sigma_\alpha$ satisfies the hypotheses
in \S\ref{ssec:nonlinear-diffus-flux-model}
(strong monotonicity and Lipschitz  continuity)  given later.
Under those assumptions~\eqref{eq:intro-model}
has a unique solution $u = u(\alpha)$. The
model-compliant norm is $\| \nabla ( u - w) \|_{L_2}$.

\medskip
\noindent{\em The strong residual.}  The strong-form loss
associated with~\eqref{eq:intro-model}, ubiquitous in collocation
PINNs~\cite{Kar1}, is
$\| f - \dive(\sigma_\alpha(\nabla w)) \|^2_{L_2}$.  It measures
a second-order residual in $L_2$, whereas the natural residual of
the weak problem lies in $H^{-1}$. No mesh-independent
equivalence can repair the mismatch.  Even for the
one-dimensional linear prototype, $u''=0=f$ on $\om = (0,\pi)$,
with $\sigma_\alpha(p) = -p$ in \eqref{eq:intro-model}, whose exact
solution is $u = 0$, for candidates $w_n(x) = n^{-2}\sin(nx)$, $n \ge 1$, the loss equals
$\| w_n'' \|^2_{L_2} = \pi/2$ for every $n$, whereas the
error $\|\nabla(u - w_n)\|^2_{L_2} = \|w_n'\|^2_{L_2} = \pi/(2n^2)$
tends to zero as $n \to \infty$.  The upper bound of~\eqref{ideally} therefore fails.

\medskip
\noindent{\em The energy functional.}  Suppose the flux derives from
an energy density, i.e., $\sigma_\alpha(p) = -\varphi_\alpha'(|p|)\,
p/|p|$ for a convex profile $\varphi_\alpha$.  The functional whose
Euler--Lagrange equation is the weak form of~\eqref{eq:intro-model} is
then
\begin{equation}
  \label{eq:energy-functional-num}
  \cl J_\alpha(w)
  := \int_\om \varphi_\alpha(|\nabla w|)\,dx - (f,w)_\om ,
\end{equation}
a valid training objective of Deep Ritz type~\cite{DeepRitz} whose
unique minimizer is $u$.  Indeed, the strong monotonicity and
Lipschitz continuity assumed in~\S\ref{ssec:nonlinear-diffus-flux-model} imply
\begin{equation}
  \label{eq:energy-gap-equivalence}
  \frac{\mu_1}{2}\, \|\nabla(u-w)\|_{L_2}^2
  \;\le\; \cl J_\alpha(w)-\cl J_\alpha(u)
  \;\le\; \frac{\mu_2}{2}\, \|\nabla(u-w)\|_{L_2}^2,
\end{equation}
so the energy {\em gap} is variationally correct.  It is not, however,
a computable certificate, because the offset $\cl J_\alpha(u)$ is
unknown: the raw energy can be minimized without knowing the offset,
but its magnitude cannot be converted into an error estimate or a
stopping tolerance, and replacing it by $|\cl J_\alpha(w)|$ neither
repairs this nor preserves the minimizer.  The energy method is thus a
legitimate optimization baseline, but its raw value is not a loss
satisfying~\eqref{ideally}.

\medskip
\noindent{\em The natural residual and its localization.}  For
$w \in \Ho^1$ the residual of the weak form,
\begin{equation}
  \label{eq:conforming-residual}
  \ip{\operatorname{Res}_\alpha(w), v}
  := (f,v)_\om + (\sigma_\alpha(\nabla w), \nabla v)_\om,
  \qquad v \in \Ho^1,
\end{equation}
then satisfies
\begin{equation}
  \label{eq:conforming-residual-equivalence}
  \frac{1}{\mu_2}\, \|\operatorname{Res}_\alpha(w)\|_{H^{-1}}
  \;\le\; \|\nabla(u-w)\|_{L_2}
  \;\le\; \frac{1}{\mu_1}\, \|\operatorname{Res}_\alpha(w)\|_{H^{-1}},
\end{equation}
by the techniques in \S\ref{ssec:nonlinear-diffus-flux-model} with  the
$\mu_i$ given there (use  the fundamental theorem of calculus along the segment
from $w$ to $u$ together with the flux bounds established there).
The continuous $H^{-1}$ residual is therefore variationally
correct---but its exact evaluation is an infinite-dimensional,
globally coupled Riesz problem, so it is not computable.  The standard computable
surrogate is the residual estimator: with $X_h \subset \Ho^1$ a
conforming finite element space on a shape-regular simplicial mesh
$\oh$ of $\om$, whose elements $K$ have diameters $h_K$ and whose
interior facets $F$ have diameters $h_F$ and unit normals $n_F$, put
$R_K(w_h) = f - \dive(\sigma_\alpha(\nabla w_h))|_K$ for the
elementwise residual and
$J_F(w_h) = [\![\sigma_\alpha(\nabla w_h)]\!]\cdot n_F$ for the jump
of the flux across $F$, and set
\begin{equation}
  \label{eq:conforming-local-estimator}
  \cl S_\alpha(w_h)^2
  := \sum_{K\in\oh} h_K^2 \|R_K(w_h)\|_{L_2(K)}^2
  + 2\sum_{F} h_F \|J_F(w_h)\|_{L_2(F)}^2 .
\end{equation}
For the Galerkin solution the usual localization argument yields
reliability and efficiency of $\cl S_\alpha$ (see,
e.g.,~\cite[Chap.~34]{ErnGuermondFEII} and, for nonlinear elliptic
problems, \cite{Verfurth1994}) but not for {\em all elements} of $X_h$. Since
$\cl S_\alpha$ is not reliable on the whole $X_h$,  it fails~\eqref{ideally}: e.g., on a
quasi-uniform mesh with $f \ne 0$, the admissible output $w_h = 0$
gives $\cl S_\alpha(0)^2 = \sum_K h_K^2 \|f\|^2_{L_2(K)} \le h^2
\|f\|^2_{L_2}$, which tends to zero with $h$, while the error
$\|\nabla u\|_{L_2}$ stays fixed.  This is exactly the failure
that disqualifies classical residual estimators as losses over a
neural network trial class.

\medskip
\noindent{\em A variationally correct conforming loss.}  What is
missing above is the part of the residual visible on the discrete test
space, i.e., the violation of the discrete Galerkin equations.  Adding
it back as an algebraic term
\begin{equation}
  \label{eq:algebraic-residual}
  \eta_{\mathrm{alg},\alpha}(w_h)
  := \sup_{0 \ne v_h \in X_h}
  \frac{|\ip{\operatorname{Res}_\alpha(w_h), v_h}|}{\|\nabla v_h\|_{L_2}},
\end{equation}
one obtains
\begin{equation}
  \label{eq:conforming-vc-loss}
  \cl L^{\mathrm{conf}}_{\alpha}(w_h)
  := \cl S_\alpha(w_h)^2 + \eta_{\mathrm{alg},\alpha}(w_h)^2,
\end{equation}
which does satisfy~\eqref{ideally} on all of $X_h$, up to the usual
data oscillation term.
This is
the analogue, in the present setting, of residual estimators for
violated Galerkin orthogonality~\cite{Angermann2002}.  Evaluating
$\eta_{\mathrm{alg},\alpha}$ amounts to one stiffness matrix solve.
Thus
$\cl L^{\mathrm{conf}}_{\alpha}(w_h)$ is indeed a  computable
variationally correct loss function, but
at the price of a global solve per parameter sample.

\medskip
\noindent{\em DPG and FOSLS losses.}  The two types of variationally correct losses
developed in this paper avoid that global solve.  The FOSLS-type losses are obtained by rewriting
the problem as a first-order system and measuring its residual in
$L_2$: with $q$ denoting the flux of the solution,
\begin{equation}
  \label{eq:nonlinear-fosls-loss}
  \cl L^{\mathrm{fos}}_{\alpha}(q_h,w_h)
  := \|q_h - \sigma_\alpha(\nabla w_h)\|_{L_2}^2
  + \|\dive q_h - f\|_{L_2}^2 .
\end{equation}
Strong monotonicity, integration by parts and the Friedrichs
inequality show that~\eqref{eq:nonlinear-fosls-loss}
satisfies~\eqref{ideally} on $\Hdiv \times \Ho^1$ with the graph
norm, where $\Hdiv$ collects the $L_2^d$ vector fields whose
divergence is in $L_2$. No Riesz solve of any kind is
needed, since $L_2$ is
its own dual; this is the {\em nonlinear counterpart of the standard FOSLS}
baseline~\cite{FOSLS_coercivity,BDO,CDG}.  The alternate approach keeps a weaker
trial norm and instead breaks the test space into element-local
pieces, so that the dual norm decouples into independent element-by-element
computations; this is the DPG methodology.  The two are developed,
under the names NFOSLS and NDPG, in the rest of the paper; the
functional~\eqref{eq:nonlinear-fosls-loss} reappears in
\S\ref{ssec:nfosls-diffusion} as an instance of the general framework.

These families control different---but in each case
model-compliant---norms, so their raw values are not comparable across
formulations.  The conforming loss controls the $\Ho^1$ energy
error and needs one global Riesz solve per sample.  The first-order
system loss controls an $\Hdiv \times \Ho^1$ graph error and asks
the trial class for a conforming flux variable, but evaluates by
quadrature alone.  The DPG loss instead controls the $L_2$ norm of the
error in the interior trial variables of lower regularity, at the cost of element-local solves: breaking the test space replaces a global Riesz
problem by independent local ones while retaining error control
throughout the trial class.  A further design opportunity in the DPG case, adapting
the test-space metric to the coefficient regime, has been
demonstrated for linear diffusion in~\cite{CDG} to gain parameter robustness.
Its assessment for
the nonlinear problems considered here is left to future work.

\subsection{Results of this paper}

We address the challenges of \S\ref{ssec:vc-definition} for parametric
families of nonlinear second-order elliptic problems.  The starting
point is the squared residual of the operator equation
$\Aa(\clu) = \ell$, measured in the dual of a test space $\Va$,
together with its computable surrogate obtained by replacing $\Va$
with a closed subspace on which the dual norm can actually be
evaluated.  A single abstract theorem
(Theorem~\ref{thm:NDPG-gen-nonlinear-dpg}) establishes a two-sided
bound of the form~\eqref{ideally} for this surrogate, up to a data
oscillation term, under four hypotheses: G\^ateaux
differentiability of the nonlinear operator, a uniform bound on its
derivative, a local nonlinear inf-sup condition, and a Fortin
compatibility condition tying the test subspace to the trial class.

The theorem covers two computationally distinct instances, arising
from the two ways of making the dual norm computable.  The first,
which we call the {\bf{NFOSLS loss function}}, takes the test subspace
to be all of $\Va$ and requires $\Va$ to be an $L_2$-type space
identified with its own dual, so that the residual is an integral  and
no Riesz solve is needed; the
first-order system reformulations that make this possible are the
familiar ones of least-squares finite element methods, and
\eqref{eq:nonlinear-fosls-loss} is an example.  The Fortin condition
is then trivially satisfied, no mesh enters the analysis, and the
resulting two-sided bound holds for every element of the trial
class---in particular for meshless neural network hypothesis classes.
The second, the {\bf{NDPG loss function}}, adopts the DPG methodology,
in which the test space is broken and the test subspace is a finite
element subspace of it.  It is inspired by the DPG approach
of~\cite{carstensen2017} to nonlinearities.
This instance uses {\em hybrid} trial spaces
that augment the interior unknowns with trace variables on element
interfaces, which is what renders the residual locally computable, and
it lowers the regularity demanded of the interior trial variables.

We verify all assumptions of the abstract theorem for two
applications of contrasting nonlinear character.  The first is a class
of quasilinear diffusion problems
$\dive \sigma_\alpha(\nabla u, u) + \gamma_\alpha(u) = f$ on a spatial
domain $\om$, with $\alpha$ ranging over a compact set, in which the
flux may depend on the gradient and on the value of the solution
simultaneously.  Two example cases are worked out in detail within it:
the monotone {\em nonlinear flux} model~\eqref{eq:intro-model}, where the
leading-order operator itself is nonlinear, and the {\em semilinear problem}
$-\dive(\kappa_\alpha\nabla u + g_\alpha(u)\, b_\alpha) +
\gamma_\alpha(u) = f$, where the
nonlinearity enters only through additive lower-order terms.
The second application is a stationary
Kirchhoff problem $-\dive(M_\alpha(G_\alpha(u))\kappa_\alpha \nabla u) = f$,
in which a parametric diffusion coefficient $\kappa_\alpha$ is
modulated by a factor depending on the solution {\em nonlinearly and
  nonlocally}, through the single number
$G_\alpha(u) = (\kappa_\alpha \nabla u, \nabla u)_\om$.  The two  applications illustrate complementary
regimes of the theory.

As already mentioned, our  results apply beyond the machine learning application that
currently  motivates us.
The obtained   a posteriori error estimators and the designed  framework apply
directly to adaptive finite element methods for nonlinear elliptic
problems, providing computable, two-sided error control of precisely
the kind needed to drive refinement in an adaptive algorithm.

\subsection{Relation to other work}

Nonlinear least-squares finite element methods have a longer history
than their DPG counterparts, and the NFOSLS instance of this paper
stands in that line.  The first analysis of a least-squares functional
for a nonlinear system appears to be~\cite{Boche97}, for the
Navier--Stokes equations, where the least-squares formulation is shown
to be well posed along a branch of nonsingular solutions.  Nonlinear
FOSLS was then developed for variably saturated subsurface flow
in~\cite{Stark00a}, together with the Gauss--Newton multilevel solvers
of~\cite{Stark00b}.  There, as in much of what followed, the
functional is linearized before it is minimized, and the equivalence
with the error is established for the linearized functional, hence
only for approximations already close to the solution.  An equivalence
of the {\em nonlinear} functional itself with the error, free of any
such closeness restriction, was obtained in~\cite{Stark07} for an
elastoplasticity model, using nonsmooth arguments, and the method was
carried to the harder problem of hyperelasticity
in~\cite{MulleStarkSchwa14}.  Adaptive least-squares methods with
built-in a posteriori error estimation have since been developed
extensively, so far mostly for linear problems, as
in~\cite{BringCarstStark18} for linear elasticity and
in~\cite{Bring26} for scaling-robust estimation with discontinuous
least-squares discretizations.  Two recent works take up the nonlinear
case.  A unified least-squares framework for nonlinear systems, again
built on a Gauss--Newton linearization, is developed
in~\cite{BertrBrodbRickSchne25}, where the functional is proved
coercive and continuous, and is used as an a posteriori error
estimator, in a neighborhood of the exact solution.  Closest to the
model class of~\S\ref{ssec:nonlinear-diffus-flux-model} below
is~\cite{BringPraet26}, which treats the first-order system
$p = \sigma(\nabla u)$, $-\dive p = f$ under ellipticity and
boundedness conditions on the derivative of $\sigma$ matching the ones
imposed there.  A Zarantonello linearization is minimized in a
weighted least-squares functional, the weights being determined by
those two constants, and the resulting two-sided estimate drives an
adaptive loop that is shown to converge R-linearly from arbitrary
initial guesses; the Zarantonello iteration serves there a purpose
close in spirit to the monotone splitting of the companion
paper~\cite{MDPGpaper} discussed below.  The two-sided bounds obtained
here are free of the linearized objective
and bound the nonlinear functional as it stands: the linearization enters only the analysis.

The theory in this paper, in addition to the standard DPG
theory~\cite{DPGacta, DemkoGopal11}, relies substantially
on~\cite{carstensen2017}, which introduced a direct nonlinear DPG
methodology, formulating the discrete problem as residual minimization
over a broken test space, and established abstract a priori
quasi-optimality and a posteriori error estimates.  A different route
to variationally correct losses, available when the nonlinearity can
be confined to a strongly monotone closure as in~\cite{CH}, is pursued
in the companion paper~\cite{MDPGpaper};
the losses obtained there are
globally variationally correct but require the monotone splitting,
whereas the losses of the present paper require only
conditions on the G\^ateaux derivative and are, in general, variationally correct
only near the solution branch (although in both our applications,
the local condition is in fact verified globally, so that global
variational correctness is actually obtained).

Very recently, DPG residual minimization for nonlinear problems has
drawn the attention of multiple researchers.  A minimum residual DPG
discretization of a semilinear elliptic problem is the subject of the
recent work~\cite{VeraHeuerPraet2026}, which introduces into the trial
space two auxiliary unknowns standing for the nonlinear terms, renders
the remaining equations linear, and enforces the auxiliary relations
by an $L_2$ least-squares penalty.  Its main result is a C\'ea-type
quasi-optimality estimate, conditional on the existence of an exact
solution and a closeness
hypothesis. We treat the same model problem as an example
in~\S\ref{sec:semilinear}  (and we offer a detailed comparison in \S\ref{ssec:quasilin-discussion}). Our theory yields the new
two-sided bound for this example at the  price of a  smallness condition
relating the Lipschitz constants of the lower-order nonlinearities to
the diffusion coefficient.

Residual minimization in dual norms has
also recently been extended to nonlinear operator equations in Banach
spaces in~\cite{MugaPererRojasRuiz26}, with the $p$-Laplacian as a
model problem: there, a residual representative computed through
duality maps furnishes an a posteriori error estimator that drives
mesh adaptivity. This example is not covered by our theory (see Remark~\ref{rem:p-Laplacian}).
The case $p\ne 2$ genuinely motivates their Banach-space setting.
For $p\ge2$, strong monotonicity yields a one-sided estimate
bounding the $(p-1)$th power of the $W_0^{1,p}$ error norm by
the $W^{-1,p'}$ residual norm. This does not imply a two-sided
equivalence between these quantities and falls outside the
residual--error equivalence considered here, but
provides an interesting example of a residual-based
loss function that deserves a separate treatment.
Closer in spirit to the present paper, in addition to our own prior
work~\cite{CDG}, is the recent  a posteriori analysis of~\cite{FuhrRojas25}
for linear well-posed problems, which shows that the error of an
arbitrary approximation, such as a neural network output, is
equivalent to the sum of a computable residual in a discrete dual norm
and a remainder that can in turn be bounded by computable estimators
of finite element type: including both parts in the training loss then
provides error control during optimization.  Such error control is
also our goal in this paper, but for the nonlinear models described
above.

For the nonlocal Kirchhoff model of \S\ref{sec:kirchhoff-ndpg}, a
residual-based a posteriori estimator and a convergent adaptive
algorithm were obtained in~\cite{PalGudi25}, following the a priori
analysis of~\cite{Gudi12}.  The estimates of
\S\ref{sec:kirchhoff-ndpg} differ in that they are two-sided at
{\em arbitrary} trial functions, rather than at the Galerkin solution,
and hold with explicit constants uniform in the parameter $\alpha$.

The paper is organized as follows.  Section~\ref{sec:NDPG} sets up the
abstract framework: it defines the computable residual loss, states the
assumptions it rests on---G\^ateaux differentiability of the nonlinear
operator, a uniform bound on its derivative, a local nonlinear inf-sup
condition and a Fortin compatibility condition---and proves in
Theorem~\ref{thm:NDPG-gen-nonlinear-dpg} that the resulting loss is bounded
above and below by the true error in the trial norm. We close this section by
introducing the least-squares and broken test space instances, giving the
NFOSLS and NDPG losses.
Section~\ref{sec:nonlinear-diffusion} specializes the framework to
quasilinear diffusion with nonlinear fluxes and lower-order terms, verifying
every assumption with explicit constants for both losses, and works out a
model class of parametric fluxes together with semilinear examples whose
nonlinearity sits in the lower-order term.
Section~\ref{sec:kirchhoff-ndpg} treats the stationary Kirchhoff model, where
the diffusion coefficient depends on the solution through a single energy
integral over the whole domain; that nonlocality makes the nonlinearity one
space-independent scalar, whose stronger differentiability makes the loss
globally variationally correct.
Section~\ref{sec:nn} turns the NDPG losses into training objectives for neural
network surrogates of the parameter-to-solution map and reports the numerical
studies: for the quasilinear and the Kirchhoff model alike, the trained
networks are compared with the finite element solution in the space they
predict into and with a high-fidelity solution on a finer mesh. The ratio
of loss to squared error is used to test variational correctness directly.

\section{Nonlinear parametric PDEs and residual loss functions}
\label{sec:NDPG}
\subsection{Problem setting and computable residuals}
\label{ssec:computable-residual}

Let $\Ua$ and $\Va$ be Hilbert spaces over $\bb R$, and let
$\Aa : \Ua \to \Vap$ be a nonlinear map into the dual space $\Vap$
of~$\Va$, for each value of a parameter $\alpha$ varying in a compact
set $\cQ$.  We consider the parametric family of nonlinear operator
equations
\begin{equation}
  \label{eq:Au=ell}
  \Aa(\cl u) = \ell,
\end{equation}
with a fixed $\ell \in \Vap$, and refer to~\eqref{eq:Au=ell} for a
single $\alpha$ as a {\em fiber problem}.  Typically $\cQ$ is high (or
even infinite) dimensional and its elements represent physical properties of the
system, such as material coefficients, while $\Aa$ arises from a
nonlinear partial differential operator on a low-dimensional spatial
domain $\om$; the case of no parametric dependence is included, with
$\cQ$ a singleton.  Our objective is to approximate the {\em
  parameter-to-solution map} $\alpha \mapsto \clu(\alpha) \in \Ua$,
whose range describes the collection of viable physical states one
wishes to explore.  This objective is to be distinguished from a
single evaluation of that map, i.e., from solving one fiber problem. Nevertheless,
the present section and the next few are concerned with the fiber
problems.

One route to solving a fiber problem is Newton's method (used in some of the
experiments below).  An alternative, closer in spirit to learning the
whole map, is to minimize the squared residual of~\eqref{eq:Au=ell},
\begin{equation}
  \label{eq:squareR}
  \Ra \clw := \|\ell - \Aa(\clw)\|_{\Vap}^2,
\end{equation}
which we call the {\em fiber loss}. On an open set
$D\subseteq\Ua$, select for each $\alpha$ an exact solution
$\clu(\alpha)$, which therefore satisfies
\begin{equation}
  \label{eq:exact-minimization}
  \clu(\alpha)
  \;\in \;
  \argmin_{\clz \in D} \| \ell - \Aa(\clz) \|^2_{\Vap}.
\end{equation}
Our assumptions below imply that the selected solution is
locally unique, but
since other minimizers in $D$ are not excluded
we have used ``$\in$'' rather than  equality above.
Measuring the residual in the dual norm is precisely what one expects
to deliver the variational correctness~\eqref{ideally}, since it
corresponds to a uniformly stable variational formulation
of~\eqref{eq:Au=ell} over the pair $\Ua, \Va$.  But it also raises the
obstacle noted in \S\ref{ssec:vc-definition}: a dual norm is generally
not evaluable, unless $\Va$ is self-dual, i.e., of $L_2$ type.  Two
devices, pursued in parallel throughout this paper, remove this
obstacle.

The first arranges the variational formulation so that $\Va$ {\em is}
an $L_2$-type space, identified with its own dual through the Riesz
map, making $\|\cdot\|_{\Vap}^2$ a plain integral.  This is possible
only if all derivatives are moved onto the trial side, which is what
first-order system reformulations accomplish, at the cost of a
stronger trial norm and extra trial variables.  It leads to the NFOSLS
loss function of~\S\ref{ssec:nfosls-abstract}.  The second keeps the
weaker trial norm and replaces $\Va$ by a finite-dimensional subspace
$\Vah \subset \Va$, on which the dual norm is a finite-dimensional
computation.  For this to be cheap one uses the DPG
approach~\cite{DPGacta}, in which $\Va$ is a ``broken'' space, a
Cartesian product of element-local spaces, so that the dual norm
decouples into independent element-by-element computations.  It leads
to the NDPG loss function of~\S\ref{ssec:ndpg-abstract}.

Both devices are captured by one construction: fix a closed subspace
$\Vah \subseteq \Va$ and measure the residual in $(\Vah)'$ rather than
in $\Vap$.  The first device is the choice $\Vah = \Va$; the second is
a genuine finite-dimensional choice.  Accordingly, the theory below is
developed for an arbitrary closed subspace~$\Vah$.  Of the trial class
$\Uah \subseteq \Ua$, over which the loss will be minimized, we
require throughout the approximation property
\begin{equation}
  \label{eq:approx'}
  \lim_{h \to 0} \inf_{\; \clz_h \in \Uah} \| \clw - \clz_h \|_{\Ua} = 0
\end{equation}
for all $\clw \in D$.  Here $h$ is a generic discretization parameter;
when $\Uah$ is a family of finite element spaces on a mesh hierarchy,
it is the mesh size, but $\Uah$ may equally well be a family of neural
network hypothesis classes.

\subsection{The residual loss function and its variational correctness}
\label{ssec:loss-and-vc}

We now proceed to distill the essentials of the  approach
of~\cite{carstensen2017}. The idea is to directly
approximate~\eqref{eq:squareR} and use the approximation in place of
$\Ra\cdot$ in the minimization~\eqref{eq:exact-minimization}.  Solving
this approximate minimization results in the ``nonlinear DPG method''
of~\cite{carstensen2017} when $\Vah$ is a broken finite element space.
Throughout this subsection, however, $\Vah$ is merely a closed subspace
of $\Va$ and $\Uah$ is merely a subset of $\Ua$: everything here,
including the main theorem, is proved at that level of generality.

Letting $D_h = \Uah \cap D$, the {\em approximation of the
  parameter-to-solution map generated by $\Vah$ and $\Uah$} is
$\alpha \mapsto \clu_h(\alpha)$ where
\begin{equation}
  \label{eq:J-min-NonlinDPG}
  \clu_h(\alpha):=
  \argmin_{\clz_h \in D_h} \| \ell - \Aa(\clz_h) \|^2_{(\Vah)'}
\end{equation}
assuming a minimizer exists.  Note that the norm here is the dual norm of the subspace $\Vah$, which
is in general different from the $\Vap$-norm
in~\eqref{eq:exact-minimization}.  The tacit presumption here is that
$\Vah$ is rich enough for the two to be comparable on the trial class
(made precise in the form needed for our analysis by Assumption~\ref{asm:fortin-a} below) and that $D_h$
is nonempty. Because of~\eqref{eq:approx'}, there is an $h$
sufficiently small such that $D$ contains an approximation of $\clu$
from $\Uah$, thus making $D_h$ nonempty.

\begin{definition}
  \label{defn:ndpg}
  We denote the objective function in~\eqref{eq:J-min-NonlinDPG} by
  $\Ja\cdot$, i.e.,
  \begin{equation}
    \label{eq:DPG-loss-defn}
    \Ja {\clz_h} = \| \ell - \Aa(\clz_h) \|^2_{(\Vah)'}, \qquad
    \clz_h \in \Uah, \alpha \in \cQ,
  \end{equation}
  and refer to it as the {\bf{nonlinear residual loss function}}
  generated by $\Vah$.  Its two instances, the NFOSLS and the NDPG
  loss functions, are described in \S\ref{ssec:nfosls-abstract}
  and~\S\ref{ssec:ndpg-abstract}, where their computability is also
  discussed.
\end{definition}

To analyze the minimization in~\eqref{eq:exact-minimization} or its
discrete version, we need the G\^ateaux derivative of $\Aa$ at a point
$\clw \in D$ along $\clz \in \Ua$, which is a functional in $\Vap$
whose action on a $\clv \in \Va$ is given
by the limit
\begin{equation}
  \label{eq:gateaux}
  \begin{aligned}
    & d\Aa(\clw)(\clz)(\clv)
    = \lim_{t \to 0}
    \frac{1}{t}\big[ \Aa(\clw + t \clz)(\clv) - \Aa(\clw)(\clv) \big],
    \text{ when the limit}
    \\
    &\text{exists and gives a bounded linear operator }d\Aa(\clw) : \Ua \to \Vap.
  \end{aligned}
\end{equation}
Emphasizing that we take boundedness
and linearity of $d\Aa(\clw)$ as  part of the G\^ateaux differentiability requirement,  we note that what
separates~\eqref{eq:gateaux} from \Frechet\ differentiability is that
the limit is asked for along each fixed direction $\clz$ only, not
uniformly over $\|\clz\|_{\Ua} = 1$. The stronger
\Frechet\ differentiability is deliberately avoided: see
Remark~\ref{rem:no-frechet}.  It will be
convenient to use the forms generated by $\Aa$ and
\begin{equation}
  \label{eq:aadaa}
  a_\alpha(\clw; \clv) = \Aa(\clw)(\clv),
  \qquad
  da_\alpha(\clw; \clz, \clv) = d\Aa (\clw)(\clz)(\clv),
\end{equation}
for $\clw \in D, \clz \in \Ua,$ and $\clv \in \Va$.  Clearly these
forms are not generally linear in their first argument, but are linear
in their remaining arguments.

Before introducing the theory,
a few more details on the computation of the loss function
in~\eqref{eq:DPG-loss-defn} are in order.
Given an arbitrary $\clw_h \in \Uah$,
the Riesz representation of its residual, denoted by $\Ea a\clwh$, is the
unique function in $\Vah$ given by
\begin{equation}
\label{eq:residual-rep'}
(\Ea a \clwh, \clv_h)_\Va
= \ell(\clv_h) - a_\alpha(\clwh; \clv_h), \qquad \clv_h \in \Vah.
\end{equation}
Both instances are designed so that this Riesz representation is
inexpensive: for the NFOSLS loss the Riesz map of $\Va$ is the
identity, while for the NDPG loss $\Va$ is broken, so that
$\Ea a\clwh$ is computed element by element by inverting small Gram
matrices.  Furthermore, from
definition~\eqref{eq:DPG-loss-defn}, we see that $\Ja\clzh$ is the
square of the supremum of the right-hand side
of~\eqref{eq:residual-rep'}, divided by $\| \clvh \|_{\Va}$, over all
$0 \ne \clvh$ in~$\Vah$. Hence
\begin{equation}
  \label{eq:L-and-E}
  \Ja {\clz_h} = \| \Ea a \clzh \|_{\Va}^2.
\end{equation}
In practice, we compute $\Ja{ \clz_h}$ by first computing
$\Ea a \clzh$ using~\eqref{eq:residual-rep'} and then
using~\eqref{eq:L-and-E}.
Of course, in~\eqref{eq:residual-rep'}  we are only testing by elements from $\Vah$, so
$\Ja {\clz_h}$ is not generally equal to $\Ra \clz_h$ (the two do coincide
when $\Vah = \Va$). The subsequent theory   addresses this gap using Assumption~\ref{asm:fortin-a}.
Note that equations~\eqref{eq:residual-rep'}--\eqref{eq:L-and-E}
also permit easy differentiation of the loss function,
which is essential for the machine learning application.

The theory below is based on the following two
assumptions.  Let the $\delta$-ball around an $\clw \in \Ua$ be
denoted by $S_\delta(\clw)=\{\clz\in\Ua:\|\clz-\clw\|_\Ua<\delta\}$. When $\delta=+\infty$, the set $S_\delta(\clw)$ is the same as $\Ua$.

\begin{assumption}
  \label{asm:NDPGdiff}
  We assume that
  \begin{enumerate}[label=(\alph*)]

  \item for every $\alpha \in \cQ$, there is a
    $\clu = \clu(\alpha) \in D$ satisfying $\Ra \clu = 0$,

  \item  $\Aa$ is G\^ateaux differentiable in the sense
    of~\eqref{eq:gateaux} at each $\clw \in D$ and
    $\alpha \in \cl Q$; the function
    $t \mapsto d\Aa(\clw + t \clz)(\clz)(\clv)$ is continuous
    on $[0,1]$ whenever the segment $\{\clw + t \clz : 0 \le t \le 1\}$
    lies in $D$; and there exists $0<\CdB<\infty$ such that
\begin{equation}
\label{linvc}
\| d\Aa(\clw)(\clz) \|_{\Vap} \le \CdB \| \clz\|_{\Ua},\quad
   \text{ for all } \clw \in D, \clz \in \Ua, \mbox{and}\,\,  \alpha \in \cQ,
\end{equation}

  \item and there exist constants $0<\delta \le +\infty$ and $c_*>0$ such that for
    every $\alpha\in\cQ$ and every
    $\clz\in S_\delta(\clu(\alpha))\cap D$,
    \begin{equation}
      \label{eq:c-prime}
      c_*\|\clu(\alpha)-\clz\|_\Ua \le \|\Aa(\clu(\alpha))-\Aa(\clz)\|_\Vap.
    \end{equation}
\end{enumerate}
\end{assumption}

Note that in Assumption~\ref{asm:NDPGdiff}(a), since $\Ra \clw \ge 0$,
any such $\clu$ is a minimizer in~\eqref{eq:exact-minimization}. No
uniqueness of $\clu(\alpha)$ is stated in (a), but it is in fact a
consequence of part~(c): by~\eqref{eq:c-prime}, $\clu(\alpha)$ is the
only element of $S_\delta(\clu(\alpha)) \cap D$ at which $\Ra\cdot$
vanishes, and hence also the unique minimizer of $\Ra\cdot$ there.
Thus the selected solution $\clu(\alpha)$ is locally unique.
In particular, when $\delta = +\infty$ and $D = \Ua$, part~(c) makes the
minimizer in~\eqref{eq:exact-minimization} globally unique.   Note
also that Assumption~\ref{asm:NDPGdiff} places no lower bound on the
derivative $d\Aa$ because none is needed in the proof of
Theorem~\ref{thm:NDPG-gen-nonlinear-dpg} below (although in examples
we have in mind, $d\Aa$ is a bijection).  We have also not assumed the
continuity of $d \Aa$ since this continuity is not available in many interesting examples: see Remark~\ref{rem:sufficient-c} for how
such continuity, when available, can be leveraged.
We refer to Assumption~\ref{asm:NDPGdiff}(c) as a ``local nonlinear inf-sup'' type condition.

\begin{assumption}
  \label{asm:fortin-a}
  Let $\Vah$ be a closed subspace of $\Va$.  We assume that there is a
  continuous linear operator $\fp: \Va \to \Vah$, referred to as a
  Fortin operator, satisfying
  \begin{equation}
    \label{eq:Fortin}
    a_\alpha(\clw_h; \clv - \fp \clv ) = 0,
    \quad
    \text{ for all }
    \clwh \in D_h, \clv \in \Va.
  \end{equation}
\end{assumption}

Note that the choice $\Vah = \Va$ with $\fp = \mathrm{id}$ always
satisfies Assumption~\ref{asm:fortin-a}, with $\|\fp\|_{L(\Va)} = 1$
and with the oscillation term $\osc(\ell)$ defined below vanishing
identically.  This trivial choice yields the least-squares instance
of~\S\ref{ssec:nfosls-abstract}, so the substance of
Assumption~\ref{asm:fortin-a} lies entirely in the DPG instance
of~\S\ref{ssec:ndpg-abstract}.

Under these assumptions, we prove a result on {\it a posteriori error
  estimation}, giving  reliability and efficiency estimates with
$\Ja\cdot$ as the error estimator. The discrete spaces occur only in
Assumption~\ref{asm:fortin-a}. In Assumption~\ref{asm:NDPGdiff}, in
addition to the existence of an exact solution, G\^ateaux differentiability
of $\Aa$, and a uniform upper bound on $d\Aa$,
part~(c) assumes a ``local nonlinear inf-sup condition.''  Let
\[
  \osc(\ell) =
  \sup_{0 \ne \clv \in \Va}
  \frac{ |\ell (\clv  - \fp  \clv )|}{ \| \clv \|_{\Va}}
  \quad\text { and } \quad
  \| \fp  \|_{L(\Va)} = \sup_{0 \ne \clv \in \Va}
  \frac{ \|\fp  \clv \|_{\Va}}{\| \clv \|_{\Va}}.
\]
The first quantity above represents a term akin to the usual data
oscillation terms and the second is the operator norm of the Fortin
operator $\fp $; they equal $0$ and $1$, respectively, in the
least-squares instance.  For linear DPG formulations,
$\|\fp\|_{L(\Va)}$ has been
extensively studied: see, e.g.,~\cite{DPGacta,GopalQiu14}. The same
techniques can also aid in verifying Assumption~\ref{asm:fortin-a} in
some nonlinear cases, as we will see in examples.

\begin{theorem}[Local two-sided bounds for nonlinear residual loss with data oscillation]
  \label{thm:NDPG-gen-nonlinear-dpg}
  Suppose Assumptions~\ref{asm:NDPGdiff} and~\ref{asm:fortin-a} hold.
  Then the error estimator (or the loss function) $\Ja\cdot$
  of~\eqref{eq:DPG-loss-defn}
  satisfies the following local reliability and efficiency estimates:
  there are positive constants $c$ and
  $C$ such that given any $\alpha \in \cQ$, there exist an $h_0>0$ and
  a set $\tilde D_{h_0} \subseteq D$, such that for all $h < h_0$,
  $\tilde D_{h_0}$ contains a nonempty open subset of $\Uah$ and every
  $\clwh \in \tilde D_{h_0} \cap \Uah$ satisfies
  \begin{align}
    \label{eq:reliability}
    c & \| \clu  - \clw_h \|_\Ua^2  \le
    \Ja{\clw_h}\,
          \| \fp \|_{L(\Va)}^2
          + \osc(\ell)^2,
    \\
    \label{eq:efficiency}
    \Ja{\clw_h} \le C & \| \clu - \clw_h \|_\Ua^2,
  \end{align}
  and we may choose $c= c_*^2/2$ and $C=\CdB^2.$ If \eqref{eq:c-prime}
  holds globally, i.e., if $D = \Ua$ and $\delta=+\infty$ in
  Assumption~\ref{asm:NDPGdiff}(c), then we may take $h_0=+\infty$ and
  $\tilde D_{h_0} = \Ua$.  If $\osc(\ell)=0$ and
  $\| \fp \|_{L(\Va)}^2$ is bounded uniformly in $\alpha$ and $h$,
  then~\eqref{eq:reliability}--\eqref{eq:efficiency} imply local
  variational correctness in the sense of~\eqref{ideally} on
  $\cl H = \tilde D_{h_0} \cap \Uah$, and if, in addition, $D = \Ua$
  and $\delta=+\infty$, then the variational correctness holds with
  $\cl H=\Uah$.
\end{theorem}
\begin{proof} {\em Step 1.}
  Abbreviate $\clu \equiv \clu(\alpha)$.
  Let $\delta$ and $c_*$ be as in Assumption~\ref{asm:NDPGdiff}(c),
  choose $\delta_1 > 0$ so small that $S_{\delta_1}(\clu) \subseteq D$,
  and set $\delta_2 = \min(\delta, \delta_1)$.
  We claim that for every $\clz \in S_{\delta_2}(\clu)$,
  \begin{equation}
    \label{eq:33-prime}
    c_* \|\clu - \clz\|_{\Ua}
    \;\le\; \|\Aa(\clu) - \Aa(\clz)\|_{\Vap}
    \;\le\; \CdB \|\clu - \clz\|_{\Ua}.
  \end{equation}
  The lower bound is immediate from Assumption~\ref{asm:NDPGdiff}(c),
  since $S_{\delta_2}(\clu) \subseteq S_\delta(\clu) \cap D$.
  For the upper bound, since $S_{\delta_2}(\clu)$ is convex, the segment
  $\clz_t = \clz + t(\clu - \clz)$ lies in $S_{\delta_2}(\clu) \subseteq D$
  for all $t \in [0,1]$.
  On this line segment, define
  $\psi(t) = a_\alpha(\clz_t; \clv)$ for a
  $\clv \in \Va$.
  The  Fundamental Theorem of Calculus (FTC) applied
  to $\psi$ yields
  \begin{equation}
    \label{eq:ftc}
    \langle \Aa(\clu) - \Aa(\clz),\, \clv \rangle
    = \psi(1) - \psi(0) = \int_0^1 \psi'(t) \, dt
    = \int_0^1 d\Aa(\clz_t)(\clu - \clz)(\clv)\, dt.
  \end{equation}
  By Assumption~\ref{asm:NDPGdiff}(b), the integrand is bounded by
  $\CdB \|\clu - \clz\|_{\Ua} \|\clv\|_{\Va}$, so taking the supremum
  over $\{ \clv \in \Va: \|\clv\|_{\Va} = 1\}$ yields the upper bound
  in~\eqref{eq:33-prime}.
  (If $\delta=+\infty$ and $D = \Ua$, then we may take
  $\delta_2 = +\infty$.)

  {\em Step 2.} We now prove the reliability
  estimate~\eqref{eq:reliability}.
  By~\eqref{eq:approx'}, we can find an $h_0>0$ such that for all
  $h < h_0$, there is some $\clph$ in
  $\Uah \cap S_{\delta_2}(\clu) \subset D$, and consequently, there is
  also an open neighborhood $U_h(\clph)$ of $\clph$ in $D_h$, which we
  may take to be contained in $S_{\delta_2}(\clu)$. Then
  $\tilde D_{h_0} = \cup_{h < h_0} U_h(\clph)$ is contained in
  $S_{\delta_2}(\clu) \cap D$ and contains  for each $h < h_0$
  the nonempty open subset   $U_h(\clph)$ of $\Uah$.
  (If $\delta=+\infty$ and $D = \Ua$, then $S_{\delta_2}(\clu) = \Ua$,
  so we may take $h_0=+\infty$ and $\tilde D_{h_0} = \Ua$, giving
  $\tilde D_{h_0} \cap \Uah = \Uah$ for all $h$.)

  Since $\Ra \clu = 0$ by Assumption~\ref{asm:NDPGdiff}(a), we have
  $\Aa(\clu) = \ell$, so applying the lower bound in~\eqref{eq:33-prime}
  with $\clz = \clwh \in \tilde D_{h_0} \cap \Uah \subseteq S_{\delta_2}(\clu)$ yields
  \begin{equation}
    \label{eq:34}
    c_* \| \clu - \clw_h \|_{\Ua}
    \le
    \sup_{0 \ne\clv \in \Va}
    \frac{|\ell(\clv) - a_\alpha(\clw_h; \clv)|}{ \| \clv \|_{\Va}}.
  \end{equation}
  To bound the numerator, observe that
  \begin{align*}
    \ell(\clv) - a_\alpha(\clw_h; \clv)
    & = \ell(\fp \clv) - a_\alpha(\clw_h; \fp  \clv)
      + \ell(\clv - \fp \clv) - a_\alpha( \clw_h; \clv - \fp  \clv)
    \\
    & = (\Ea{a}{\clwh}, \fp  \clv)_{\Va}
      + \ell(\clv - \fp \clv)
    \\
    & \le \Big(
      \Ja{\clwh}^{1/2}  \| \fp  \|_{L(\Va)}
      + \osc (\ell ) \Big)
      \| \clv \|_{\Va}
  \end{align*}
  where we have used
  Assumption~\ref{asm:fortin-a}'s equation~\eqref{eq:Fortin},  \eqref{eq:residual-rep'},
  and~\eqref{eq:L-and-E} in succession. Using this to bound the right-hand
  side of~\eqref{eq:34} after squaring,
  \begin{align*}
    c_*^2 \| \clu - \clw_h \|_{\Ua}^2
    \le
    2 \| \fp  \|^2_{L(\Va)}
      \Ja{\clw_h}
      + 2 \osc(\ell)^2,
  \end{align*}
  proves  the reliability inequality~\eqref{eq:reliability}.

  {\em Step 3.}  To prove the efficiency
  estimate~\eqref{eq:efficiency}, we start
  from~\eqref{eq:residual-rep'} to get
  \begin{align*}
    \| \Ea{a}{\clwh} \|_{\Va}^2 & = \ell ( \Ea{a}{\clw_h}) - a_\alpha( \clw_h; \Ea{a}{\clwh})
    \\
    & = a_\alpha( \clu; \Ea{a}{\clwh}) - a_\alpha(\clwh; \Ea{a}{\clwh})
    \\
    & \le \CdB \| \clu - \clwh\|_{\Ua} \| \Ea{a}{\clwh} \|_{\Va}
  \end{align*}
  where the second equality holds because $\ell = \Aa(\clu)$ by
  Assumption~\ref{asm:NDPGdiff}(a) and, in the last step, we have
  used the upper bound in~\eqref{eq:33-prime}, which applies since
  $\tilde D_{h_0} \subset S_{\delta_2}(\clu).$
  In view of \eqref{eq:L-and-E}, the
  resulting bound for $\| \Ea{a}{\clwh} \|_{\Va}$ gives the claimed
  estimate~\eqref{eq:efficiency}.
\end{proof}

\begin{remark}[Best approximation from the trial {\em set}]
  \label{rem:quasioptimality}
  Theorem~\ref{thm:NDPG-gen-nonlinear-dpg} is an {\it a posteriori}
  statement, but it constrains the loss minimizer {\it a priori} as
  well.  Such a passage from two-sided {\it a posteriori} bounds to
  {\it a priori} error estimates is by now routine: see the linear
  precedent in~\cite[Remark~2.5]{CarstDemkoGopal14}. Repeating the
  idea, let $\clu_h$ realize
  $\inf\{ \Ja{\clw} : \clw \in \tilde D_{h_0} \cap \Uah\}$.  Reading
  the reliability bound~\eqref{eq:reliability} at $\clu_h$, replacing
  $\Ja{\clu_h}$ by $\Ja{\clw}$ for any
  $\clw \in \tilde D_{h_0} \cap \Uah$ (which minimality permits), and
  then discarding $\Ja\clw$ in favor of the error through the
  efficiency bound~\eqref{eq:efficiency}, one arrives at
  \begin{equation}
    \label{eq:quasioptimality}
    \| \clu - \clu_h \|^2_{\Ua}
    \;\le\;
    \varLambda \!\!\!
    \inf_{\clw \in \tilde D_{h_0} \cap \Uah} \!\!\!
    \| \clu - \clw \|^2_{\Ua}
    \; + \; \frac 1 c\, \osc(\ell)^2,
    \qquad
    \varLambda := \frac C c \, \| \fp \|^2_{L(\Va)} .
  \end{equation}
  An exact minimizer of the loss therefore achieves, up to the factor
  $\varLambda$ and the oscillation term, the best approximation error
  attainable in the admissible trial set.  Two aspects
  of~\eqref{eq:quasioptimality} deserve emphasis.  The first is that
  no linear structure on $\Uah$ was used either
  in~\eqref{eq:quasioptimality} or anywhere in the proof of
  Theorem~\ref{thm:NDPG-gen-nonlinear-dpg}.  The infimum
  in~\eqref{eq:quasioptimality} therefore runs over precisely the set
  on which the loss was minimized, whatever that set may be.  For the
  neural network classes of \S\ref{sec:nn}, where the loss is
  minimized in expectation over $\alpha$ rather than for each $\alpha$
  separately, the same argument applied to the integrated
  estimate~\eqref{vc} gives a genuine best-approximation guarantee,
  in the norm~\eqref{eq:Xnorm}, for the minimizer of the expected loss
  over a nonlinear set of realizations.  The amplification $\varLambda$, in turn, is the
  reliability-to-efficiency ratio delivered by variational
  correctness, inflated by the Fortin operator norm. In the
  least-squares instance of \S\ref{ssec:nfosls-abstract}, provided
  $D=\Ua$ and $\delta = +\infty$, this inflation disappears along with
  the oscillation, leaving
  $\| \clu - \clu_h \|^2_{\Ua} \le (C/c) \inf_{\clw \in \Uah} \| \clu
  - \clw\|^2_{\Ua}$
  for any chosen trial class $\Uah\subseteq\Ua$.
\end{remark}

\begin{remark}[A {\em sufficient condition} for the local nonlinear inf-sup condition]
  \label{rem:sufficient-c}
  When the neighborhoods $S_r(\clu(\alpha))$ are in $D$ for some $r>0$
  independent of $\alpha$, the condition~(c) of
  Assumption~\ref{asm:NDPGdiff} is implied by parts~(a) and~(b)
  together with
  \begin{itemize}
  \item[{\rm{(b$'$)}}]
    a quantitative uniform lower bound on the derivative, i.e.,
    there is a $\cdB > 0$ such that
    \begin{equation}
      \label{eq:lower-dA}
      \cdB \|\clz\|_{\Ua}
      \le \|d\Aa(\clw)(\clz)\|_{\Vap},
      \qquad \clw \in D,\; \clz \in \Ua,\; \alpha \in \cQ,
    \end{equation}
  \item[{\rm{(c$'$)}}] and the continuity of the derivative
    $d\Aa : D \to L(\Ua, \Vap)$, uniformly in $\alpha$, in the sense that
    the quantity
    \begin{equation}
      \label{eq:uniform-cty}
      c(\delta) := \sup_{\alpha\in\cQ}
      \sup_{\clz \in S_\delta(\clu(\alpha))}
      \|d\Aa(\clz) - d\Aa(\clu(\alpha))\|_{L(\Ua,\Vap)}
      \to 0 \quad \text{ as } \delta \to 0,
    \end{equation}
  \end{itemize}
  where $L(\Ua, \Vap)$ denotes the space of bounded linear operators
  from $\Ua$ to $\Vap$ endowed with the operator norm topology. To show this,
  by~\eqref{eq:uniform-cty},
  choose $0< \delta \le r $, independent of $\alpha$, with
  $c(\delta) < \cdB/2$. Fix $\alpha$, write $\clu=\clu(\alpha)$,
  and let $\clz\in S_\delta(\clu)$. Then the segment
  $\clz_t = \clu + t(\clz - \clu)$, $t\in [0, 1]$, lies
  in $S_\delta(\clu) \subset D$.
  Proceeding as in the proof of
  Theorem~\ref{thm:NDPG-gen-nonlinear-dpg} to get \eqref{eq:ftc} and
  then subtracting $d\Aa(\clu)(\clu-\clz)(\clv)$ from both
  sides yields
  \[
     \ip{ \Aa(\clu) - \Aa(\clz)      - d\Aa(\clu)(\clu-\clz), \clv}
    = \int_0^1 \Big[ d\Aa(\clz_t)(\clu-\clz)(\clv)
    - d\Aa(\clu)(\clu-\clz)(\clv)\Big]\,dt.
  \]
  Since every $\clz_t \in S_\delta(\clu)$, the integrand is bounded by
  $c(\delta)\|\clu-\clz\|_\Ua\|\clv\|_\Va$, so
  \[
    \bigl\|d\Aa(\clu)(\clu-\clz)
    - \bigl[\Aa(\clu)-\Aa(\clz)\bigr]\bigr\|_{\Vap}
    \le c(\delta)\|\clu-\clz\|_{\Ua}.
  \]
  Combining with~\eqref{eq:lower-dA} and the triangle inequality,
  \[
    \|\Aa(\clu)-\Aa(\clz)\|_{\Vap}
    \ge \cdB\|\clu-\clz\|_{\Ua} - c(\delta)\|\clu-\clz\|_{\Ua}
    \ge \frac{\cdB}{2}\|\clu-\clz\|_{\Ua},
  \]
  which is~\eqref{eq:c-prime} with $c_* = \cdB/2$.
  We hasten to add that {\rm(b$'$)} and {\rm(c$'$)} are sufficient,
  but far from necessary, for Assumption~\ref{asm:NDPGdiff}(c).  The
  proof of Theorem~\ref{thm:NDPG-gen-nonlinear-dpg} uses no lower
  bound on $d\Aa$ at all, so the uniformity of~\eqref{eq:lower-dA} in
  both $\clw$ and $\alpha$ demands more than the theory requires;
  indeed, even in the above argument, \eqref{eq:lower-dA} was used
  only at $\clw = \clu(\alpha)$.  Moreover, the operator-norm
  continuity demanded by {\rm(c$'$)} may be unavailable for
  Nemytskii-type nonlinearities such as some of the later examples
  (cf.~Remark~\ref{rem:no-frechet}).
  This is why we assumed the weaker~\eqref{eq:c-prime}, which
  will be verified directly for such examples in
  Proposition~\ref{prop:NDPGdiff-verif'}, without recourse to~{\rm(c$'$)}.
\end{remark}

\begin{remark}[Uniformity of the constants in $\alpha$]
  \label{rem:alpha-uniform}
  Since $c_*$ and $\CdB$ in Assumption~\ref{asm:NDPGdiff} are independent of the
  parameter~$\alpha$, the constants in~\eqref{eq:reliability} and
  \eqref{eq:efficiency} can be chosen independently of the
  parameter~$\alpha$. However, the set $\tilde D_{h_0}$
  in Theorem~\ref{thm:NDPG-gen-nonlinear-dpg}
  may  depend on $\alpha$. Also, if one chooses to consider parameters taken  from
  an unbounded random distribution (such as  the normal
  distribution), then it may be unreasonable to ask for the constants
  $c_*$ and $\CdB$ in Assumption~\ref{asm:NDPGdiff}(b)--(c)
  to hold uniformly for all $\alpha$.
\end{remark}

\begin{remark}[The Fortin condition for the {\em linearized} form]
  \label{rem:fortin-linearized}
  Suppose  $\Uah$ is not just a set, but a  linear subspace of
  $\Ua$, so that a perturbation of one trial function by another
  remains in the trial class. Then, we claim that
  \begin{equation}
    \label{eq:1}
    \text{Assumption~\ref{asm:fortin-a}}\implies
    da_\alpha(\clw_h; \clz_h, \clv - \fp \clv ) = 0, \quad
    \clw_h \in D_h,\; \clz_h \in \Uah, \;\clv \in \Va.
  \end{equation}
  Indeed, since $\clw_h + t \clz_h$ lies in $\Uah$ and $D$ is open,
  the point $\clw_h + t \clz_h$ is in $D_h$ for
  $t \in (-t_0, t_0)$ for some sufficiently
  small~$t_0>0$. Then, by~\eqref{eq:Fortin}, the function
  $\psi(t) = a_\alpha(\clw_h + t \clz_h; \clv - \fp \clv)$ is zero for
  $t \in (-t_0, t_0)$. Hence its derivative $\psi'(0)$
  vanishes. Since
  $\psi'(t) = da_\alpha(\clw_h + t \clz_h; \clz_h, \clv - \fp \clv)$,
  the claim~\eqref{eq:1} follows.  Several results in
  \cite{carstensen2017} are stated under the assumption
  that
  $da_\alpha(\clw_h; \clz_h, \clv - \fp \clv ) = 0$.  In view of~\eqref{eq:1}, they also hold under
  Assumption~\ref{asm:fortin-a}.
\end{remark}

\begin{remark}[A nonlinear {\em mixed} reformulation]
  \label{rem:mixed-reformulation}
  Suppose again that  $\Uah$ is a  linear subspace of
  $\Ua$, so that the perturbations considered  below stay within the trial
class.  Let us observe what the first-order
optimality conditions of the optimization
in~\eqref{eq:exact-minimization} and~\eqref{eq:J-min-NonlinDPG} yield.
A $\clw_h$ in $D_h$ is  a {\em stationary point} of  $\Ja\cdot$
if
\[
  \lim_{\delta\to 0 }
  \frac{\Ja{\clw_h + \clz_h\, \delta} - \Ja{\clw_h}}{\delta}
  = 0
\]
for all directions $\clz_h \in \Uah.$ It is easy to
see~\cite{carstensen2017} that a $\clu_h \in D_h$ is a stationary
point of $\Ja\cdot$ if and only if $\cl u_h$ together with
$\veps_h
\in \Vah$
solves the nonlinear mixed system
\begin{subequations}
  \label{eq:nonlinear-mixed}
  \begin{align}
    \label{eq:eps-u}
    (\veps_h, \clv_h)_\Va + a_\alpha(\clu_h; \clv_h)
    &  = \ell( \clv_h),
    && \clv_h \in \Vah,
    \\
    da_\alpha  (\clu_h; \clz_h, \veps_h)
    & = 0,
    && \clz_h \in \Uah.
  \end{align}
\end{subequations}
This {\em nonlinear mixed DPG reformulation} generalizes the
well-known~\cite{DHSW,DPGacta} reformulation of the DPG method as a mixed
method for linear problems.
The prior equation~\eqref{eq:residual-rep'} is
exactly~\eqref{eq:eps-u}, i.e., $\veps_h = \Ea a{\clu_h} \in
\Vah$. The undiscretized version of~\eqref{eq:nonlinear-mixed} can be
written in operator form yielding the stationarity
conditions for
\eqref{eq:exact-minimization}: a stationary point $\clu$ in $D$,
together with its residual representative $\veps$ in $\Va$, satisfies
\begin{equation}\label{eq:mixedcarst}
  \begin{pmatrix}
    \Rza & \Aa \\
    d\Aa'(\cl u) & 0
  \end{pmatrix}
  \begin{pmatrix}
    \veps \\ \cl u
  \end{pmatrix}
  =
  \begin{pmatrix}
    \ell \\ 0
  \end{pmatrix}
  \qquad
  \text{ in } \; \Vap \times \Uap.
\end{equation}
Here, $\Rza : \Va \to \Vap$ is the Riesz isometry and $d\Aa'$ denotes
the dual operator of $d\Aa$.  At an exact solution $\clu$, the
residual representative $\veps$ vanishes. Conversely, a solution of
\eqref{eq:mixedcarst} with $\veps=0$ satisfies $\Aa(\clu)=\ell$.
The discrete residual representation $\veps_h$ does not vanish in
general.
\end{remark}

We conclude this section by describing, in the next two subsections,
the two instances of Definition~\ref{defn:ndpg} used in the remainder
of the paper.

\subsection{The least-squares instance}
\label{ssec:nfosls-abstract}

Consider the choice $\Vah = \Va$ and $\fp = \mathrm{id}$, which, as
already noted, satisfies Assumption~\ref{asm:fortin-a} trivially with
$\|\fp\|_{L(\Va)} = 1$, and which makes $\osc(\ell) = 0$ for every
$\ell \in \Vap$.  The loss is then the exact continuous residual,
\begin{equation}
  \label{eq:ls-loss}
  \Ja\clz = \| \ell - \Aa(\clz)\|_{\Vap}^2 = \Ra \clz ,
\end{equation}
and Theorem~\ref{thm:NDPG-gen-nonlinear-dpg} yields
\begin{equation}
  \label{eq:ls-vc}
  c\, \| \clu - \clz \|_{\Ua}^2
  \;\le\; \Ja \clz
  \;\le\; C\, \| \clu - \clz \|_{\Ua}^2,
  \qquad \clz \in \tilde D_{h_0} \cap \Uah,
\end{equation}
with $c = c_*^2/2$ and $C = \CdB^2$.  When
Assumption~\ref{asm:NDPGdiff}(c) holds globally, i.e., $D = \Ua$ and
$\delta = +\infty$, the set $\tilde D_{h_0}$ may be taken to be all of
$\Ua$, so that~\eqref{eq:ls-vc} holds for every $\clz \in \Ua$.  In
that case no mesh, and indeed no trial class at all, enters
the statement: the trial class $\Uah$ in~\eqref{eq:ls-vc} is
arbitrary, so $\clz$ may be taken from a meshless neural network
hypothesis class within~$\Ua$.

For this instance to be useful, the norm $\|\cdot\|_{\Vap}$
in~\eqref{eq:ls-loss} must be directly evaluable.  This is the case
when $\Va$ is an $L_2$-type space identified with its own dual.
Since the examples we have in mind
achieve this by first-order system reformulations, in which all
derivatives are on   trial functions, we think of this
instance as a nonlinear first-order system least-squares (FOSLS)
functional.  Accordingly, we call it the {\bf{NFOSLS loss function}}.
Concrete realizations are given in~\S\ref{ssec:nfosls-diffusion}
and~\S\ref{ssec:kirchhoff-nfosls}.

\subsection{The DPG instance}
\label{ssec:ndpg-abstract}

In the DPG methodology~\cite{DPGacta,DemkoGopal11}, the trial space is
enlarged by interface unknowns so that the test space $\Va$ can be
taken to be ``broken'', and the Riesz map of $\Va$ becomes
element-local.  Namely, we pick a conforming simplicial mesh $\oh$
from a hierarchy of uniformly shape-regular meshes (of mesh size $h$)
and set
\begin{equation}
  \label{eq:V-broken'}
  \Va = \bigtimes_{K \in \oh} \VK,
  \qquad
  \Vah = \bigtimes_{K \in \oh} \VhK,
\end{equation}
where $\VhK \subseteq \VK$ are finite dimensional.  Functions in
$\Vah$ have no interelement continuity across element interfaces, a
property often stated as ``$\Vah$ is
broken''~\cite{CarstDemkoGopal16}.  The trial space $\Ua$ is then also
mesh-dependent, through the interface (or ``skeleton'') components of
its elements, but all our estimates remain mesh-independent.  For such
$\Vah$, the computation of the squared dual norm
\begin{equation}
  \label{eq:6}
  \| \ell - \Aa(\clz_h) \|_{(\Vah)'}^2
  = \sup_{0 \ne \clv_h \in \Vah}
  \frac{ |\ell(\clv_h) - a_\alpha(\clz_h; \clv_h)|^2 }{ \| \clv_h \|_{\Va}^2}
\end{equation}
decouples into independent element-by-element computations: as
described in~\eqref{eq:residual-rep'}--\eqref{eq:L-and-E}, one solves
small symmetric positive definite systems for $\Ea a{\clz_h}$, one per
element $K \in \oh$, with the Gram matrix of the $\VhK$-inner product.
We call the resulting loss the {\bf{NDPG loss function}}, or the
nonlinear DPG loss function. This is similar to that given in~\cite{carstensen2017} for a primal DPG formulation.
In this instance, $\Vah$ should be chosen large enough
that $ \|\ell- \Aa(\clw_h)\|^2_{(\Vah)'}$
approximates $  \Ra{\clw_h}$
for all $\clw_h\in \Uah,$
which is exactly what Assumption~\ref{asm:fortin-a} delivers, through
Theorem~\ref{thm:NDPG-gen-nonlinear-dpg}, in the form of the two
factors $\|\fp\|_{L(\Va)}$ and $\osc(\ell)$.  Concrete realizations
are given in~\S\ref{ssec:ndpg-diffusion}
and~\S\ref{ssec:kirchhoff-setting}.

\section{Quasilinear diffusion with nonlinear fluxes and lower-order terms}
\label{sec:nonlinear-diffusion}
This section develops the first of the two applications of the theory
of~\S\ref{sec:NDPG}.
This application pertains to
quasilinear second-order elliptic PDEs where the   flux
is allowed to depend nonlinearly on the gradient and
on the value of the solution simultaneously, as in
$\sigma_\alpha(p,s) = -\phi_\alpha(|p|,s)\, p$, where the conductivity
responds to the solution as well as to its gradient.
Two example cases, extremes of the same spectrum,  are
worked out in detail: the {\em monotone nonlinear flux} case, in which
the leading-order operator itself is nonlinear and no lower-order
terms are present, and the {\em semilinear} case, in which the
leading-order operator is linear and all the nonlinearity sits in
additive lower-order terms.

On an open bounded set
$\om \subset \bb R^d$, given $f \in L_2$, we consider the
quasilinear boundary value problem
\begin{equation}
  \label{eq:4}
  \begin{aligned}
    \dive\, \sigma_\alpha( \nabla u, u) + \gamma_\alpha(u) & = f  && \text{ in } \om,
    \\
    u & = 0 && \text{ on } \partial \om,
  \end{aligned}
\end{equation}
in which the {\em flux function} $\sigma_\alpha : \bb R^d \times \bb R
\to \bb R^d$ may depend nonlinearly on both the gradient and the value
of the solution, and $\gamma_\alpha : \bb R \to \bb R$ is a nonlinear
reaction term.
Writing the flux of the solution as a new variable
$q = \sigma_\alpha(\nabla u, u)$ and inverting $\sigma_\alpha$ in its
first argument when possible, so that $\nabla u = \rho_\alpha(q, u)$,
problem~\eqref{eq:4} becomes the first-order system
\begin{subequations}
  \label{eq:4-mixed}
  \begin{align}
    \label{eq:4-mixed-a}
    \rho_\alpha(q, u) - \nabla u &= 0, \quad\text{in }\om, \\
    \label{eq:4-mixed-b}
    \operatorname{div} q + \gamma_\alpha(u) &= f, \quad\text{in }\om, \\
    u &= 0, \quad \text{on } \partial \om.
  \end{align}
\end{subequations}
When computing with the primal formulation~\eqref{eq:4} we need
$\sigma_\alpha$ explicitly (as seen in the NFOSLS loss of
\S\ref{ssec:nfosls-diffusion}).  In contrast, when computing with the
dual formulation~\eqref{eq:4-mixed}, we need $\rho_\alpha$ explicitly
(as seen in the NDPG loss of~\S\ref{ssec:ndpg-diffusion}).
The main results of this section are Theorems~\ref{thm:vc-NFOSLS-example} and~\ref{thm:vc-NDPG-example}.

\subsection{A model class of quasilinear problems}
\label{ssec:sigma-class}

We isolate a class of nonlinearities for further study where the
assumptions of \S\ref{sec:NDPG} can be verified (leaving the parameter set $\cl Q$
unspecified until the examples).  Throughout, the
symmetric Jacobian matrix
$\partial_p \sigma_\alpha(p,s) \in \bb R^{d\times d}$ and the vector
$\partial_s\sigma_\alpha(p,s) \in \bb R^d$ denote the partial
derivatives of $\sigma_\alpha$ with respect to its first (vector) and
second (scalar) arguments.  As already stipulated
in~\S\ref{sec:vc-verification}, and used throughout the paper,
$\Ho^1$ carries the norm
$ \| w \|_{\Ho^1} := \| \nabla w \|_{L_2}, $ which is
equivalent to the full $H^1(\om)$ norm due to the Friedrichs inequality
\begin{gather}
  \label{eq:Friedrichs-std}
  c_P \| w \|_{L_2}^2
  \le  \|\nabla w \|_{L_2}^2.
\end{gather}

\begin{hypothesis}
  \label{hyp:sigma}
  There are constants $0 < \mu_1 \le \mu_2 < \infty$,
  $\lambda \ge 0$, and $C_\gamma \ge C_\gamma^- \ge 0$ such that, for
  every $\alpha \in \cQ$, the functions
  $\sigma_\alpha \in C^1(\bb R^d\times\bb R; \bb R^d)$ and
  $\gamma_\alpha \in C^1(\bb R)$ satisfy, for all
  $(p,s) \in \bb R^d\times\bb R$ and all $\xi \in \bb R^d$,
  \begin{subequations}
    \label{eq:sigma-asm}
    \begin{align}
      \label{eq:Q1}
      &\mu_1 |\xi|^2
      \le -\xi \cdot \partial_p\sigma_\alpha(p,s)\, \xi
        \le \mu_2 |\xi|^2,
      \\
      \label{eq:Q2}
      &| \partial_s \sigma_\alpha(p,s) |
      \le \lambda,
      \\
      \label{eq:Q3}
      &|\gamma_\alpha'(s)|
      \le C_\gamma,
        \quad
        \gamma_\alpha'(s)
        \ge -C_\gamma^-,
      \\
      \label{eq:Q4}
      &c_0  := \mu_1
      - \frac{\lambda}{\sqrt{c_P}}
      - \frac{C_\gamma^-}{c_P}
       \;>\; 0 .
    \end{align}
  \end{subequations}
\end{hypothesis}

Condition~\eqref{eq:Q1} is the familiar uniform monotonicity and Lipschitz
requirement on the flux, expressed through its derivative.
Conditions~\eqref{eq:Q2}--\eqref{eq:Q4} restrict
lower-order terms. We think of~\eqref{eq:Q4}
as a {\em smallness condition}
on the  lower-order nonlinearities in relation to
the diffusion coefficient.
Note  that~\eqref{eq:Q4} is {vacuous} whenever
$\lambda = C_\gamma^- = 0$, i.e., whenever $\sigma_\alpha$ does not
depend on its second argument and $\gamma_\alpha$ is nondecreasing.
The condition
only uses the {\em negative} part of $\gamma_\alpha'$: often
reaction terms have nonnegative
$\gamma_\alpha'(s)$, in which case the condition holds with~$C_\gamma^-=0$.
Two consequences of Hypothesis~\ref{hyp:sigma} are used repeatedly.
The first is that the inverse flux $\rho_\alpha$
of~\eqref{eq:4-mixed-a} exists and is as regular as $\sigma_\alpha$.

\begin{lemma}[The inverse flux]
  \label{lem:rho-exists}
  Let Hypothesis~\ref{hyp:sigma} hold.  Then, for each fixed
  $s \in \bb R$, the map $p \mapsto \sigma_\alpha(p,s)$ is a
  $C^1$-isomorphism of $\bb R^d$ onto itself.  Its inverse
  $\rho_\alpha(\cdot, s)$ satisfying
  \begin{equation}
    \label{eq:rho-defn}
    \sigma_\alpha\bigl( \rho_\alpha(q,s), s\bigr) = q,
    \qquad (q,s) \in \bb R^d \times \bb R,
  \end{equation}
  is $C^1$ jointly in $(q,s)$, with
  \begin{equation}
    \label{eq:rho-derivs}
    \partial_q \rho_\alpha(q,s)
    = \bigl[ \partial_p\sigma_\alpha(\rho_\alpha(q,s), s) \bigr]^{-1},
    \qquad
    \partial_s\rho_\alpha(q,s)
    = -\,\partial_q\rho_\alpha(q,s)\,
    \partial_s \sigma_\alpha(\rho_\alpha(q,s), s),
  \end{equation}
  and consequently
  \begin{equation}
    \label{eq:rho-bounds}
    \mu_2^{-1} |\xi|^2
    \le - \xi\cdot \partial_q\rho_\alpha(q,s)\, \xi
    \le \mu_1^{-1} |\xi|^2,
    \qquad
    |\partial_s\rho_\alpha(q,s)| \le \frac{\lambda}{\mu_1}.
  \end{equation}
\end{lemma}
\begin{proof}
  Using the FTC along the segment joining
  two arbitrary points in $\bb R^d$, we find that the map
  $p \mapsto -\sigma_\alpha(p,s)$ is strongly monotone with constant
  $\mu_1$ and Lipschitz with constant $\mu_2$, by~\eqref{eq:Q1}.  So
  the map is a bijection of $\bb R^d$, by the Browder--Minty
  theorem~\cite{Zeidl90}.  Since $\sigma_\alpha$ is $C^1$ and the
  Jacobian matrix $\partial_p\sigma_\alpha$ is everywhere invertible
  by~\eqref{eq:Q1}, the inverse function theorem shows that the
  inverse $\rho_\alpha(\cdot,s)$ is also $C^1$. The implicit function
  theorem applied to $F(p, q, s) := \sigma_\alpha(p, s) -q =0$, a
  restatement of~\eqref{eq:rho-defn}, shows that the explicit solution
  $p = \rho_\alpha(q, s)$ is jointly $C^1$ as a function of both
  $q \in \bb R^d$ and $s \in \bb  R$.
  Differentiating~\eqref{eq:rho-defn} with respect to $q$, we
  obtain the first equality of~\eqref{eq:rho-derivs}, while
  differentiating it with respect to $s$ gives the second equality
  of~\eqref{eq:rho-derivs}.
  The
  bounds~\eqref{eq:rho-bounds} follow from~\eqref{eq:rho-derivs},
  \eqref{eq:Q1}, and~\eqref{eq:Q2}.
\end{proof}

The second consequence is the unique solvability of~\eqref{eq:4},
which is where the smallness condition~\eqref{eq:Q4} is used.  Its weak
form  finds $u \in \Ho^1$ satisfying
\begin{equation}
  \label{eq:weak-quasilinear}
  - \bigl(\sigma_\alpha(\nabla u, u), \nabla v\bigr)_\om
  + \bigl( \gamma_\alpha(u), v \bigr)_\om
  = (f,v)_\om,
  \qquad v \in \Ho^1.
\end{equation}
Here and throughout $(\cdot, \cdot)_{\om}$ denotes the integral over
$\om$ of the appropriate product of its arguments. In the next result,
we refer to a triple of measurable functions
$S : \om \to \bb R^{d\times d}$, $\tau : \om \to \bb R^d$ and
$c : \om \to \bb R$ as an {\em admissible triple} if, for a.e.~$x \in \om$, $S(x)$
is symmetric with spectrum in $[\mu_1,\mu_2]$, $|\tau(x)| \le \lambda$,
and $-C_\gamma^- \le c(x) \le C_\gamma$.

\begin{lemma}[Well-posedness of~\eqref{eq:weak-quasilinear}]
  \label{lem:quasilinear-wellposed}
  Let Hypothesis~\ref{hyp:sigma} hold and $f \in L_2$.  Then, for
  every $\alpha \in \cQ$, problem~\eqref{eq:weak-quasilinear} has a
  unique solution $u = u(\alpha) \in \Ho^1$.  Moreover, for every
  admissible triple $(S, \tau, c)$, the bilinear form
  \begin{equation}
    \label{eq:frozen-form}
    B_{S,\tau,c}(z, v)
    := (S \nabla z, \nabla v)_\om
    - (\tau\, z, \nabla v)_\om
    + (c\, z, v)_\om
  \end{equation}
  and its adjoint  $(z,v) \mapsto B_{S,\tau,c}(v,z)$ are bounded on
  $\Ho^1$ with constant
  $\mu_2 + \lambda c_P^{-1/2} + C_\gamma c_P^{-1}$ and coercive with
  constant $c_0$ of~\eqref{eq:Q4}.
\end{lemma}
\begin{proof}
  For $z \in \Ho^1$, the Cauchy--Schwarz and
  Friedrichs~\eqref{eq:Friedrichs-std} inequalities give
  \[
    B_{S,\tau,c}(z,z)
    \ge \mu_1 \|\nabla z\|^2_{L_2}
    - \lambda \|z\|_{L_2} \|\nabla z\|_{L_2}
    - C_\gamma^- \|z\|^2_{L_2}
    \ge c_0 \|\nabla z\|^2_{L_2},
  \]
  and the same bound holds for the adjoint form, since
  $B_{S,\tau,c}(z,z)$ is unchanged when its arguments are swapped.  Boundedness is
  immediate from the stated bounds on $S, \tau, c$
  and~\eqref{eq:Friedrichs-std}.

  Now define $N_\alpha : \Ho^1 \to H^{-1}$ by letting
  $\ip{N_\alpha w, v}$ be the left-hand side
  of~\eqref{eq:weak-quasilinear}.
  Given $w^0, w^1 \in \Ho^1$,
  put $e = w^0 - w^1$ and $w^t := w^0 - t\, e$, so that $w^t$ traverses
  the segment from $w^0$ to $w^1$ as $t$ runs from $0$ to $1$.  By the
  chain rule,
  \[\frac{d}{dt} \sigma_\alpha(\nabla w^t, w^t) =
    -\,\partial_p\sigma_\alpha(\nabla w^t, w^t)  \nabla e
    - \partial_s \sigma_\alpha(\nabla w^t, w^t) e,
  \]
  and similarly $\frac{d}{dt}\gamma_\alpha(w^t) = -\gamma_\alpha'(w^t)\, e$.
  Since $w^0$ and $w^1$ are the values of $w^t$ at $t=0$ and $t=1$,
  the FTC applied to $t \mapsto \sigma_\alpha(\nabla w^t, w^t)$ and to
  $t \mapsto \gamma_\alpha(w^t)$, and a change of sign, give
  $\sigma_\alpha(\nabla w^0, w^0) - \sigma_\alpha(\nabla w^1, w^1) =
  -\bar S \nabla e + \bar\tau\, e$ and
  $\gamma_\alpha(w^0) - \gamma_\alpha(w^1) = \bar c\, e$, where
  \begin{equation}
    \label{eq:means}
    \bar S := \int_0^1 \bigl(-\partial_p\sigma_\alpha\bigr)
    (\nabla w^t, w^t)\, dt,
    \quad
    \bar\tau := \int_0^1 \partial_s\sigma_\alpha (\nabla w^t, w^t)\, dt,
    \quad
    \bar c := \int_0^1 \gamma_\alpha'(w^t)\, dt,
  \end{equation}
  are admissible, by~\eqref{eq:Q1}--\eqref{eq:Q3}. Hence $
    \ip{N_\alpha w^0 - N_\alpha w^1, v}
    = B_{\bar S, \bar\tau, \bar c}(e, v),$
for all $ v \in \Ho^1.$
  Taking $v = e$ shows that $N_\alpha$ is strongly monotone with
  constant $c_0$. It  is Lipschitz by the boundedness of
  $B_{\bar S, \bar\tau, \bar c}$.
  The Browder--Minty theorem~\cite{Zeidl90} therefore yields a unique
  solution of~\eqref{eq:weak-quasilinear}.
\end{proof}

A first-order form of the  problem~\eqref{eq:frozen-form} uses
the operator
\begin{equation}
  \label{eq:frozen-operator}
  A_{S,\tau,c}(q, z)
  := \bigl( q + S \nabla z - \tau\, z, \;\; \dive q + c\, z \bigr),
  \qquad (q,z) \in \Hdiv \times \Ho^1.
\end{equation}
We will need it, as well as the following operator that stems from its adjoint,
\begin{equation}
  \label{eq:frozen-transposed}
  T_{S,\tau,c}(r, v)
  := \bigl( -S^{-1} r - \nabla v, \;\;
  \dive r + \tau\cdot S^{-1} r + c\, v \bigr),
  \qquad (r, v) \in \Hdiv \times \Ho^1,
\end{equation}
while performing linearizations.
    Let $\| (q,z) \|^2_{\Hdiv\times\Ho^1}
    := \|q\|^2_{L_2} + \|\dive q\|^2_{L_2} + \|\nabla z\|^2_{L_2}.$

\begin{lemma}
  \label{lem:frozen-operator}
  Let Hypothesis~\ref{hyp:sigma} hold and let $(S, \tau, c)$ be an
  admissible triple.  Then both $A_{S,\tau,c}$ and $T_{S,\tau,c}$ map
  $\Hdiv\times\Ho^1$ bijectively onto $L_2^d \times L_2$ and their
  inverses are bounded by constants depending only on $\mu_1$,
  $\mu_2$, $\lambda$, $C_\gamma$, $c_0$ and $c_P$ (and the bounds are
  uniform over admissible triples).  Moreover, for $A_{S,\tau,c}$, one
  may take
  \begin{equation}
    \label{eq:frozen-two-sided}
    c_A\, \| (q,z) \|_{\Hdiv\times\Ho^1}
    \;\le\; \bigl\| A_{S,\tau,c}(q,z) \bigr\|_{L_2}
    \;\le\; C_A\, \| (q,z) \|_{\Hdiv\times\Ho^1},
    \qquad
  \end{equation}
  for all $(q,z) \in \Hdiv\times\Ho^1$, and, for $T_{S,\tau,c}$,
  \begin{equation}
    \label{eq:frozen-transposed-two-sided}
    c_T\, \| (r,v) \|_{\Hdiv\times\Ho^1}
    \;\le\; \bigl\| T_{S,\tau,c}(r,v) \bigr\|_{L_2}
    \;\le\; C_T\, \| (r,v) \|_{\Hdiv\times\Ho^1}
  \end{equation}
  for all $(r,v) \in \Hdiv\times\Ho^1$, with
  \begin{equation}
    \label{eq:frozen-constants}
    \begin{aligned}
      C_A^2 & = 2 \max \Bigl( 1, \;
      \Bigl( \mu_2 + \tfrac{\lambda}{c_P^{1/2}} \Bigr)^2
      + \tfrac{C_\gamma^2}{c_P} \Bigr),
      \\
      c_A^{-2} & = \Bigl( 1 + c_P^{-1} \Bigr)
      \left[
        \Big( 1 + \frac{\mu_2 + \lambda\, c_P^{-1/2}}{c_0} \Big)^2
        + \Bigl( c_P^{1/2} + \frac{C_\gamma}{c_P^{1/2}\,c_0} \Bigr)^2
        + \frac{1}{c_0^2}
      \right],
      \\
      C_T^2 & = 2 + \frac{1 + \lambda^2}{\mu_1^2}
      + \frac{C_\gamma^2}{c_P},
      \\
      c_T^{-2} & = \mu_2^2 \Bigl( 1 + \frac{C_1}{c_0} \Bigr)^2
      + \left[ (1+\lambda^2)^{1/2}
        + \bigl( \lambda + C_\gamma\, c_P^{-1/2} \bigr) \frac{C_1}{c_0}
      \right]^2
      + \frac{C_1^2}{c_0^2},
    \end{aligned}
  \end{equation}
  where
  $C_1^2 := \bigl( \mu_2 + \lambda\, c_P^{-1/2} \bigr)^2 + c_P^{-1} .$
\end{lemma}
\begin{proof}
  For an $r \in \Hdiv$ and $v \in \Ho^1$, put
   $g = (g_1, g_2) = T_{S,\tau,c}(r, v)$, which is a system of two equations.
  The  first equation
  forces $r = -S(\nabla v + g_1)$. Substituting this into the
  second shows that $v$ must satisfy
  \[
    B_{S,\tau,c}(\zeta, v)
    = (g_2, \zeta)_\om + (\tau\cdot g_1, \zeta)_\om
    - (S g_1, \nabla \zeta)_\om,
    \qquad \text{ for all } \zeta \in \Ho^1.
  \]
  By Lemma~\ref{lem:quasilinear-wellposed} this
  transposed form is bounded and coercive with constant $c_0$, so the
  Lax--Milgram theorem provides a unique $v \in \Ho^1$ with
  $\|\nabla v\|_{L_2} \le c_0^{-1} C_1 \|(g_1,g_2)\|_{L_2}$ (with
  $C_1$ as stated).
  Then
  $r = -S(\nabla v + g_1) \in L_2^d$ and, by the second component,
  $\dive r = g_2 + \tau\cdot g_1 + \tau\cdot\nabla v - c\, v \in L_2$,
  so $r \in \Hdiv$ and $T_{S,\tau,c}$ is bijective.  Moreover
  $\|r\|_{L_2} \le \mu_2 ( \|g_1\|_{L_2} + \|\nabla v\|_{L_2})$ and
  $\|\dive r\|_{L_2} \le (1+\lambda^2)^{1/2}\|(g_1,g_2)\|_{L_2} +
  (\lambda + C_\gamma c_P^{-1/2}) \|\nabla v\|_{L_2}$,
  by~\eqref{eq:Friedrichs-std}. The  squares of these two bounds and of
  the bound on $\|\nabla v\|_{L_2}$ sum to the $c_T^{-2}$
  of~\eqref{eq:frozen-constants}, which is the lower bound
  in~\eqref{eq:frozen-transposed-two-sided}.  Its upper bound follows
  from $\|S^{-1}\| \le \mu_1^{-1}$, the bounds on $\tau$ and $c$,
  and~\eqref{eq:Friedrichs-std}.

  A similar argument also shows the bijectivity of
  $A_{S,\tau,c}$.
  The first equation of $A_{S,\tau,c}(q,z) = (g_1,g_2)$ gives
  $q = g_1 - S\nabla z + \tau\, z$, and substituting it into the
  second shows that $z$ must solve
  \begin{equation}
    \label{eq:BStc}
    B_{S,\tau,c}(z, \zeta)
    = (g_2, \zeta)_\om + (g_1, \nabla \zeta)_\om,
    \qquad \zeta \in \Ho^1,
  \end{equation}
  whose right-hand side is a functional on $\Ho^1$ of norm at most
  $\| g \|_\star := \|g_1\|_{L_2} + c_P^{-1/2}\|g_2\|_{L_2}$, by~\eqref{eq:Friedrichs-std},
  so the argument proceeds with minimal changes.
  For the lower bound in~\eqref{eq:frozen-two-sided},
  put $\zeta = z$ in \eqref{eq:BStc} to get
  \[
    B_{S,\tau,c}(z, z)
    = (S \nabla z, \nabla z)_\om - (\tau\, z, \nabla z)_\om
    + (c\, z, z)_\om
    = (g_1, \nabla z)_\om + (g_2, z)_\om
    \le \| g \|_\star\, \|\nabla z\|_{L_2},
  \]
  where~\eqref{eq:Friedrichs-std} was used.  By
  Lemma~\ref{lem:quasilinear-wellposed} the left-hand side is at least
  $c_0 \|\nabla z\|^2_{L_2}$, so
  $\|\nabla z\|_{L_2} \le \| g \|_\star/c_0$.  Now, we can bound
  $\|q\|_{L_2}$ and $\| \dive q \|_{L_2}$ using
  $q + S \nabla z - \tau\, z= g_1$ and $\dive q + c\, z = g_2$,
  respectively.  Since $\|g_1\|_{L_2} \le \| g \|_\star$ and
  $\|g_2\|_{L_2} \le c_P^{1/2} \| g \|_\star$, the three resulting
  bounds on $\|q\|_{L_2}$, $\|\dive q\|_{L_2}$ and
  $\|\nabla z\|_{L_2}$ are the multiples of $\| g \|_\star$ whose
  squares are summed in the bracket
  of~\eqref{eq:frozen-constants}, while the remaining factor
  $1 + c_P^{-1}$ there accounts for
  $\| g \|^2_\star \le (1 + c_P^{-1}) \|g\|^2_{L_2}$, by the
  Cauchy--Schwarz inequality in $\bb R^2$.  This proves the lower
  bound.
    The upper bound in~\eqref{eq:frozen-two-sided} follows from the
  bounds on $S, \tau, c$ and~\eqref{eq:Friedrichs-std}.
\end{proof}

\subsection{The NFOSLS loss function}
\label{ssec:nfosls-diffusion}

The least-squares instance of \S\ref{ssec:nfosls-abstract} is obtained
from the same first-order system~\eqref{eq:4-mixed}, but written in
the $\sigma_\alpha$ direction and tested against $L_2$ functions.
Namely, we now set
\begin{subequations}
  \label{eq:FOS-A}
  \begin{align}
    & \Ua = \Hdiv \times \Ho^1,
      \qquad
      \Va = L_2^d \times L_2,
    \\
    \label{eq:fos-norms}
    & \| (q, u) \|_{\Ua}^2 = \|q\|^2_{L_2} + \|\dive q\|^2_{L_2}
      + \|\nabla u \|_{L_2}^2,
      \qquad
      \|(r, v)\|_{\Va}^2 = \|r\|^2_{L_2} + \|v\|^2_{L_2},
    \\
    \label{eq:fos-A}
    & \Aa (q, u) = \big( q - \sigma_\alpha(\nabla u, u), \;
      \dive q + \gamma_\alpha(u) \big),
      \qquad
      \ell = (0, f).
  \end{align}
Here  $\Va$ is identified with its own dual through the $L_2$ inner
product,  the $\Rza$ in \eqref{eq:mixedcarst} is $\mathrm{id}$,  $\Vah = \Va$, and
$\fp = \mathrm{id}$ in Definition~\ref{defn:ndpg}. The  resulting
NFOSLS loss function is
\begin{equation}
  \label{eq:fos-loss}
  \Ja{(q, u)}
  = \| q - \sigma_\alpha(\nabla u, u) \|_{L_2}^2
  + \| \dive q + \gamma_\alpha(u) - f \|_{L_2}^2 ,
\end{equation}
\end{subequations}
which requires only quadrature to evaluate.
Since $\Uah$ enters neither~\eqref{eq:fos-loss} nor the analysis
below, it may be any subset of $\Ua$; a standard finite element choice,
in the notation of~\S\ref{ssec:ndpg-diffusion} below, is
\begin{equation}
  \label{eq:fos-discrete}
  \Uah = \RT_k \times U_{k+1},
  \qquad
  \RT_k := (P_k^d + x P_k) \cap \Hdiv,
  \qquad
  U_{k+1} := P_{k+1} \cap \Ho^1,
\end{equation}
but $\Uah$ may equally be a class of neural networks whose outputs lie
in $\Hdiv \times \Ho^1$.  Observe that only $\sigma_\alpha$, and not
its inverse $\rho_\alpha$, is needed to evaluate the loss.

We now verify the assumptions of \S\ref{sec:NDPG} for the
least-squares setting~\eqref{eq:FOS-A}.  Since Assumption~\ref{asm:fortin-a} is
automatic there ($\Vah = \Va$, $\fp = \mathrm{id}$), only
Assumption~\ref{asm:NDPGdiff} remains, and it holds.

\begin{proposition}[Verification of Assumption~\ref{asm:NDPGdiff} for
  the quasilinear NFOSLS setting]
  \label{prop:nfosls-verif}
  Let Hypothesis~\ref{hyp:sigma} hold.  Then the operator
  $\Aa : \Ua \to \Vap$ of~\eqref{eq:fos-A} satisfies all conditions of
  Assumption~\ref{asm:NDPGdiff} with $D = \Ua$ and
  $\delta = +\infty$.  One may take $\CdB = C_A$ and $c_* = c_A$, the
  constants of~\eqref{eq:frozen-constants}.
\end{proposition}
\begin{proof}
  {\em Proof of} {\rm(a)}.  Let $u = u(\alpha) \in \Ho^1$ be the
  solution of~\eqref{eq:weak-quasilinear} furnished by
  Lemma~\ref{lem:quasilinear-wellposed}.
  Using the chain rule to differentiate $\sigma(t p , ts)$ with respect to $t$, the FTC gives
  \[
    \sigma(p, s) - \sigma(0, 0) =
    \int_0^1
    \Big[ (\partial_p \sigma)(tp, ts) \,p +
    (\partial_s \sigma) (tp, ts) \; s \Big] \, dt .
  \]
  Therefore,~\eqref{eq:Q1}--\eqref{eq:Q3} imply
  $|\sigma_\alpha(p,s)| \le |\sigma_\alpha(0,0)| + \mu_2|p| +
  \lambda|s|.$ Similarly, $|\gamma_\alpha(s)| \le |\gamma_\alpha(0)| +
  C_\gamma|s|$. We conclude that
  $q := \sigma_\alpha(\nabla u, u) \in L_2^d$ and
  $\gamma_\alpha(u) \in L_2$.
  Equation~\eqref{eq:weak-quasilinear} says that
  $\dive q = f - \gamma_\alpha(u)$ holds distributionally, so
  $q \in \Hdiv$ and
  $\clu := (q,u) \in \Ua$ satisfies $\Aa(\clu) = \ell$, i.e.,
  $\Ra\clu = 0$.  Since $\Aa$ is defined on all of $\Ua$, we take
  $D = \Ua$.

  {\em Proof of} {\rm(b)}.  Only $\sigma_\alpha(\nabla u, u)$ and
  $\gamma_\alpha(u)$ in~\eqref{eq:fos-A} are nonlinear, so the
  difference quotient of~\eqref{eq:gateaux}, computed in
  $\Va = L_2^d \times L_2$, takes the form
  \begin{equation}
    \label{eq:fos-dq}
    \frac{1}{t}\big[ \Aa(q + t\varsigma, \, u + t z) - \Aa(q, u)\big]
    = \bigl( \varsigma - \Delta_t\sigma, \;\;
    \dive \varsigma + \Delta_t\gamma \bigr),
  \end{equation}
  in which every term but the two involving the pointwise defined
  difference quotients
  \begin{equation}
    \label{eq:fos-dq-nonlinear}
    \Delta_t\sigma
    := \frac{\sigma_\alpha(\nabla u + t \nabla z, \, u + t z)
      - \sigma_\alpha(\nabla u, u)}{t},
    \qquad
    \Delta_t\gamma
    := \frac{\gamma_\alpha(u + t z) - \gamma_\alpha(u)}{t},
  \end{equation}
  is independent of $t$.  Both quotients are generated by $C^1$
  functions with bounded derivatives, by~\eqref{eq:Q1}--\eqref{eq:Q3}.
  Indeed, writing $p_t^\theta := \nabla u + \theta t \nabla z$ and
  $u_t^\theta := u + \theta t z$, the FTC gives, a.e.\ in $\om$,
  \begin{equation}
    \label{eq:fos-dq-ftc}
    \Delta_t\sigma
    = \int_0^1 \Big[
    \partial_p\sigma_\alpha(p_t^\theta, u_t^\theta) \, \nabla z
    + \partial_s\sigma_\alpha(p_t^\theta, u_t^\theta) \, z
    \Big] \, d\theta,
    \qquad
    \Delta_t\gamma
    = \int_0^1 \gamma_\alpha'(u_t^\theta)\, z \; d\theta .
  \end{equation}
  Therefore, by~\eqref{eq:Q1}--\eqref{eq:Q3}, the difference quotients
  are dominated, uniformly in $t$, by
  \[
    |\Delta_t\sigma| \le \mu_2 |\nabla z| + \lambda |z|,
    \qquad
    |\Delta_t\gamma| \le C_\gamma |z|,
  \]
  which are functions in $L_2$.  Letting $t \to 0$
  in~\eqref{eq:fos-dq-ftc}, the continuity of
  $\partial_p\sigma_\alpha$, $\partial_s\sigma_\alpha$ and
  $\gamma_\alpha'$ shows that
  $\Delta_t\sigma \to \partial_p\sigma_\alpha(\nabla u, u)\, \nabla z
  + \partial_s\sigma_\alpha(\nabla u, u)\, z$ and
  $\Delta_t\gamma \to \gamma_\alpha'(u)\, z$ pointwise almost
  everywhere.  The dominated convergence theorem upgrades these to
  convergence in $L_2$, so the limit in~\eqref{eq:gateaux} exists and
  \begin{equation}
    \label{eq:fos-dA}
    d\Aa(q, u)(\varsigma, z)
    = \bigl( \varsigma - \partial_p\sigma_\alpha(\nabla u,u)\, \nabla z
    - \partial_s\sigma_\alpha(\nabla u,u)\, z, \;
    \dive \varsigma + \gamma_\alpha'(u)\, z \bigr),
  \end{equation}
  for $(\varsigma,z) \in \Ua$.  The same domination and continuity
  give the continuity along segments.  Comparing~\eqref{eq:fos-dA}
  with~\eqref{eq:frozen-operator}, we see that
  $d\Aa(q,u) = A_{S,\tau,c}$ for the triple
  $S = -\partial_p\sigma_\alpha(\nabla u, u)$,
  $\tau = \partial_s\sigma_\alpha(\nabla u,u)$ and
  $c = \gamma_\alpha'(u)$, which is admissible
  by~\eqref{eq:Q1}--\eqref{eq:Q3}.  Since the norm of $\Ua$
  in~\eqref{eq:fos-norms} is the norm $\|\cdot\|_{\Hdiv\times\Ho^1}$
  of Lemma~\ref{lem:frozen-operator}, the upper bound
  in~\eqref{eq:frozen-two-sided} gives~\eqref{linvc} with
  $\CdB = C_A$.

  {\em Proof of} {\rm(c)}.  Let $\clz^i = (\varsigma^i, z^i) \in \Ua$,
  $i = 0, 1$, be arbitrary and abbreviate
  $e_\varsigma = \varsigma^0 - \varsigma^1$, $e_z = z^0 - z^1$.
  Put $z^t := z^0 - t\, e_z$ and let
  $\bar S, \bar\tau, \bar c$ be the mean
  coefficients~\eqref{eq:means} obtained by replacing $\nabla w^t$ and
  $w^t$ there by $\nabla z^t$ and $z^t$, respectively; they are
  an admissible triple in the sense of
  Lemma~\ref{lem:quasilinear-wellposed}.  By the FTC,
  $\sigma_\alpha(\nabla z^0, z^0) - \sigma_\alpha(\nabla z^1, z^1) =
  -\bar S \nabla e_z + \bar\tau\, e_z$ and
  $\gamma_\alpha(z^0) - \gamma_\alpha(z^1) = \bar c\, e_z$.
  Subtracting two copies of~\eqref{eq:fos-A},
  \begin{align*}
    \Aa(\clz^0) - \Aa(\clz^1)
    & = \bigl(\, e_\varsigma
    - \bigl[ \sigma_\alpha(\nabla z^0, z^0)
    - \sigma_\alpha(\nabla z^1, z^1) \bigr], \;\;
    \dive e_\varsigma
      + \bigl[ \gamma_\alpha(z^0) - \gamma_\alpha(z^1) \bigr] \,\bigr)
    \\
    & = \bigl( e_\varsigma + \bar S \nabla e_z - \bar\tau\, e_z, \;\;
    \dive e_\varsigma + \bar c\, e_z \bigr)
    \\
    & = A_{\bar S, \bar\tau, \bar c}(\clz^0 - \clz^1).
  \end{align*}
  The lower bound of Lemma~\ref{lem:frozen-operator} now gives
  $\|\Aa(\clz^0) - \Aa(\clz^1)\|_{\Vap} \ge c_A \|\clz^0 -
  \clz^1\|_{\Ua}$.  As this holds for every pair, it
  is~\eqref{eq:c-prime} with $c_* = c_A$ and $\delta = +\infty$, and
  it also shows that the $\clu$ of part {\rm(a)} is the
  unique element at which $\Ra\cdot$ vanishes.
\end{proof}

\begin{theorem}[Variational correctness of the NFOSLS loss for the quasilinear case]
  \label{thm:vc-NFOSLS-example}
  Let Hypothesis~\ref{hyp:sigma} hold
  and adopt the setting of
  \eqref{eq:FOS-A}.
  Then the NFOSLS loss function
  of~\eqref{eq:fos-loss} is globally variationally correct on the
    entire trial space:  the estimate
  \begin{equation}
    \label{eq:nfosls-vc-example}
    \frac{c_*^2}{2}\, \| \clu - \clz \|_{\Ua}^2
    \;\le\; \Ja \clz
    \;\le\; \CdB^2\, \| \clu - \clz \|_{\Ua}^2
  \end{equation}
  holds for any $\clz \in \Ua = \Hdiv \times \Ho^1$, any
  $\alpha \in \cQ$, and for arbitrary polynomial degree $k\ge 0$
  in~\eqref{eq:fos-discrete}. We may take $c_*=c_A$ and $\CdB=C_A$,
  the constants defined in~\eqref{eq:frozen-constants}.
\end{theorem}
\begin{proof}
  Proposition~\ref{prop:nfosls-verif} verifies
  Assumption~\ref{asm:NDPGdiff} with $D = \Ua$ and $\delta = +\infty$,
  while Assumption~\ref{asm:fortin-a} holds with $\Vah = \Va$ and
  $\fp = \mathrm{id}$, so that $\|\fp\|_{L(\Va)} = 1$ and
  $\osc(\ell) = 0$.  Theorem~\ref{thm:NDPG-gen-nonlinear-dpg} then
  applies with $h_0 = +\infty$ and $\tilde D_{h_0} = \Ua$, giving the
  result (as already noted in~\eqref{eq:ls-vc}).
\end{proof}

\subsection{The NDPG loss function}
\label{ssec:ndpg-diffusion}

We turn now to the DPG case, using a shape-regular simplicial finite
element mesh $\oh$ of $\om \subset \bb R^d$.
Denote by $n$
the outward unit normals on element boundaries, and let $\tr(w)$ and
$\trn(r)$ denote element-wise trace operators characterized as
$\tr(w)|_{\partial K} = w|_{\partial K}$ and
$\trn(r)|_{\partial K} = (r \cdot n )|_{\partial K}$. By standard
trace theorems, $\tr(\cdot)$ can be applied to $w \in \Ho^1$ and
$\trn(\cdot)$ to $q \in \Hdiv$.  Let $\Hoh = \tr \Ho^1$ and
$\Hmh = \trn \Hdiv$, endowed with the quotient norms
\begin{equation}
  \label{eq:quotient-norms}
  \|\hat u\|_{{\Hoh}}
  := \inf_{u \in \text{tr}^{-1}\{\hat u\}} \|\nabla u\|_{L_2},
  \qquad
  \|\hat r_n\|_{H^{-1/2}(\partial \Omega_h)}
  := \inf_{ q \in \text{tr}_n^{-1}\{\hat r_n\}} \| q\|_{H(\div)},
\end{equation}
respectively.
In addition to the standard  spaces, we now also need the  product spaces, colloquially referred to as ``broken'' spaces
\begin{equation}
  \label{eq:13}
  \Hdivoh = \bigtimes_{K \in \oh} H(\div, K), \qquad
  H^1(\oh) = \bigtimes_{K \in \oh} H^1(K),
\end{equation}
representing piecewise $H(\div)$ and $H^1$ spaces. Functions in them
do not have any interelement continuity constraints. When computing differential operators on them, we do so piecewise, and use an  abbreviated notation
for sums over elements,
\[
  (\div r, w)_{\oh} = \sum_{K \in \oh} \int_K w \div r,\qquad
  (r, \nabla w)_{\oh} = \sum_{K \in \oh} \int_K r \cdot \nabla w,
  \qquad
\]
for all $r \in \Hdivoh, \; w \in H^1(\oh),$ where the derivatives are
taken element by element.  For
$\hat r_n \in H^{-1/2}(\partial \Omega_h), r \in \Hdiv, w \in \Ho^1,
\hat w \in \Hoh$, the  abbreviated notation
\[
  \ip{ \hat r_n, w }_{\partial \oh} = \sum_{K \in \oh} \ip{ \hat r_n,
    w}_{H^{1/2}(\partial K)}, \quad
  \ip{ r \cdot n, \hat w}_{\partial\oh} =
  \sum_{K \in \oh} \ip{ r\cdot n, \hat w}_{H^{1/2}(\partial K)},
\]
for these often-occurring sums over element boundaries will be useful.
The broken spaces in~\eqref{eq:13} are normed using standard product norms, i.e.,
$\| r \|_{\Hdivoh}^2 = (r, r)_\om + (\div r, \div r)_\oh$ and $\| v \|_{H^1(\oh)}^2 = (v,v)_\om + (\grad v, \grad v)_\oh$.
There is an analogue of  the standard Friedrichs inequality~\eqref{eq:Friedrichs-std}  for the broken space $H^1(\oh)$,  namely, there is a mesh-independent $\tilde c_P>0$ such that
\begin{gather}
\label{eq:Friedrichs-generalized}
  \tilde c_P \| v \|_{L_2}^2
  \le
  \sup_{q \in \Hdiv} \left(
    \frac{ \ip{ q\cdot n, v}_{\partial \oh}}
    { \| q \|_{\Hdiv}} \right)^2
   + \sum_{K\in \oh} \|\nabla v \|_{L_2(K)}^2
 \end{gather}
for all $v \in H^1(\oh)$, which is proved
in~\cite[Lemma~4.1]{DemkoGopal11a}.
It is easy to see that  $\tilde c_P$ cannot be larger than  the  $c_P$
of~\eqref{eq:Friedrichs-std}.

The DPG formulation of~\eqref{eq:4-mixed} seeks $\clu \in \Ua$ satisfying
\begin{subequations}
  \label{eq:DPG-A}
\begin{equation}
  \label{eq:12}
  a_\alpha(\clu, \clv) = \ell(\clv), \qquad \text{ for all } \clv \in \Va,
\end{equation}
where
\begin{align}
    & \Ua
    = L_2(\om)^d \times L_2(\om) \times \Hoh\times  H^{-1/2}(\partial\oh),
    \\
    & \Va
    = \Hdivoh \times H^1(\oh),
    \\
    & a_\alpha((q,u,\hat u, \hat q_n),    (r,v))
      =
      -\ip{ r\cdot n, \hat u}_{\partial\oh}
      + \ip{ \hat q_n, v}_{\partial\oh}
      \nonumber
  \\
    &
      \hspace{2.2cm} + (\rho_\alpha(q,u), r)_{\oh} + (u, \div r)_\oh - (q, \nabla v )_\oh
      + (\gamma_\alpha(u), v)_\om
      \label{eq:DPGinfiniteBF}
  \\
    & \ell((r,v)) = (f, v)_{L_2}, \label{eq:DPGinfiniteLF}
  \end{align}
By~\eqref{eq:aadaa}, the above $a_\alpha$ defines the nonlinear operator
$\Aa$ that casts~\eqref{eq:DPG-A} into the form~\eqref{eq:Au=ell},
\begin{equation}
  \label{eq:14}
  \Aa(\clu) = \ell, \quad\text{with}\quad\clu = (q, u, \hat u, \hat q_n),
\end{equation}
\end{subequations}
and we take $D = \Ua$, the forms being well defined there because
$\rho_\alpha$ and $\gamma_\alpha$ have linear growth by
Lemma~\ref{lem:rho-exists} and~\eqref{eq:Q3}.  The
formulation~\eqref{eq:DPG-A} is derived by the standard DPG technique
of multiplying the equations of~\eqref{eq:4-mixed-a}
and~\eqref{eq:4-mixed-b} by test functions $r$ and $v$, respectively,
integrating by parts element by element, and declaring unknown traces
on element boundaries as new variables---see, e.g.,~\cite{DPGacta} for
more details on this process.  The norms on $\Ua$ and $\Va$ are
product norms built from~\eqref{eq:quotient-norms}
and~\eqref{eq:13}, namely
\begin{equation}
  \label{eq:uw-norms}
  \begin{aligned}
    \|\clu\|^2_{\Ua}
    & = \| ( q, u, \hat u, \hat q_n) \|_{\Ua}^2
      = \|q\|^2_{L_2} + \|u\|^2_{L_2}
    + \|\hat u\|^2_{H^{1/2}(\partial\oh)}
    + \|\hat q_n\|^2_{H^{-1/2}(\partial\oh)},
    \\
    \| \clv \|^2_{\Va}
    & = \| (r, v) \|_{\Va}^2 = \| r \|_{\Hdivoh}^2 + \| v \|^2_{H^1(\oh)}.
  \end{aligned}
\end{equation}

For stating a conforming computable discrete formulation, let ${P}_k= P_k(\oh)$
denote the space of piecewise polynomials of degree at most $k$ on
each element of $\om_h$ (without any inter-element continuity
conditions), i.e., if $P_k(K)$ denotes the space of restrictions of
polynomials to an element $K \in \oh$, the piecewise space is
$P_k = \bigtimes_{K \in \oh} P_k(K).$ The lowest-order Raviart--Thomas
subspace of $\Hdiv$ is $ \RT= \left [ P_0^d + x P_0 \right ] \cap \Hdiv $ where $x$ is the
coordinate vector. We will also need the lowest-order Lagrange space
on the same mesh, $ U_1 = {P}_1 \cap {\Ho}^1.$ The subspaces of traces
are $\hat U_1 = \tr( U_1)$ and $\hat\RT = \trn (\RT)$ on element
boundaries, i.e.,
$ \hat U_1 = \{ \hat w: \; \hat w|_{\partial K} = w|_{\partial K}
\text{ for some } w \in U_1 \text{ for all } K \in \oh \}$ and
$ \hat\RT = \{ \hat q_n: \; \hat q_n|_{\partial K} = (r\cdot
n)|_{\partial K} \text{ for some } r \in \RT \text{ for all } K \in
\oh \}$.  Then
\begin{equation}
  \label{eq:uw-discrete}
\Uah = P_0^d \times P_0 \times U_1 \times \hat\RT,
\qquad
\Vah = P_2^d \times P_{d}
\end{equation}
are subspaces of $\Ua$ and $\Va$, respectively. With these choices,
the DPG method is defined as usual~\cite{DemkoGopal11,DPGacta}. In
particular, all ingredients to specify the NDPG loss function of
\S\ref{ssec:ndpg-abstract}, i.e., $\Ja \clw$
of~\eqref{eq:DPG-loss-defn} with the broken $\Vah$
of~\eqref{eq:uw-discrete}, are now in place for any $\clw \in \Ua$.

We now verify Assumptions~\ref{asm:NDPGdiff} and~\ref{asm:fortin-a}
for the ultraweak formulation~\eqref{eq:DPG-A}, which defines $\Aa$,
$a_\alpha$, $\Ua$, and $\Va$, under Hypothesis~\ref{hyp:sigma}.  We begin with
Assumption~\ref{asm:NDPGdiff}.

\begin{proposition}[Verification of Assumption~\ref{asm:NDPGdiff}]
  \label{prop:NDPGdiff-verif'}
  Let Hypothesis~\ref{hyp:sigma} hold.  Then, in the setting
  of~\eqref{eq:DPG-A}, the operator $\Aa : \Ua \to \Vap$ satisfies all
  conditions of Assumption~\ref{asm:NDPGdiff} with $D = \Ua$ and
  $\delta = +\infty$.  Specifically:
  \begin{enumerate}[label={\rm(\alph*)}]
  \item for each $\alpha \in \cQ$ there is a $\clu = \clu(\alpha) \in
    \Ua$ with $\Ra\clu = 0$;
  \item the G\^ateaux derivative $d\Aa(\clw) : \Ua \to \Vap$
    of~\eqref{eq:gateaux} exists at each $\clw \in \Ua$ and
    $\alpha \in \cQ$, is given by~\eqref{eq:dA}, depends continuously
    on $\clw$ along line segments, and
    satisfies $\|d\Aa(\clw)(\clz)\|_{\Vap} \le \CdB\|\clz\|_{\Ua}$ with
    \begin{equation}
      \label{eq:uw-Cd}
      \CdB^2 = 3 + c_P^{-1} + \frac{1}{\mu_1^2}
      + \Bigl( 1 + \frac{\lambda}{\mu_1} \Bigr)^2 + C_\gamma^2 ;
    \end{equation}
  \item inequality~\eqref{eq:c-prime} holds for all $\clz \in \Ua$,
    i.e., with $\delta = +\infty$, with $c_* = \cdB$ given by
    \begin{equation}
      \label{eq:uw-cd}
      \frac{1}{\cdB^2}
      = \frac{1 + c_P^{-1}}{c_T^2}
      + \Bigl( 1 + \frac{(1 + c_P^{-1})^{1/2}\, C_T}{c_T} \Bigr)^2,
    \end{equation}
    where $c_T$ and $C_T$ are the constants
    of~\eqref{eq:frozen-transposed-two-sided}; in particular $c_*$
    depends only on $\mu_1$, $\mu_2$, $\lambda$, $C_\gamma$, $c_0$ and
    $c_P$.
  \end{enumerate}
  In particular, since $\delta = +\infty$ and $D = \Ua$, part {\rm(c)}
  implies that the $\clu$ in part {\rm(a)} is unique, and is the
  unique minimizer in~\eqref{eq:exact-minimization}.
\end{proposition}
\begin{proof}
  {\em Proof of} {\rm(a)}.  As in the proof of
  Proposition~\ref{prop:nfosls-verif}{\rm(a)}, the solution
  $u = u(\alpha) \in \Ho^1$ of~\eqref{eq:weak-quasilinear} gives
  $q := \sigma_\alpha(\nabla u, u) \in \Hdiv$ with
  $\dive q = f - \gamma_\alpha(u)$.  Hence
  the trace $\hat q_n := \trn(q)$, in addition
  to $\hat u := \tr(u)$,
   exists.  Elementwise integration by parts then
   shows that $\clu(\alpha) = (q,u,\hat u,\hat q_n) \in \Ua$
   satisfies $\Aa(\clu) = (0, f)$,
   i.e.,
   $\Ra\clu = 0$.  Since
   $\Aa$ is defined on all of $\Ua$
   and this argument did not restrict $\clu(\alpha)$ to any neighborhood,
   we take
  $D = \Ua$.

  {\em Proof of} {\rm(b)}.  Let $\clw = (\varpi, w, \hat w, \hat\varpi_n)$ and
  $\clz = (\varsigma, z, \hat z, \hat\varsigma_n)$ be in $\Ua$ and
  $\clv = (r, v) \in \Va$.
  Among the terms of $\Aa$ in~\eqref{eq:DPG-A}, only
  $(\rho_\alpha(\varpi, w), r)_{\oh}$ and $(\gamma_\alpha(w), v)_\om$
  fail to be linear in $\clw$.
  These terms are treated exactly as in the proof of
  Proposition~\ref{prop:nfosls-verif}{\rm(b)}, with $\rho_\alpha$ in
  place of $\sigma_\alpha$: their difference quotients are generated
  by $C^1$ functions with bounded derivatives, by
  Lemma~\ref{lem:rho-exists} and~\eqref{eq:Q3}, and the FTC represents
  them as integrals whose integrands are dominated, uniformly in $t$,
  by $\mu_1^{-1}|\varsigma| + (\lambda/\mu_1)|z|$ and $C_\gamma|z|$,
  respectively, in view of~\eqref{eq:rho-bounds}
  and~\eqref{eq:Q3}.  The dominated convergence theorem therefore
  gives the limit in~\eqref{eq:gateaux},
  \begin{align}
    \label{eq:dA}
    d\Aa(\clw)(\clz)(\clv)
    &=-\langle r\cdot n,\hat z\rangle_{\partial\Omega_h}+
      \langle \hat\varsigma_n,v\rangle_{\partial\Omega_h}
      + \bigl( \partial_q\rho_\alpha(\varpi, w)\, \varsigma
      + \partial_s\rho_\alpha(\varpi,w)\, z, \; r\bigr)_{\Omega_h}
      \nonumber\\
    &\quad +(z,\dive\,r)_{\Omega_h}
      - (\varsigma,\nabla v)_{\Omega_h}
      + \bigl( \gamma_\alpha'(w)\, z, v \bigr)_{\om}
  \end{align}
  for any $\clv = (r,v) \in \Va$.  The same domination and continuity
  show that $t \mapsto d\Aa(\clw + t\clz)(\clz)(\clv)$ is
  continuous.  Finally, to show the boundedness of
  $d \Aa (\clw): \Ua \to \Vap$, group the terms of~\eqref{eq:dA} by
  test component.  Elementwise integration by parts, the quotient
  norms~\eqref{eq:quotient-norms} and~\eqref{eq:Friedrichs-std} give
  \[
    |\ip{ r\cdot n, \hat z}_{\partial\oh}| \le (1 + c_P^{-1})^{1/2}\,
    \|\hat z\|_{H^{1/2}(\partial\oh)} \|r\|_{\Hdivoh}, \qquad
    |\ip{ \hat\varsigma_n, v}_{\partial\oh}| \le \|\hat
    \varsigma_n\|_{H^{-1/2}(\partial\oh)} \|v\|_{H^1(\oh)},
  \]
  while the
  volume terms are bounded using the Cauchy--Schwarz inequality
  and~\eqref{eq:rho-bounds}, \eqref{eq:Q3}.  The terms tested by $r$
  are thus bounded by
  $\bigl[ (1 + c_P^{-1})^{1/2} \|\hat z\| + \mu_1^{-1}\|\varsigma\| + (1
  + \lambda\mu_1^{-1})\|z\| \bigr] \|r\|_{\Hdivoh}$ and those tested
  by $v$ by
  $\bigl[ \|\hat\varsigma_n\| + \|\varsigma\| + C_\gamma \|z\| \bigr]
  \|v\|_{H^1(\oh)}$.  Two applications of the Cauchy--Schwarz
  inequality, first in each bracket and then in the two test
  components, give~\eqref{eq:uw-Cd}.  This proves {\rm(b)}.

  {\em Proof of} {\rm(c)}.  Fix $\clz^0$ and $\clz^1$ in $\Ua$, with
  components $\clz^i = (\varsigma^i, z^i, \hat z^i, \hat\varsigma^i_n)$,
  and abbreviate
  $\clz^0 - \clz^1
  = (e_\varsigma, e_z, \hat e_z, \hat e_{\varsigma_n})$.
  Subtracting two copies of~\eqref{eq:DPGinfiniteBF}, we obtain, for
  any $\clv = (r,v) \in \Va$,
  \begin{equation}
    \label{eq:uw-difference}
    \begin{aligned}
      \ip{ \Aa(\clz^0) - \Aa(\clz^1), \clv}
      & = -\ip{ r\cdot n, \hat e_z}_{\partial\oh}
        + \ip{ \hat e_{\varsigma_n}, v}_{\partial\oh}
        + \bigl( \rho_\alpha(\varsigma^0, z^0)
        - \rho_\alpha(\varsigma^1, z^1), \; r \bigr)_{\oh}
      \\
      & \quad
        + (e_z, \dive r)_{\oh} - (e_\varsigma, \nabla v)_{\oh}
        + \bigl( \gamma_\alpha(z^0) - \gamma_\alpha(z^1), \; v
        \bigr)_{\om},
    \end{aligned}
  \end{equation}
  where every term other than the two nonlinear differences is already
  a linear function of $\clz^0 - \clz^1$.  It remains to express those
  two differences through $\clz^0 - \clz^1$ as well.  Putting
  $z^t := z^0 - t\, e_z$ for $t \in [0,1]$, which runs from $z^0$ to
  $z^1$, the reaction difference is immediate from the FTC,
  \begin{equation}
    \label{eq:reaction-ftc}
    \gamma_\alpha(z^0) - \gamma_\alpha(z^1) = \bar c\, e_z,
    \qquad
    \bar c := \int_0^1 \gamma_\alpha'(z^t)\, dt .
  \end{equation}

  For the flux difference, we proceed a bit differently from the
  least-squares case: we linearize in
  the arguments on which the forward flux
  law $\sigma_\alpha$ acts, rather than in the trial variables
  $(\varsigma^i, z^i)$ on which its inverse  $\rho_\alpha$ acts.  Accordingly, set,
  pointwise a.e.\ in $\om$,
  $
    p^i := \rho_\alpha(\varsigma^i, z^i)$ and
    $p^t := p^0 - t\, (p^0 - p^1)$
    for all   $t \in [0,1]$, so that $p^t$ and $z^t$ run from $p^0$
    and $z^0$ to $p^1$ and $z^1$,
  and let $\bar S, \bar\tau$ be given by~\eqref{eq:means} with
  $(\nabla w^t, w^t)$ there replaced by $(p^t, z^t)$.  Then
  $\bar S, \bar\tau, \bar c$ is an admissible triple in the sense of
  Lemma~\ref{lem:quasilinear-wellposed}.  Since
  $\sigma_\alpha(p^i, z^i) = \varsigma^i$ by~\eqref{eq:rho-defn}, the FTC
  applied to
  $t \mapsto \sigma_\alpha(p^t, z^t)$ gives
  $e_\varsigma = \varsigma^0 - \varsigma^1   = -\bar S\, (p^0 - p^1) + \bar\tau\, e_z,
  $ which implies
  $\rho_\alpha(\varsigma^0,z^0) - \rho_\alpha(\varsigma^1,z^1)
    = p^0 - p^1
    = \bar S^{-1}\bigl( \bar\tau\, e_z - e_\varsigma \bigr),$
    since the matrix $\bar S$ is invertible by~\eqref{eq:Q1}.  Substituting this
    and~\eqref{eq:reaction-ftc}
  into~\eqref{eq:uw-difference}, we arrive at
  \begin{equation}
    \label{eq:uw-frozen}
    \ip{ \Aa(\clz^0) - \Aa(\clz^1), \clv}
    = b_{\bar S, \bar\tau, \bar c}(\clz^0 - \clz^1, \clv),
    \qquad \clv \in \Va,
  \end{equation}
  where, for an admissible triple $S, \tau, c$,
  \begin{align}
    \label{eq:uw-frozen-form}
    b_{S,\tau,c}\bigl( (\varsigma, z, \hat z, \hat \varsigma_n), (r,v) \bigr)
    & := -\ip{ r\cdot n, \hat z}_{\partial\oh}
      + \ip{ \hat\varsigma_n, v}_{\partial\oh}
      + \bigl( S^{-1}( \tau\, z - \varsigma), r \bigr)_{\oh}
      \nonumber
    \\
    & \qquad
      + (z, \dive r)_\oh - (\varsigma, \nabla v)_\oh
      + (c\, z, v)_\om .
  \end{align}
  We claim that
  \begin{equation}
    \label{eq:uw-infsup}
    \sup_{0 \ne \clv \in \Va}
    \frac{ | b_{S,\tau,c}(\clz, \clv) |}{\|\clv\|_{\Va}}
    \;\ge\; \cdB\, \|\clz\|_{\Ua},
    \qquad \clz \in \Ua,
  \end{equation}
  with the $\cdB$ of~\eqref{eq:uw-cd}, uniformly over all admissible
  triples $S, \tau, c$ and all $\alpha \in \cQ$.  Granting the claim, \eqref{eq:uw-frozen} gives
  $\|\Aa(\clz^0) - \Aa(\clz^1)\|_{\Vap} \ge \cdB \|\clz^0 -
  \clz^1\|_{\Ua}$ for {\em all} $\clz^0, \clz^1 \in \Ua$, which
  is~\eqref{eq:c-prime} with $c_* = \cdB$ and $\delta = +\infty$ upon
  setting $\clz^0 = \clu(\alpha)$; since the estimate holds for all
  pairs, the element $\clu$ of part {\rm(a)} is also the unique one at
  which $\Ra\cdot$ vanishes.

  It remains to prove the claim~\eqref{eq:uw-infsup}.  The form $b_{S,\tau,c}$ is the
  ultraweak DPG bilinear form of the linear boundary value
  problem
  \begin{equation}
    \label{eq:uw-frozen-bvp}
    \varsigma = \tau\, z - S \nabla z,
    \qquad
    \dive \varsigma + c\, z = g \; \text{ in } \om,
    \qquad z|_{\partial\om} = 0,
  \end{equation}
  i.e., of $-\dive( S\nabla z - \tau\, z) + c\, z = g$, whose primal weak
  form is $B_{S,\tau,c}$ of~\eqref{eq:frozen-form}.  By the theory of
  broken variational
  formulations (see~\cite[Theorem~3.3]{CarstDemkoGopal16} or \cite[Theorem~4.3]{DPGacta})
  the claim~\eqref{eq:uw-infsup} holds, with a constant depending only on the
  corresponding unbroken constant and on the norms
  in~\eqref{eq:quotient-norms}, as soon as the {\em unbroken}
  ultraweak form (which is \eqref{eq:uw-frozen-form} with interface terms removed)
  \[
    b^0_{S,\tau,c}\bigl( (\varsigma, z), (r,v)\bigr)
    := \bigl( S^{-1}(\tau\,z - \varsigma), r \bigr)_\om
    + (z, \dive r)_\om - (\varsigma, \nabla v)_\om + (c\, z, v)_\om
  \]
  is inf-sup stable on
  $\bigl( L_2^d \times L_2 \bigr) \times \bigl( \Hdiv \times
  \Ho^1\bigr)$.  Since $S$, and hence $S^{-1}$, is symmetric, observe that $
    b^0_{S,\tau,c}\bigl( (\varsigma, z), (r,v) \bigr)
    = \bigl( (\varsigma, z), \, T_{S,\tau,c}(r,v) \bigr)_\om
  $
  with $T_{S,\tau,c}$ as in~\eqref{eq:frozen-transposed}.  The trial
  space being an $L_2$ product, the required stability is therefore
  equivalent to the bijectivity of $T_{S,\tau,c}$ onto
  $L_2^d\times L_2$ with a uniformly bounded inverse, which is
  provided by Lemma~\ref{lem:frozen-operator}.

  This also makes $\cdB$ explicit.  Indeed,
  \cite[Theorem~3.3]{CarstDemkoGopal16} gives
  $\cdB^{-2} = c^{-2}_{\rm ub} + \hat c^{-2}( \|b^0\|/c_{\rm ub} +
  1)^2$, where $c_{\rm ub}$ and $\|b^0\|$ are the inf-sup and
  continuity constants of $b^0_{S,\tau,c}$ and $\hat c$ is the
  stability constant of the interface terms, which equals one because
  the quotient norms~\eqref{eq:quotient-norms} are the dual norms of
  those terms.  Taking $(r,v) = T_{S,\tau,c}^{-1}(\varsigma, z)$
  in the supremum, and using
  $\|(r,v)\|_{\Va} \le (1+c_P^{-1})^{1/2}\|(r,v)\|_{\Hdiv\times\Ho^1}$
  by~\eqref{eq:Friedrichs-std}, the two-sided
  bound~\eqref{eq:frozen-transposed-two-sided} gives
  $c_{\rm ub} \ge c_T (1+c_P^{-1})^{-1/2}$ and $\|b^0\| \le C_T$,
  whence~\eqref{eq:uw-cd}.
\end{proof}

Next, we turn to the verification of Assumption~\ref{asm:fortin-a}.  For
the simplicial DPG constructions here and in
\S\ref{ssec:kirchhoff-setting}, we take $d\ge2$.  At lowest order we
need a scalar Fortin operator that also preserves element averages.
We adapt a mean correction trick of~\cite{CortesGopal26, FuhreHeuer24}
(but unlike both those works, we now use a quadratic polynomial having
zero face averages).  On a simplex $K$ with barycentric coordinates
$\lambda_0,\ldots,\lambda_d$, set
\begin{equation}
  \label{eq:fortin-mean-correction}
  \begin{aligned}
    \psi_K
    &:= (d+2)\left(1-\frac{d+1}{2}
      \sum_{i=0}^d\lambda_i^2\right),
    \\
    \tilde\Pi_d v &:= \Pi_d^{\rm grad}v
      + \frac{(v-\Pi_d^{\rm grad}v,1)_K}{|K|}\,\psi_K,
      \qquad v\in H^1(K),
  \end{aligned}
\end{equation}
where $\Pi_d^{\rm grad}$ is the operator
of~\cite[Lemma~3.2]{GopalQiu14} at $p=0$.  The barycentric integration
formulas give $(\psi_K,1)_K=|K|$ and
$\int_F\psi_K=0$ for every face $F$ of $K$.  Consequently,
$\tilde\Pi_d$ preserves element and face averages and reproduces constants:
\begin{equation}
  \label{eq:fortin-preservation}
    (v-\tilde\Pi_d v,1)_K
    = 0,
    \qquad
    (v-\tilde\Pi_d v, 1)_F
    = 0,
    \qquad
    \tilde\Pi_d c = c,
\end{equation}
for all $v\in H^1(K)$, all facets $F$ of $K$, and all constants
$c \in P_0(K)$.  It maps into $P_d(K)$ since
$\psi_K\in P_2(K)\subseteq P_d(K)$.  It is uniformly bounded in
$H^1(K)$ on shape-regular mesh families.  Indeed, the construction
in~\cite[Lemma~3.2]{GopalQiu14} and the Poincar\'e inequality give
$\|v-\Pi_d^{\rm grad}v\|_{L_2(K)}\le C h_K\|\nabla v\|_{L_2(K)}$.
Thus the absolute value of the correction coefficient is at most
$C h_K|K|^{-1/2}\|\nabla v\|_{L_2(K)}$.  Since
$\|\psi_K\|_{L_2(K)}\le C|K|^{1/2}$ and
$\|\nabla\psi_K\|_{L_2(K)}\le C h_K^{-1}|K|^{1/2}$, the correction is
uniformly bounded, just as the original $\Pi_d^{\rm grad}$ operator,
on $H^1(\oh)$.

\begin{proposition}[Verification of Assumption~\ref{asm:fortin-a}]
  \label{prop:fortin-a-verif}
  Let Hypothesis~\ref{hyp:sigma} hold and adopt the discrete
  spaces~\eqref{eq:uw-discrete}.  Suppose that the nonlinearities
  preserve piecewise constants on $\oh$, i.e.,
  \begin{equation}
    \label{eq:P0-preserving}
    \rho_\alpha(q_h, u_h) \in P_0^d
    \quad\text{and}\quad
    \gamma_\alpha(u_h) \in P_0
    \qquad\text{for all } q_h \in P_0^d,\; u_h \in P_0 .
  \end{equation}
  Then Assumption~\ref{asm:fortin-a} holds with $D_h = \Uah$.
\end{proposition}
\begin{proof}
  Set $\fp(r,v):=(\Pi^{\rm div}_2 r,\tilde\Pi_d v)$, with
  $\Pi^{\rm div}_2$ from~\cite[Lemma~3.3]{GopalQiu14} and
  $\tilde\Pi_d$ from~\eqref{eq:fortin-mean-correction}.
  Both components are uniformly bounded in their respective test norms.
  The form $a_\alpha$ in~\eqref{eq:DPGinfiniteBF} differs from its
  linear counterpart only in the terms
  $(\rho_\alpha(q_h,u_h),r-\Pi^{\rm div}_2 r)_{\oh}$ and
  $(\gamma_\alpha(u_h),v-\tilde\Pi_d v)_\om$.
  By~\eqref{eq:P0-preserving}, their coefficients are in $P_0^d$
  and $P_0$, respectively.  The first term vanishes by the volume
  orthogonality of~\cite[Lemma~3.3]{GopalQiu14}, and the second by
  preservation of element averages in~\eqref{eq:fortin-preservation}.
  The remaining terms tested by $r$ vanish by the standard Fortin
  identities of~\cite[Lemma~3.3]{GopalQiu14}.  Those tested by $v$
  vanish by the face-average identity~\eqref{eq:fortin-preservation}:
  $\hat q_{n,h}$ is constant on each face, and integration by parts
  reduces $(q_h,\nabla(v-\tilde\Pi_d v))_K$ to face integrals since
  $q_h$ is constant on $K$.  Hence~\eqref{eq:Fortin} holds for every
  $\clw_h\in\Uah$.
\end{proof}

\begin{theorem}[Variational correctness of the NDPG loss function]
  \label{thm:vc-NDPG-example}
  Let Hypothesis~\ref{hyp:sigma} and~\eqref{eq:P0-preserving} hold,
  and adopt the setting of~\eqref{eq:DPG-A}
  and~\eqref{eq:uw-discrete}.  Then the NDPG loss function
  $\Ja{\cdot}$ of~\eqref{eq:DPG-loss-defn} is {\em globally}
  variationally correct: with $c = \cdB^2/2$ and $C = \CdB^2$ from
  Proposition~\ref{prop:NDPGdiff-verif'},
  the estimates~\eqref{eq:reliability}--\eqref{eq:efficiency} hold for
  every mesh and every $\clw_h \in \Uah$.  If, moreover, $f \in P_0$,
  then $\osc(\ell) = 0$, so that there are mesh- and
  $\alpha$-independent constants $c', C' > 0$ with
  \begin{align}
    \label{eq:ndpg-vc-example}
    c'\,\|\clu - \clw_h\|_{\Ua}^2
    \le \Ja{\clw_h}
    \le C'\,\|\clu - \clw_h\|_{\Ua}^2,
    \qquad \clw_h \in \Uah .
  \end{align}
\end{theorem}
\begin{proof}
  Propositions~\ref{prop:NDPGdiff-verif'}
  and~\ref{prop:fortin-a-verif} verify Assumptions~\ref{asm:NDPGdiff}
  and~\ref{asm:fortin-a}, respectively.  Since $D = \Ua$ and
  $\delta = +\infty$, the result of
  Theorem~\ref{thm:NDPG-gen-nonlinear-dpg} holds globally on all of
  $\Uah$ with $h_0 = +\infty$ and $\tilde D_{h_0} = \Ua$, giving
  \eqref{eq:reliability} and~\eqref{eq:efficiency} with
  $c = c_*^2/2 = \cdB^2/2$ and $C = \CdB^2$.  When $f \in P_0$, the
  element-average identity~\eqref{eq:fortin-preservation} gives
  $(f, v - \tilde\Pi_d v)_{\om} = 0$ for all test functions $v$,
  giving $\osc(\ell) = 0$; absorbing the mesh-independent factor
  $\|\fp\|_{L(\Va)}^2$ of~\eqref{eq:reliability} into the constant
  yields~\eqref{eq:ndpg-vc-example}.
\end{proof}

We now work out the two example cases announced at the start of this
section, recording in each the constants that Hypothesis~\ref{hyp:sigma}
produces.

\subsection{Examples of parametric nonlinear diffusive fluxes}
\label{ssec:nonlinear-diffus-flux-model}

We now single out a class of nonlinear flux profiles that is standard in the literature~\cite{carstensen2017, Zeidl90}.  Let the parameter
enter only as a multiplicative material coefficient, so that
\begin{subequations}
  \label{eq:phi-asm}
  \begin{equation}
    \label{eq:10}
    \phi_\alpha := \alpha^{-1}\phi
  \end{equation}
  for one fixed positive $\phi \in C^1(0,\infty)$ that does not depend
  on $\alpha$, and set
\begin{equation}
  \label{eq:5}
  \sigma_\alpha(p, s) := -\phi_\alpha( |p|)\, p,
  \qquad
  \gamma_\alpha := 0.
\end{equation}
Two quantitative restrictions accompany this.
  The profile is asked to satisfy
  \begin{equation}
    \label{eq:phi-props}
    0 < \beta_1 \leq \phi(t) \leq \beta_2,
    \quad \beta_1 \leq \phi(t) + t \phi'(t) \leq \beta_2,
    \qquad \text{ for all } t>0,
  \end{equation}
  for some constants $0 < \beta_1 < \beta_2 < \infty$, the second pair
  of inequalities controlling the derivative of $t \mapsto t\phi(t)$.
  The parametric  coefficient, in turn, is asked to stay away from zero and
  infinity, i.e., $\cl Q$ consists of bounded positive functions on $\om$ satisfying
  \begin{equation}
    \label{eq:15}
    \alpha_1 \le \alpha(x) \le \alpha_2
    \qquad \text{ for a.e.~} x \in \om,
  \end{equation}
  uniformly over $\alpha \in \cQ$, for some constants
  $0 < \alpha_1 \le \alpha_2 < \infty$.
\end{subequations}
The resulting  model nonlinear diffusion problems~\cite{carstensen2017, Zeidl90} on
$\om$ arise from minimizers of
nonlinear ``energy functionals'' of the form
\begin{equation}
  \label{eq:min-with-varphi}
  u = \argmin_{w \in \Ho^1} \int_\om \varphi_\alpha(|\nabla w|) - f w
\end{equation}
where $f \in L_2$, and
$
  \varphi_\alpha(t) = \int_{0}^{t} s \,\phi_\alpha(s) \; ds.
$
Then, since $\varphi_\alpha'(t) = t \phi_\alpha(t)$,
the Euler--Lagrange equations
of~\eqref{eq:min-with-varphi}
give~\eqref{eq:weak-quasilinear}, while the
boundary value problem~\eqref{eq:4} reduces to
$\dive(\sigma_\alpha(\nabla u)) = f$.  In this instance the inverse
flux $\rho_\alpha(q,s)$ of Lemma~\ref{lem:rho-exists} is likewise
independent of $s$, and we write it simply as
$\rho_\alpha = \sigma_\alpha^{-1}$.

\begin{example}[Invertible profile]
  \label{ex:invertible}
  Set $\phi(t) = 2 + (1+t)^{-1}$ for $t \ge 0$.  Then
  $\phi'(t) = -(1+t)^{-2}$ and $\phi(t) + t\phi'(t) = 2 + (1+t)^{-2}$,
  so~\eqref{eq:phi-props} holds with $\beta_1 = 2$ and $\beta_2 = 3$.
  This profile has the additional property that $\sigma_\alpha$ is
  explicitly invertible.  Indeed, the monotone function
  $m(t) := t\phi(t) = 2t + 1 - (1+t)^{-1}$ satisfies $m(t) = s$ if and
  only if $2t^2 + (3-s)\,t - s = 0$, whence
  $m^{-1}(s) = \bigl( s - 3 + \sqrt{s^2 + 2s + 9}\,\bigr)/4$ for
  $s \ge 0$, on taking the positive root.  Since
  $q = \sigma_\alpha(p)$ gives $|q| = \alpha^{-1} m(|p|)$, a reversal
  of direction yields the explicit inverse
  \begin{equation}
    \label{eq:rho-explicit}
    \rho_\alpha(q) = \sigma_\alpha^{-1}(q)
    = -\frac{m^{-1}(\alpha|q|)}{|q|}\, q
    = -\frac{\alpha|q| - 3 + \sqrt{\alpha^2|q|^2 + 2\alpha|q| + 9}}{4|q|}\, q.
  \end{equation}
  The explicit inverse makes
  both NFOSLS and  NDPG losses
  computable for this PDE.
\end{example}

\begin{remark}
  \label{rem:p-Laplacian}
  An example of a nonlinearity that is not covered under the conditions of~\eqref{eq:phi-asm} is
   the  $p$-Laplacian ($p\ne 2$) studied recently in~\cite{MugaPererRojasRuiz26}. For this example, in the sign convention of~\eqref{eq:5}, $\sigma(p) = -|p|^{p-2} p$, so $\phi(t) = t^{p-2}$.
 For $p>2,$ the function  $\phi(t) \to 0$ as $t\to 0$, so no $\beta_1$ can exist.
 For $1<p<2,$ the function $\phi(t) \to \infty$ as $t\to 0$, so no $\beta_2$ as in~\eqref{eq:phi-asm} can exist.
\end{remark}

The properties of $\sigma_\alpha$ needed below are the following.
By~\cite[Lemma~4.4]{MDPGpaper}, which adapts arguments
of~\cite{carstensen2017}, the flux~\eqref{eq:5} under~\eqref{eq:phi-asm}
is $C^1$ on $\bb R^d$, with a derivative that extends continuously to
$p = 0$, and $-d\sigma_\alpha(p)$ is symmetric for every $p$.
Applying that lemma with $\alpha \equiv \alpha(x)$ and
using~\eqref{eq:15}, we find that, for a.e.~$x \in \om$,
\begin{gather}
  \label{eq:globalized-sigma}
  \frac{\beta_1}{\alpha_2} |\xi|^2
  \;\le\; - \xi \cdot d\sigma_\alpha(x, p)\, \xi
  \;\le\; \frac{\beta_2}{\alpha_1} |\xi|^2,
  \qquad p, \xi \in \bb R^d,
  \\
  \label{eq:globalized-rho}
  |\rho_\alpha(x, q) - \rho_\alpha(x, q')|
  \;\le\; \frac{\alpha_2}{\beta_1} |q - q'|,
  \qquad q, q' \in \bb R^d .
\end{gather}
The first of these is the strongly monotone form of the nonlinearity
in $-\sigma_\alpha$, which is what drives every verification of
Assumption~\ref{asm:NDPGdiff}(c) for this model class; the second is
the Lipschitz continuity of the inverse flux, used in
Remark~\ref{rem:no-frechet} below.  Since $\sigma_\alpha$ is
independent of its second argument while $\gamma_\alpha$ vanishes,
\eqref{eq:globalized-sigma} is precisely~\eqref{eq:Q1}, and we
conclude:

\begin{corollary}
  \label{cor:flux-instance}
  Under~\eqref{eq:phi-asm}, the choice~\eqref{eq:5} satisfies
  Hypothesis~\ref{hyp:sigma} with
  \[
    \mu_1 = \frac{\beta_1}{\alpha_2},
    \qquad
    \mu_2 = \frac{\beta_2}{\alpha_1},
    \qquad
    \lambda = C_\gamma = C_\gamma^- = 0,
    \qquad
    c_0 = \frac{\beta_1}{\alpha_2}.
  \]
  In particular the smallness condition~\eqref{eq:Q4} holds
  automatically, for every parameter range.
  Hence the NFOSLS loss is variationally correct, i.e.,
  \eqref{eq:nfosls-vc-example} holds with
  \begin{equation}
    \label{eq:Cd-c*-constants}
    \CdB = \sqrt 2 \max\Bigl( 1, \frac{\beta_2}{\alpha_1}\Bigr),
    \qquad
    c_*^{-2} = \Bigl( 1 + \frac 1 {c_P}\Bigr)
    \left[ \Bigl( 1 + \frac{\alpha_2\beta_2}{\alpha_1\beta_1}\Bigr)^2
      + c_P + \Bigl( \frac{\alpha_2}{\beta_1}\Bigr)^2 \right],
  \end{equation}
  and the NDPG loss is variationally correct when $\alpha$ and $f$ are in $P_0$.
\end{corollary}

\begin{remark}[$\Aa$ is not \Frechet\ differentiable unless
  $\rho_\alpha$ is affine]
  \label{rem:no-frechet}
  The present setting, in which $\rho_\alpha = \sigma_\alpha^{-1}$ is a
  function of a single vector argument, is a convenient one in which
  to record why the G\^ateaux differentiability required in
  Assumption~\ref{asm:NDPGdiff}{\rm(b)} was not strengthened to
  \Frechet\ differentiability: the latter simply fails.  Suppose
  $\rho_\alpha$ is not affine.  Then there are a $V \in \bb R^d$, a
  $\eta>0$, and a direction vector $e \in \bb R^d$ with $|e|=1$,
  such that the vector
  $ R = \rho_\alpha(V + \eta e) - \rho_\alpha( V) -
  d\rho_\alpha(V)(\eta e) \ne 0 $. Let $\varpi_V$ be a smooth flux such
  that $\varpi_V \equiv V$ in a ball $B$ contained in an element $K$ of
  $\oh$.  For small $\veps>0$, let $B_\veps \subset B$ be the
  concentric ball of radius $\veps$.  Let $\chi_\veps$ be a smooth
  scalar function with $0 \le \chi_\veps \le 1$, $\chi_\veps \equiv 1$
  on $B_\veps$, and support contained in $B_{(1+\theta)\veps}$, for a
  $\theta \in (0,1)$ chosen below independently of $\veps$.  Consider
  trial functions $\clw =(\varpi_V, w, \hat w, \hat\varpi_n)$ and
  $\clz_\veps = (\eta e \chi_\veps, 0, 0, 0)$ in $\Ua$ and test
  function $\clv = (r, 0)$ with $r \in \Hdivoh$ supported in
  $B_{2\veps} \subset K$ satisfying $\div r = 0$, $r \equiv R$ on
  $B_{(1+\theta)\veps}$, and $|r| \lesssim |R|$ throughout (e.g., make
  $r$ by extending $R$ as the flux of the solution of a Neumann
  problem on $B_{2\veps}\setminus B_{(1+\theta)\veps}$).
  By~\eqref{eq:DPGinfiniteBF}, \eqref{eq:dA}, and that
  $\chi_\veps \equiv 1$ and $r \equiv R$ on $B_\veps$,
  \begin{gather*}
    \nonumber
    \ip{\Aa(\clw + \clz_\veps) - \Aa(\clw)
      - d\Aa(\clw)(\clz_\veps), \clv}
    \\
    = \big(\rho_\alpha(\varpi_V + \eta e \chi_\veps)
    - \rho_\alpha(\varpi_V)
    - d\rho_\alpha(\varpi_V)(\eta e \chi_\veps), r\big)_{\oh}
    \\
    = (g, r)_{L_2(B_{(1+\theta)\veps})}
    =  |R|^2 |B_\veps| + E_\veps,
  \end{gather*}
  because the function
  $g = \rho_\alpha(\varpi_V + \eta e \chi_\veps)
  - \rho_\alpha(\varpi_V) -
  d\rho_\alpha(\varpi_V)(\eta e \chi_\veps)$ vanishes off
  $\supp \chi_\veps$ and equals $R$ in $B_\veps$. Here $E_\veps$
  equals the integral of $g \cdot r$ over the shell
  $B_{(1+\theta)\veps} \setminus B_\veps$, whose volume is
  $\eqsim \theta \veps^d$.  There $|r| = |R|$ and
  $|g| \le 2\alpha_2\beta_1^{-1}\eta$ by~\eqref{eq:globalized-rho}, so
  $|E_\veps| \lesssim \theta\, \eta\, |R| \veps^d$.  As
  $|B_\veps| \eqsim \veps^d$, choosing $\theta$ small enough in terms
  of the fixed ratio $|R|/\eta$ (independently of $\veps$) gives
  $|R|^2 |B_\veps| + E_\veps\gtrsim \veps^d |R|^2$.  Since
  $\| \clv\|_{\Va} = \| r \|_{L_2} \lesssim |R| \veps^{d/2}$ and
  $\| \clz_\veps\|_{\Ua} \lesssim \eta \veps^{d/2}$,
    \[
    \frac{\|\Aa(\clw + \clz_\veps) - \Aa(\clw)
      - d\Aa(\clw)(\clz_\veps)\|_{\Vap}}
    {\|\clz_\veps\|_{\Ua}}
    \;\gtrsim\;
    \frac{\veps^{d}\,|R|^2 }{ |R| \veps^{d/2} }
    \frac{1}{\eta\, \veps^{d/2}}
    \;=\; \frac{|R|}{\eta},
  \]
  which does not vanish as $\veps \to 0$, even though
  $\|\clz_\veps\|_{\Ua} \to 0$.  Hence $\Aa$ is not \Frechet\
  differentiable, even though it is G\^ateaux differentiable as proved
  in~(b): indeed, the direction $\clz_\veps/\|\clz_\veps\|_{\Ua}$
  moves in $\Ua$ with $\veps$, whereas along any fixed direction, the
  derivative exists. This lack of \Frechet\ differentiability reflects
  a general obstruction: a Nemytskii operator generated by a function
  $f$ is \Frechet\ differentiable from $L_p$ into $L_p$ only if $f$ is
  affine~\cite{AppellZabrejko90}. Also recall that a G\^ateaux
  differentiable map whose derivative is continuous in the operator
  norm is automatically \Frechet\ differentiable. Therefore our $d\Aa$
  cannot be operator-norm continuous, which is the reason for the
  caveat on~{\rm(c$'$)} raised in Remark~\ref{rem:sufficient-c}.
\end{remark}

\subsection{Semilinear parametric examples}
\label{sec:semilinear}
At the opposite extreme, the leading-order operator may be linear,
with all the nonlinearity confined to additive lower-order terms.
This is the model problem of~\cite{VeraHeuerPraet2026},
\begin{equation}
  \label{eq:semilin-classical}
  -\dive \bigl( \kappa_\alpha \nabla u + g_\alpha(u)\, b_\alpha \bigr)
  + \gamma_\alpha(u) = f \quad \text{ in } \om,
  \qquad u = 0 \quad \text{ on } \partial\om,
\end{equation}
which we treat here to exhibit the reach of
Theorem~\ref{thm:NDPG-gen-nonlinear-dpg} and to permit
comparison.  Here the
parameter $\alpha$ in some compact set $\cQ$ determines a diffusion tensor
$\kappa_\alpha$, a constant advection vector $b_\alpha$, and two
scalar nonlinearities $g_\alpha, \gamma_\alpha$
(say as in Example~\ref{ex:semilin-vhp} below).  Problem~\eqref{eq:semilin-classical} is~\eqref{eq:4} with
\begin{equation}
  \label{eq:semilin-sigma}
  \sigma_\alpha(p,s) := - \kappa_\alpha\, p - g_\alpha(s)\, b_\alpha,
\end{equation}
and the corresponding inverse flux is available in closed form,
$\rho_\alpha(q,s) = -\kappa_\alpha^{-1}( q + g_\alpha(s) b_\alpha)$.

\begin{corollary}
  \label{cor:semilin-instance}
  Suppose that, for every $\alpha \in \cQ$, the tensor
  $\kappa_\alpha \in L_\infty(\om;\bb R^{d\times d})$ is symmetric with
  spectrum contained in $[\kappa_1,\kappa_2]$ for a.e.~$x \in \om$,
  the advection vector $b_\alpha \in \bb R^d$ is constant with
  $|b_\alpha| \le b_*$, and $g_\alpha, \gamma_\alpha \in C^1(\bb R)$
  satisfy $\sup |g_\alpha'| \le C_g$ and
  $\sup|\gamma_\alpha'| \le C_\gamma$, with
  $\gamma_\alpha' \ge -C_\gamma^-$.  Then~\eqref{eq:semilin-sigma}
  satisfies~\eqref{eq:Q1}--\eqref{eq:Q3} of Hypothesis~\ref{hyp:sigma} with
  \[
    \mu_1 = \kappa_1,
    \qquad
    \mu_2 = \kappa_2,
    \qquad
    \lambda = C_g\, b_*,
  \]
  and $C_\gamma, C_\gamma^-$ as given, so that whenever the smallness
  condition~\eqref{eq:Q4}
  \begin{equation}
    \label{eq:semilin-small}
    c_0 = \kappa_1
    - \frac{C_g\, b_*}{\sqrt{c_P}}
    - \frac{C_\gamma^-}{c_P}
    \; > \; 0
  \end{equation}
  holds, the NFOSLS loss is variationally correct, i.e.,
  \eqref{eq:nfosls-vc-example} holds with
  \begin{equation}
    \label{eq:Cd-c*-semilin}
    \begin{aligned}
      \CdB^2 & = 2 \max\Bigl( 1, \;
      \Bigl( \kappa_2 + \frac{C_g\, b_*}{c_P^{1/2}} \Bigr)^2
      + \frac{C_\gamma^2}{c_P} \Bigr),
      \\
      c_*^{-2} & = \Bigl( 1 + \frac 1 {c_P}\Bigr)
      \left[
        \Bigl( 1 + \frac{\kappa_2 + C_g\, b_*\, c_P^{-1/2}}{c_0}\Bigr)^2
        + \Bigl( c_P^{1/2} + \frac{C_\gamma}{c_P^{1/2}\, c_0}\Bigr)^2
        + \frac{1}{c_0^2}
      \right],
    \end{aligned}
  \end{equation}
  with $c_0$ as in~\eqref{eq:semilin-small},
    and the NDPG loss is variationally correct when $\kappa_\alpha$ and $f$ are constant on each element of $\oh$.
\end{corollary}
\begin{proof}
  This is immediate from $-\partial_p\sigma_\alpha = \kappa_\alpha$ and
  $\partial_s\sigma_\alpha = -g_\alpha'(s)\, b_\alpha$.
\end{proof}

\begin{example}
  \label{ex:semilin-vhp}
  Let $\om = (0,1)^2$, subdivided into $m \times m$ congruent squares
  $\om_1, \dots, \om_{m^2}$, and let
  \[
    \cQ = \Big\{ \alpha = (\alpha^{(1)}, \dots, \alpha^{(m^2)}) :
    \alpha^{(i)} \in [\kappa_1, \kappa_2] \Big\},
    \qquad
    \kappa_\alpha(x) = \Big( \sum_{i=1}^{m^2} \alpha^{(i)} \chi_i(x) \Big) I,
  \]
  with $\chi_i$ the indicator function of $\om_i$ as
  in~\eqref{eq:Qdefn}.  In~\eqref{eq:semilin-classical}, fix
  $b_\alpha = (1,2)^T$, $g_\alpha(t) = \cos t$, and
  $\gamma_\alpha(t) = \arctan t$,
  independently of $\alpha$,
  to get the  data of the first
  numerical example of~\cite{VeraHeuerPraet2026}.  Then $C_g = 1$,
  $b_* = \sqrt 5$, and $\gamma_\alpha$ is increasing, so
  $C_\gamma^- = 0$ and $C_\gamma = 1$.  Since $c_P = 2\pi^2$ on the
  unit square, \eqref{eq:semilin-small} reads
  $\kappa_1 > \sqrt{5}/(\pi \sqrt 2)$, giving the
  coefficient range
  $[\kappa_1, \kappa_2]$ where
  Hypothesis~\ref{hyp:sigma} holds.
\end{example}

\subsection{Discussion}
\label{ssec:quasilin-discussion}

The contrast between Theorems~\ref{thm:vc-NDPG-example}
and~\ref{thm:vc-NFOSLS-example} summarizes the trade-offs.  The NFOSLS result places no
restriction on $\alpha$, on the mesh, or even on the existence of a
mesh, so it certifies the error of an arbitrary
$\Hdiv \times \Ho^1$-conforming neural network output.  The
NDPG bound
of Theorem~\ref{thm:vc-NDPG-example}
holds in the weaker ultraweak norm on the interior variables,
but its Fortin argument,
through~\eqref{eq:P0-preserving},
ties it to a mesh, to lowest-order
trial functions, and to a piecewise constant parameter.
Condition~\eqref{eq:P0-preserving} is easy to check in the two
instances of \S\ref{ssec:nonlinear-diffus-flux-model}
and~\S\ref{sec:semilinear}.  For the nonlinear flux model,
$\rho_\alpha$ is independent of $u_h$ and $\gamma_\alpha = 0$, and
$\rho_\alpha(q_h) \in P_0^d$ holds whenever $\alpha$ is piecewise
constant on $\oh$: by~\eqref{eq:rho-explicit}, or more generally since
$\rho_\alpha(x,q)$ depends on $x$ only through $\alpha(x)$, the value of
$\rho_\alpha(q_h)$ on each element is determined by the constant values
of $\alpha$ and $q_h$ there.  For the semilinear model,
$\rho_\alpha(q_h,u_h) = -\kappa_\alpha^{-1}(q_h + g_\alpha(u_h)
b_\alpha)$ lies in $P_0^d$ and $\gamma_\alpha(u_h) \in P_0$ whenever
$\kappa_\alpha$ is piecewise constant on $\oh$ and $g_\alpha,
\gamma_\alpha$ do not depend on $x$ (or are piecewise constant in $x$
on $\oh$).  The NFOSLS instance of~\S\ref{ssec:nfosls-diffusion}
carries no such restriction.

It is instructive to set our treatment beside the scheme
of~\cite{VeraHeuerPraet2026}, from which the semilinear model
problem~\eqref{eq:semilin-classical} is taken.  Their scheme adjoins to
the trial space two auxiliary $L_2$ unknowns standing for
$g_\alpha(u)$ and $\gamma_\alpha(u)$, which renders the remaining
equations linear and open to the standard linear DPG treatment, and
appends the two squared $L_2$ mismatches to the DPG residual as
least-squares penalties.  Four differences are worth recording.
First, we have no auxiliary variables,  the nonlinearity
stays inside $\Aa$, and the trial space in~\eqref{eq:DPG-A} is the
one the linear problem would use.
Second, the estimates differ in kind.  The main result
of~\cite{VeraHeuerPraet2026} is a C\'ea-type quasi-optimality
estimate for the {\em discrete minimizer}, conditional on the
existence of an exact solution and on an
Aubin--Nitsche-type closeness hypothesis.  Theorems~\ref{thm:vc-NFOSLS-example}
and~\ref{thm:vc-NDPG-example}, by contrast, do not require these conditions,
and provide  a posteriori error estimators (from which quasi-optimality of
the discrete minimizer follows as in Remark~\ref{rem:quasioptimality}) which
 bound the error of an
{\em arbitrary} member of the trial class from above and below by
the computable loss.  (Their residual is likewise local and drives
adaptivity, but no two-sided bound relating it to the error is proved
there.)  The latter feature is important  for our NN applications,
since a loss function must certify error
at predictions beyond the exact minimizer.
Third, we pay for this with the smallness
condition~\eqref{eq:semilin-small}, which the C\'ea estimate
of~\cite{VeraHeuerPraet2026} does not require.  Fourth, \cite{VeraHeuerPraet2026}~uses
a primal DPG formulation, whereas we have developed the ultraweak one.

\section{The stationary Kirchhoff model}
\label{sec:kirchhoff-ndpg}
In this section, we
verify Assumptions~\ref{asm:NDPGdiff} and~\ref{asm:fortin-a} for a
class of nonlinearities qualitatively different from the quasilinear
one of \S\ref{sec:nonlinear-diffusion}, a stationary {\em Kirchhoff
  model}, in which the diffusion coefficient is modulated by a factor
that depends on the solution $u$ only through a single global scalar,
its Dirichlet energy.  Whereas the quasilinear fluxes of \S\ref{sec:nonlinear-diffusion} act
pointwise, the coefficient here is {\em
  nonlocal}.  As we shall see,
this provides an
application where
$\Aa$ is \Frechet\ differentiable.
The main results of this section, Theorems~\ref{thm:vc-NFOSLS-Kirchhoff}
and \ref{thm:vc-NDPG-Kirchhoff},
give globally reliable and efficient error estimators
of the FOSLS type and the primal  DPG type
for the Kirchhoff model.

\subsection{The parametric Kirchhoff boundary value problem}
\label{ssec:kirchhoff-model}

Each $\alpha$ in the parameter set $\cQ$ determines two objects here: a
diffusion coefficient $\kappa_\alpha$, exactly as in the parametric
linear diffusion of \S\ref{sec:semilinear}, and a Kirchhoff function
$M_\alpha : [0,\infty) \to \bb R$.  Writing
\begin{equation}
  \label{eq:KH-energy}
  G_\alpha(w) := (\kappa_\alpha \nabla w, \nabla w)_\om,
  \qquad w \in \Ho^1,
\end{equation}
for the $\kappa_\alpha$-weighted Dirichlet energy, the Kirchhoff
problem in classical form reads
\begin{equation}
  \label{eq:KH-classical}
  -\dive\bigl( M_\alpha(G_\alpha(u))\, \kappa_\alpha \nabla u \bigr) = f
  \quad\text{ in } \om,
  \qquad
  u = 0 \quad\text{ on } \partial\om,
\end{equation}
i.e., the flux is $M_\alpha(G_\alpha(u))\,\kappa_\alpha\nabla u$: an
ordinary parametric diffusion coefficient $\kappa_\alpha$, modulated by
the scalar $M_\alpha(G_\alpha(u))$, which is constant in space but
depends on $u$ through the number $G_\alpha(u)$.
Equations of this type originate from Kirchhoff's
model~\cite{Kirchhoff1876} for the transverse vibration of an elastic
string, in which the axial tension is augmented in proportion to the
change in length of the deformed string (a quantity determined by
integrating the first derivative of $u$ over the whole string, not by
pointwise values of the
strain).
The modern functional-analytic treatment of these {\em nonlocal}
problems was initiated by~\cite{Lions1978} and has since
generated an extensive literature on equations of Kirchhoff type.
Finite element approximation of the stationary
problem~\eqref{eq:KH-classical} with $\kappa_\alpha = I$ was analyzed
in~\cite{Gudi12}, and a residual-based a posteriori estimator driving a
convergent adaptive algorithm was obtained in~\cite{PalGudi25}.  The
weak formulation of~\eqref{eq:KH-classical} finds
$u \in \Ho^1$ satisfying
\begin{equation}
  \label{eq:KH-weak}
  M_\alpha(G_\alpha(u))\, (\kappa_\alpha \nabla u, \nabla v)_\om = (f, v)_\om
  \qquad \text{ for all } v \in \Ho^1.
\end{equation}
The family $\{ (\kappa_\alpha, M_\alpha) : \alpha \in \cQ \}$ is
assumed to satisfy the
following hypothesis with $\alpha$-independent constants.  In it,
$t_- := \max(-t, 0)$ denotes the negative part of any real number~$t$.

\begin{hypothesis}
  \label{hyp:M}
  There are constants $0 < M_1 \le M_2 < \infty$,
  $C_M^* \ge C_M^- \ge 0$, and $0 < \kappa_1 \le \kappa_2 < \infty$
  such that for every $\alpha \in \cQ$:
  \begin{enumerate}[label={\rm(H\arabic*)}]
  \item $M_\alpha \in C^1([0,\infty))$ and
    $M_1 \le M_\alpha(s) \le M_2$ for all $s \ge 0$;
  \item $\sup_{s \ge 0} s\,|M_\alpha'(s)| \le C_M^*$,
    $\sup_{s \ge 0} s\,\bigl(M_\alpha'(s)\bigr)_- \le C_M^-$, and
    $M_1 > 2 C_M^-$;
  \item for each $\bar s > 0$ there is a $C_M(\bar s) < \infty$,
    independent of $\alpha$, such that $M_\alpha$ and $M_\alpha'$ are
    Lipschitz on $[0, \bar s]$ with Lipschitz constant $C_M(\bar s)$;
  \item $\kappa_\alpha \in L_\infty(\om; \bb R^{d\times d})$ is
    symmetric with spectrum contained in $[\kappa_1, \kappa_2]$ for
    a.e.~$x \in \om$.
  \end{enumerate}
\end{hypothesis}

The  decrease of $M_\alpha$ is restricted quantitatively
in~{\rm(H2)}: the first bound there merely asks that $s M_\alpha'(s)$
be bounded. In fact,  $C_M^- = 0$ whenever $M_\alpha$ is nondecreasing, in
which case the smallness condition $M_1 > 2C_M^-$ is vacuous.  This is
the physically relevant case where the  Kirchhoff tension
increases with the elongation.

\begin{example}
  \label{ex:mu-kirchhoff}
  Let $\kappa_\alpha$ and $\cQ$ be the $m^2$-dimensional
  piecewise-constant family of Example~\ref{ex:semilin-vhp}, i.e.,
  $\kappa_\alpha = \bigl( \sum_{i=1}^{m^2} \alpha^{(i)} \chi_i \bigr) I$
  with $\alpha^{(i)} \in [\kappa_1, \kappa_2]$, so that~{\rm(H4)}
  holds, and let
  \begin{equation}
    \label{eq:carrier}
    M_\alpha(s) = M(s) := a + \frac{b\, s}{1 + \veps s},
    \qquad a, b, \veps > 0,
  \end{equation}
  independently of $\alpha$.  This is the Kirchhoff--Carrier law
  $a + b s$, saturated at large $s$: the two agree to $O(\veps)$ on
  bounded ranges of $s$ and coincide as $\veps \to 0$.  The saturation
  is what buys~{\rm(H1)}, which the unbounded law $a + bs$ violates,
  while the feature that matters physically, that $M$ increases with
  the elongation, is retained.  Indeed
  $M'(s) = b (1 + \veps s)^{-2} > 0$, so $C_M^- = 0$
  and~{\rm(H2)} holds with no smallness condition whatsoever, for
  every $a, b, \veps$.  The remaining constants are $M_1 = a$,
  $M_2 = a + b/\veps$, and, since $s(1+\veps s)^{-2}$ is maximized at
  $s = 1/\veps$, $C_M^* = b/(4\veps)$.  Finally $|M'| \le b$ and
  $|M''| \le 2\veps b$ on all of $[0,\infty)$, so~{\rm(H3)} holds with
  $C_M(\bar s) = \max(b, 2\veps b)$ for every $\bar s$.

  By contrast,
  for a decreasing law such as $M_\alpha(s) = a + (1+s)^{-1}$ one finds
  $C_M^- = C_M^* = 1/4$, and~{\rm(H2)} then genuinely restricts the
  model, to $a > 1/2$.
\end{example}

Since $M_\alpha(G_\alpha(u))$ is a scalar, the nonlinear
Kirchhoff solutions can be parameterized by
solutions of  a linear problem: as the next lemma shows, each is a
rescaling of a $\kappa_\alpha$-diffusion solution by a factor
determined from one scalar equation.  For $\kappa_\alpha = I$, this
device appears in~\cite{ChipoValenVerga03} and
\cite{ChipoSavit14} extends it to the $p$-Laplacian.

\begin{lemma}[Well-posedness of~\eqref{eq:KH-weak}]
  \label{lem:kirchhoff-wellposed}
  Let Hypothesis~\ref{hyp:M} hold and $f \in L_2$.  For
  $\alpha \in \cQ$ let $w_\alpha \in \Ho^1$ be the solution of the
  linear problem
  $(\kappa_\alpha \nabla w_\alpha, \nabla v)_\om = (f, v)_\om$ for all
  $v \in \Ho^1$, and put $g_\alpha := G_\alpha(w_\alpha)$.  Then the
  function $\psi_\alpha(s) := s\, M_\alpha(s)^2$ is a strictly
  increasing bijection of $[0, \infty)$, and letting
  $s^* = s^*(\alpha) \ge 0$ denote the unique root of
  $\psi_\alpha(s^*) = g_\alpha$, the function
  \begin{equation}
    \label{eq:KH-solution}
    u(\alpha) := M_\alpha(s^*)^{-1}\, w_\alpha
  \end{equation}
  is the unique solution of~\eqref{eq:KH-weak}.  Moreover, the exact
  flux is
  $M_\alpha(G_\alpha(u(\alpha)))\, \kappa_\alpha \nabla u(\alpha) =
  \kappa_\alpha \nabla w_\alpha$, and
  \begin{equation}
    \label{eq:KH-apriori}
    G_\alpha(u(\alpha)) = s^*(\alpha) \le \frac{g_\alpha}{M_1^2},
    \qquad
    g_\alpha \le \frac{\|f\|^2_{L_2}}{c_P\, \kappa_1} =: g_*,
  \end{equation}
  so that $s^*$ and $g_\alpha$ are bounded independently of $\alpha$.
\end{lemma}
\begin{proof}
  The function $\psi_\alpha$ is continuous, $\psi_\alpha(0) = 0$, and
  $\psi_\alpha(s) \ge M_1^2\, s \to \infty$, so $\psi_\alpha$ maps
  $[0,\infty)$ onto itself.  By Hypothesis~\ref{hyp:M}, using
  $s M_\alpha'(s) \ge - s (M_\alpha'(s))_- \ge -C_M^-$,
  \[
    \psi_\alpha'(s)
    = M_\alpha(s) \bigl( M_\alpha(s) + 2 s M_\alpha'(s) \bigr)
    \ge M_1 (M_1 - 2 C_M^-) > 0,
  \]
  so $\psi_\alpha$ is strictly increasing and the root $s^*$ exists
  uniquely.  For $u(\alpha)$ defined by~\eqref{eq:KH-solution},
  $G_\alpha(u(\alpha))
  = (\kappa_\alpha \nabla u(\alpha), \nabla u(\alpha))_\om
  = M_\alpha(s^*)^{-2} g_\alpha =
  M_\alpha(s^*)^{-2} \psi_\alpha(s^*) = s^*$, so
  \[
    M_\alpha(G_\alpha(u(\alpha)))\,
    (\kappa_\alpha \nabla u(\alpha), \nabla v)_\om
    = M_\alpha(s^*)\, M_\alpha(s^*)^{-1}
    (\kappa_\alpha \nabla w_\alpha, \nabla v)_\om
    = (f, v)_\om,
  \]
  i.e., $u(\alpha)$ solves~\eqref{eq:KH-weak}, and
  $M_\alpha(G_\alpha(u(\alpha)))\kappa_\alpha \nabla u(\alpha) =
  \kappa_\alpha \nabla w_\alpha$.  If $\tilde u$ is any solution
  of~\eqref{eq:KH-weak}, then, since
  $\tilde m := M_\alpha(G_\alpha(\tilde u))$ is positive, the
  uniqueness of $w_\alpha$ gives $\tilde u = \tilde m^{-1} w_\alpha$,
  so
  $\psi_\alpha(G_\alpha(\tilde u)) = \tilde m^2\, G_\alpha(\tilde u) =
  g_\alpha$, and the injectivity of $\psi_\alpha$ yields
  $G_\alpha(\tilde u) = s^*$, hence $\tilde u = u(\alpha)$.  Finally,
  $\psi_\alpha(s) \ge M_1^2 s$ gives $s^* \le g_\alpha/M_1^2$, while
  testing the equation for $w_\alpha$ with $v = w_\alpha$ and
  using~\eqref{eq:Friedrichs-std} and~{\rm(H4)},
  \[
    g_\alpha
    = (f, w_\alpha)_\om
    \le c_P^{-1/2} \|f\|_{L_2} \|\nabla w_\alpha\|_{L_2}
    \le c_P^{-1/2} \|f\|_{L_2} \bigl( g_\alpha/\kappa_1 \bigr)^{1/2},
  \]
  which is the second bound in~\eqref{eq:KH-apriori}.
\end{proof}

The formula~\eqref{eq:KH-solution} exhibits the exact
parameter-to-solution map as the composition of the linear parametric
map $\alpha \mapsto w_\alpha$ with a rescaling by the scalar
$M_\alpha(s^*(\alpha))^{-1}$, the latter determined by the single
nonlinear equation $\psi_\alpha(s^*) = G_\alpha(w_\alpha)$.  With
$\kappa_\alpha$ as in Example~\ref{ex:mu-kirchhoff}, this is a
genuinely $m^2$-dimensional parametric family, whose flux
$\kappa_\alpha \nabla w_\alpha$ varies with $\alpha$ as well.  At the
same time, the availability of $u(\alpha)$ in closed form, up to one
linear solve and one scalar root-finding procedure, makes the model a transparent
test problem for the learning application of the next section.

The upcoming error estimators
are analyzed through one and the same
linearization, which we record now.  For each $w \in \Ho^1$, consider the linear map
$B_w : \Ho^1 \to L_2^d$ defined by
\begin{equation}
  \label{eq:KH-frozen-form}
  B_w \zeta
  := M_\alpha(G_\alpha(w))\, \kappa_\alpha \nabla \zeta
  + 2 M_\alpha'(G_\alpha(w))\,
  (\kappa_\alpha \nabla w, \nabla \zeta)_\om\, \kappa_\alpha \nabla w ,
\end{equation}
which is nonlocal, since $\zeta$ enters the second term  through
the number $(\kappa_\alpha \nabla w, \nabla \zeta)_\om$.

\begin{lemma}[Linearization]
  \label{lem:kirchhoff-frozen}
  Let Hypothesis~\ref{hyp:M} hold and set
  \begin{equation}
    \label{eq:KH-cK}
    c_K := (M_1 - 2 C_M^-)\, \kappa_1 > 0,
    \qquad
    C_K := (M_2 + 2 C_M^*)\, \kappa_2 .
  \end{equation}
  Then, for every $w, \zeta \in \Ho^1$ and every $\alpha \in \cQ$,
  \begin{equation}
    \label{eq:KH-coercive}
    \| B_w \zeta \|_{L_2}
    \le C_K\, \| \nabla \zeta \|_{L_2},
    \qquad
    ( B_w \zeta, \nabla \zeta)_\om
    \ge c_K\, \| \nabla \zeta \|^2_{L_2} .
  \end{equation}
  Both constants are independent of $w$ and of $\alpha$.  Consequently
  the same two bounds hold for the averaged map
  $\int_0^1 B_{w^t}\, dt$ generated by an arbitrary family
  $\{ w^t : t \in [0,1] \} \subset \Ho^1$.
\end{lemma}
\begin{proof}
  Since $\kappa_\alpha$ is symmetric and positive definite by
  {\rm(H4)}, the form $(\kappa_\alpha \nabla \cdot, \nabla\cdot)_\om$
  is an inner product on $\Ho^1$, so
  \begin{equation}
    \label{eq:KH-CS}
    | (\kappa_\alpha \nabla w, \nabla \zeta)_\om |
    \le G_\alpha(w)^{1/2} G_\alpha(\zeta)^{1/2},
    \qquad
    \| \kappa_\alpha \nabla w \|_{L_2}^2
    \le \kappa_2\, G_\alpha(w).
  \end{equation}
  Hence the second term
  of~\eqref{eq:KH-frozen-form} has $L_2$ norm at most
  $2 |M_\alpha'(G_\alpha(w))|\, G_\alpha(w)\, \kappa_2^{1/2}
  G_\alpha(\zeta)^{1/2} \le 2 C_M^* \kappa_2 \|\nabla\zeta\|_{L_2}$
  by {\rm(H2)} and {\rm(H4)}, while the first is at most
  $M_2 \kappa_2^{1/2} G_\alpha(\zeta)^{1/2} \le M_2 \kappa_2
  \|\nabla\zeta\|_{L_2}$ by {\rm(H1)}.  This gives the first bound
  of~\eqref{eq:KH-coercive}.
  For the second, \eqref{eq:KH-CS} gives
  \begin{align*}
    ( B_w \zeta, \nabla \zeta)_\om
    & = M_\alpha(G_\alpha(w))\, G_\alpha(\zeta)
      + 2 M_\alpha'(G_\alpha(w))
      (\kappa_\alpha \nabla w, \nabla\zeta)^2_\om
    \\
    & \ge \bigl[ M_\alpha(G_\alpha(w))
      - 2\bigl(M_\alpha'(G_\alpha(w))\bigr)_- G_\alpha(w) \bigr]
      G_\alpha(\zeta).
  \end{align*}
  By {\rm(H1)},
  {\rm(H2)} and {\rm(H4)}, the right-hand side is at least
  $(M_1 - 2C_M^-) \kappa_1 \|\nabla\zeta\|^2_{L_2}$, which is the
  second bound  of~\eqref{eq:KH-coercive}.
  Since both hold pointwise in $t$ with
  $w$-independent constants, they also hold for the average.
\end{proof}

\subsection{The NFOSLS loss function}
\label{ssec:kirchhoff-nfosls}

We take the simpler least-squares instance first.
Introducing the flux $q := M_\alpha(G_\alpha(u))\kappa_\alpha\nabla u$,
the Kirchhoff problem~\eqref{eq:KH-classical} is the first-order system
$q = M_\alpha(G_\alpha(u))\kappa_\alpha\nabla u$, $-\dive q = f$.  (Here,
unlike in~\S\ref{sec:nonlinear-diffusion}, where $\sigma_\alpha$ is the
negative of the flux, $q$ denotes the flux itself.)  Accordingly we set
\begin{subequations}
  \label{eq:KH-FOS}
  \begin{align}
    & \Ua = \Hdiv \times \Ho^1,
      \qquad
      \Va = L_2^d \times L_2,
    \\
    \label{eq:KH-fos-A}
    & \Aa(q,u) = \bigl( q - M_\alpha(G_\alpha(u))\, \kappa_\alpha \nabla u, \; -\dive q \bigr),
      \qquad
      \ell = (0,f),
  \end{align}
\end{subequations}
with the graph norm~\eqref{eq:fos-norms} on $\Ua$, with $\Va$
identified with its own dual through the $L_2$ inner product, and
with $D = \Ua$.  Taking $\Vah = \Va$ and $\fp = \mathrm{id}$ in
Definition~\ref{defn:ndpg}, the resulting NFOSLS loss function is
\begin{equation}
  \label{eq:KH-fos-loss}
  \Ja{(q,u)}
  = \bigl\| q - M_\alpha(G_\alpha(u))\, \kappa_\alpha \nabla u \bigr\|^2_{L_2}
  + \bigl\| \dive q + f \bigr\|^2_{L_2} .
\end{equation}
As in \S\ref{ssec:nfosls-diffusion}, the trial class $\Uah$ enters
neither~\eqref{eq:KH-fos-loss} nor the analysis, so it may be any
subset of $\Ua$, finite element or meshless.

This is arguably the least expensive loss function we have seen so far.
Since the entire nonlinearity is in the single scalar
$m := M_\alpha(G_\alpha(u))$, evaluation just requires
one global reduction (i.e., a single sum over all mesh elements) to
compute $G_\alpha(u) = (\kappa_\alpha\nabla u, \nabla u)_\om$ followed by one evaluation
of $M_\alpha$.  With $m$ in hand, \eqref{eq:KH-fos-loss} is a
quadratic functional of $(q,u)$, so no special quadrature is needed
for nonpolynomial constitutive terms (in contrast to the
quasilinear NFOSLS loss~\eqref{eq:fos-loss}, where
$\sigma_\alpha(\nabla w, w)$ must be evaluated at every quadrature
point).  Differentiation with respect to network parameters brings in,
beyond linear terms, only the scalar $M_\alpha'(G_\alpha(u))$.

\begin{proposition}[Verification of Assumption~\ref{asm:NDPGdiff} for
  the Kirchhoff NFOSLS setting]
  \label{prop:kirchhoff-nfosls}
  Let Hypothesis~\ref{hyp:M} hold.  Then the operator $\Aa$
  of~\eqref{eq:KH-fos-A} satisfies all conditions of
  Assumption~\ref{asm:NDPGdiff} with $D = \Ua$ and
  $\delta = +\infty$.  With $c_K$ and $C_K$ as in~\eqref{eq:KH-cK},
  one may take
  \[
    \CdB^2 = 2\max\bigl(1, C_K^2\bigr),
    \qquad
    c_*^{-2} = \Bigl(1 + \frac{1}{c_P}\Bigr)
    \left[ \Bigl( 1 + \frac{C_K}{c_K} \Bigr)^2
      + c_P + \frac{1}{c_K^2} \right].
  \]
\end{proposition}
\begin{proof}
  {\em Proof of} {\rm(a)}.  Let $u = u(\alpha)$ be the solution
  of~\eqref{eq:KH-weak} given by
  Lemma~\ref{lem:kirchhoff-wellposed}.  Hence  the flux
  $q := M_\alpha(G_\alpha(u))\kappa_\alpha\nabla u = \kappa_\alpha\nabla w_\alpha$ satisfies
   $-\dive (\kappa_\alpha \nabla w_\alpha) = f \in L_2$
  distributionally, so $q \in \Hdiv$ and $\clu := (q, u) \in \Ua$
  satisfies $\Aa(\clu) = \ell$, i.e., $\Ra \clu = 0$.

  {\em Proof of} {\rm(b)}.  Since $G$ is a continuous quadratic form
  on $\Ho^1$ and $M_\alpha \in C^1$ by {\rm(H1)}, the map
  $u \mapsto M_\alpha(G_\alpha(u))\kappa_\alpha\nabla u$ is \Frechet\ differentiable from
  $\Ho^1$ into $L_2^d$, with derivative $B_u$
  of~\eqref{eq:KH-frozen-form}.  Hence $\Aa$ is \Frechet\ differentiable, and a
  fortiori G\^ateaux differentiable at every $(q,u) \in \Ua$, with
  \begin{equation}
    \label{eq:KH-fos-dA}
    d\Aa(q,u)(\varsigma,z)
    = \bigl( \varsigma - B_u z, \; -\dive \varsigma \bigr),
    \qquad (\varsigma,z) \in \Ua .
  \end{equation}
  By the first bound of~\eqref{eq:KH-coercive},
  $\|\varsigma - B_u z\|^2_{L_2} \le 2\|\varsigma\|^2_{L_2} + 2 C_K^2
  \|\nabla z\|^2_{L_2}$, so
  $\|d\Aa(q,u)(\varsigma,z)\|^2_{\Vap}
  \le \CdB^2 \|(\varsigma,z)\|^2_{\Ua}$ with
  $\CdB$ as claimed, which is~\eqref{linvc}.

  {\em Proof of} {\rm(c)}.  Let $\clz^i = (\varsigma^i, z^i) \in \Ua$,
  $i = 0,1$, be arbitrary, put $e_\varsigma = \varsigma^0 - \varsigma^1$,
  $e_z = z^0 - z^1$,
  and let $z^t := z^0 - t\, e_z$.  By {\rm(b)} the map
  $t \mapsto M_\alpha(G_\alpha(z^t))\kappa_\alpha \nabla z^t$ is continuously
  differentiable from $[0,1]$ into $L_2^d$, with derivative
  $-B_{z^t} e_z$, so the FTC gives
  \[
    M_\alpha(G_\alpha(z^0))\kappa_\alpha\nabla z^0
    - M_\alpha(G_\alpha(z^1))\kappa_\alpha\nabla z^1
    = \bar B\, e_z,
    \qquad
    \bar B := \int_0^1 B_{z^t}\, dt .
  \]
  Hence $\Aa(\clz^0) - \Aa(\clz^1) = (r_1, r_2)$ with
  $r_1 = e_\varsigma - \bar B e_z$ and $r_2 = -\dive e_\varsigma$; set
  $A_*^2 := \|r_1\|^2_{L_2} + \|r_2\|^2_{L_2}$.  Since
  $e_z \in \Ho^1$ and $e_\varsigma \in \Hdiv$, integration by parts gives
  \[
    (e_\varsigma, \nabla e_z)_\om = -(\dive e_\varsigma, e_z)_\om = (r_2, e_z)_\om,
  \]
  so that
  \[
    ( \bar B e_z, \nabla e_z)_\om
    = (e_\varsigma, \nabla e_z)_\om - (r_1, \nabla e_z)_\om
    = (r_2, e_z)_\om - (r_1, \nabla e_z)_\om
    \le A\, \|\nabla e_z\|_{L_2},
  \]
  where $A := \|r_1\|_{L_2} + c_P^{-1/2}\|r_2\|_{L_2}$
  and~\eqref{eq:Friedrichs-std} was used.  By the second bound
  of~\eqref{eq:KH-coercive}, applied to the average $\bar B$, the
  left-hand side is at least $c_K \|\nabla e_z\|^2_{L_2}$, whence
  $\|\nabla e_z\|_{L_2} \le A / c_K$.  Consequently
  \[
    \|e_\varsigma\|_{L_2}
    = \| r_1 + \bar B e_z \|_{L_2}
    \le \|r_1\|_{L_2} + C_K \|\nabla e_z\|_{L_2},
    \qquad
    \|\dive e_\varsigma\|_{L_2} = \|r_2\|_{L_2},
  \]
  and, since $\|r_1\|_{L_2} \le A$ and
  $\|r_2\|_{L_2} \le \sqrt{c_P}\, A$, adding the three bounds and
  using $A^2 \le (1 + c_P^{-1}) A_*^2$ gives
  $\|\clz^0 - \clz^1\|^2_{\Ua} \le c_*^{-2} A_*^2$.  This
  is~\eqref{eq:c-prime} for all $\clz^1 \in \Ua$, i.e., with
  $\delta = +\infty$; as before, its global validity also shows that
  the $\clu$ of part {\rm(a)} is the unique element of $\Ua$ at which
  $\Ra\cdot$ vanishes.
\end{proof}

\begin{theorem}[Variational correctness of the Kirchhoff NFOSLS loss]
  \label{thm:vc-NFOSLS-Kirchhoff}
  Let Hypothesis~\ref{hyp:M} hold.  Then the loss
  function~\eqref{eq:KH-fos-loss} is globally variationally correct on
  the entire trial space: with $c_*$ and $\CdB$ as in
  Proposition~\ref{prop:kirchhoff-nfosls},
  \begin{equation}
    \label{eq:KH-nfosls-vc}
    \frac{c_*^2}{2}\, \| \clu - \clz \|^2_{\Ua}
    \;\le\; \Ja \clz
    \;\le\; \CdB^2\, \| \clu - \clz \|^2_{\Ua},
    \qquad \clz \in \Ua = \Hdiv \times \Ho^1 .
  \end{equation}
\end{theorem}
\begin{proof}
  Proposition~\ref{prop:kirchhoff-nfosls} verifies
  Assumption~\ref{asm:NDPGdiff} with $D = \Ua$ and
  $\delta = +\infty$, while Assumption~\ref{asm:fortin-a} holds with
  $\Vah = \Va$ and $\fp = \mathrm{id}$, so that
  $\|\fp\|_{L(\Va)} = 1$ and $\osc(\ell) = 0$.
  Theorem~\ref{thm:NDPG-gen-nonlinear-dpg} then applies with
  $h_0 = +\infty$ and $\tilde D_{h_0} = \Ua$, as recorded
  in~\eqref{eq:ls-vc}.
\end{proof}

\subsection{The NDPG loss function}
\label{ssec:kirchhoff-setting}

We use a {\em primal} DPG formulation of~\eqref{eq:KH-classical},
posed using the broken space $H^1(\oh)$ of~\eqref{eq:13} and the
interface space $H^{-1/2}(\partial\oh)$ normed as
in~\eqref{eq:quotient-norms}.  It is derived by the standard
primal DPG technique~\cite{DemkoGopal13a,DPGacta} of multiplying
\eqref{eq:KH-classical} by test functions, integrating by parts
element by element, and declaring the interelement trace of the flux
to be an independent unknown $\hat r_n$.  The formulation seeks
$\clu \in \Ua$ satisfying
\begin{subequations}
  \label{eq:KH-DPG}
  \begin{equation}
    \label{eq:KH-DPG-eq}
    a_\alpha(\clu; \clv) = \ell(\clv),
    \qquad \text{ for all } \clv \in \Va,
  \end{equation}
  where, now,
  \begin{align}
    & \Ua
      = \Ho^1 \times H^{-1/2}(\partial\oh),
    \\
    & \Va
      = H^1(\oh),
    \\
    & a_\alpha( (w, \hat r_n); \clv)
      = M_\alpha( G_\alpha(w))\, (\kappa_\alpha \nabla w, \nabla \clv)_{\oh}
      - \ip{ \hat r_n, \clv}_{\partial\oh},
      \label{eq:KH-BF}
    \\
    & \ell( \clv) = (f, \clv)_{L_2},
  \end{align}
  with the norms
  $\| (w, \hat r_n)\|^2_{\Ua} = \|\nabla w\|^2_{L_2}
  + \|\hat r_n\|^2_{H^{-1/2}(\partial\oh)}$ on the trial space and the broken
  $H^1(\oh)$ norm on the test space.  By~\eqref{eq:aadaa}, the above
  $a_\alpha$ defines the nonlinear operator $\Aa$ that
  casts~\eqref{eq:KH-DPG-eq} into the form~\eqref{eq:Au=ell},
  \begin{equation}
    \label{eq:KH-op}
    \Aa(\clu) = \ell, \quad\text{with}\quad \clu = (u, \hat q_n),
  \end{equation}
\end{subequations}
and we take $D = \Ua$.  For the discrete spaces, we choose, for any
polynomial degree $k \ge 0$,
\begin{equation}
  \label{eq:KH-discrete}
  \Uah = U_{k+1} \times \hat\RT_k,
  \qquad
  \Vah = P_{k+d},
\end{equation}
where $U_{k+1} = P_{k+1} \cap \Ho^1$, and $\hat\RT_k = \trn(\RT_k)$ is
the space of normal traces on element boundaries of the degree-$k$
Raviart--Thomas space $\RT_k$ of~\eqref{eq:fos-discrete}, so that
$k = 0$ recovers the spaces $U_1$ and $\hat\RT$
of~\eqref{eq:uw-discrete}.
The approximation
property~\eqref{eq:approx'} holds for this $\Uah$ by standard finite
element approximation theory.  All ingredients defining the NDPG loss
function $\Ja\cdot$ of~\eqref{eq:DPG-loss-defn} are now again in
place.

Some remarks on the computation of this loss
function are in order.  As described in
\eqref{eq:residual-rep'}--\eqref{eq:L-and-E}, evaluating
$\Ja{\clw_h}$ at a $\clw_h = (w_h, \hat r_{n,h}) \in \Uah$ amounts to
solving, element by element, small linear symmetric positive definite
systems for the residual representation $\Ea a \clwh \in \Vah$.  The
nonlinearity enters these local computations only through the single
number $m_h := M_\alpha(G_\alpha(w_h))$.  Since $\nabla w_h$ is a
polynomial of degree $k$ on each element for $w_h \in U_{k+1}$,
whenever $\kappa_\alpha$ is piecewise constant on $\oh$ the global
quantity
\[
  G_\alpha(w_h) = \sum_{K \in \oh}
  \kappa_\alpha|_K \int_K \nabla w_h \cdot \nabla w_h
\]
is computed exactly by one pass over the mesh, using a quadrature rule
exact for polynomials of degree $2k$, after which a single
evaluation of $M_\alpha$ yields $m_h$.
The derivatives of $\Ja{\clw_h}$ required by gradient-based training
are equally simple: by the chain rule, they involve, beyond linear
terms, only the additional scalar $M_\alpha'(G_\alpha(w_h))$.

\begin{proposition}[Verification of Assumption~\ref{asm:NDPGdiff}]
  \label{prop:kirchhoff-ndpg}
  Let Hypothesis~\ref{hyp:M} hold.  Then, in the setting
  of~\eqref{eq:KH-DPG}, all conditions of
  Assumption~\ref{asm:NDPGdiff} hold with $D = \Ua$ and
  $\delta = +\infty$.  Specifically:
  \begin{enumerate}[label={\rm(\alph*)}]
  \item for each $\alpha \in \cQ$, the element
    $\clu(\alpha) := (u(\alpha), \hat q_n) \in \Ua$, with $u(\alpha)$
    as in Lemma~\ref{lem:kirchhoff-wellposed} and
    $\hat q_n := \trn(\kappa_\alpha\nabla w_\alpha)$, satisfies $\Ra{\clu(\alpha)} = 0$;
    it is in fact the unique element of $\Ua$ with this property;
  \item $\Aa$ is \Frechet\ differentiable at every $\clw \in \Ua$
    and~\eqref{linvc} holds with
    $\CdB = \bigl[ C_K^2 + 1 \bigr]^{1/2}$;
  \item inequality~\eqref{eq:c-prime} holds for all $\clz \in \Ua$,
    i.e., with $\delta = +\infty$, with $c_* = \cdB$ given by
    \begin{equation}
      \label{eq:KH-cd}
      \frac{1}{\cdB^2}
      = \frac{1 + c_P^{-1}}{c_K^2}
      + \Bigl( 1 + \frac{(1 + c_P^{-1})^{1/2}\, C_K}{c_K} \Bigr)^2 ;
    \end{equation}
    in particular $c_*$ is independent of $\alpha \in \cQ$ and of the
    mesh.
  \end{enumerate}
\end{proposition}
\begin{proof}
  Since $G$ is a continuous quadratic form on $\Ho^1$, it is
  infinitely differentiable with
  $DG_\alpha(w)(\zeta) = 2 (\kappa_\alpha\nabla w, \nabla \zeta)_\om$.  By {\rm(H1)} and
  the chain and product rules, $\Aa$ is \Frechet\ differentiable at
  every $\clw = (w, \hat r_n) \in \Ua$ with, for
  $\clz = (\zeta, \hat\varsigma_n) \in \Ua$ and $\clv \in \Va$,
  \begin{equation}
    \label{eq:KH-dA}
    \begin{aligned}
      d\Aa(\clw)(\clz)(\clv)
      & = M_\alpha(G_\alpha(w))\, (\kappa_\alpha \nabla \zeta, \nabla \clv)_{\oh}
      - \ip{ \hat\varsigma_n, \clv}_{\partial\oh}
      \\
      & \quad
      + 2 M_\alpha'(G_\alpha(w))\, (\kappa_\alpha \nabla w, \nabla \zeta)_\om\,
      (\kappa_\alpha \nabla w, \nabla \clv)_{\oh} .
    \end{aligned}
  \end{equation}

  {\em Proof of {\rm(a)}.}  Since $-\dive (\kappa_\alpha\nabla w_\alpha) = f \in
  L_2$ holds distributionally, the flux $\kappa_\alpha\nabla w_\alpha$ is in
  $\Hdiv$, so $\hat q_n = \trn(\kappa_\alpha\nabla w_\alpha)$ is well defined.  Using
  $M_\alpha(G_\alpha(u(\alpha)))\, \kappa_\alpha\nabla u(\alpha)
  = \kappa_\alpha\nabla w_\alpha$ from
  Lemma~\ref{lem:kirchhoff-wellposed} and elementwise integration by
  parts,
  $
    a_\alpha( \clu(\alpha); \clv)
    = (\kappa_\alpha\nabla w_\alpha, \nabla \clv)_{\oh}
    - \ip{ \trn(\kappa_\alpha\nabla w_\alpha), \clv}_{\partial \oh}
    = -(\dive (\kappa_\alpha\nabla w_\alpha), \clv)_\om
    = (f, \clv)_\om
    = \ell(\clv)
  $
  for all $\clv \in \Va$, so $\Ra{\clu(\alpha)} = 0$.  For
  uniqueness, suppose $\clw = (\tilde u, \tilde q_n) \in \Ua$ also
  satisfies $\Aa(\clw) = \ell$.  Testing with conforming
  $\clv \in \Ho^1 \subset \Va$ annihilates the interface term,
  so $\tilde u$ solves~\eqref{eq:KH-weak}, and
  Lemma~\ref{lem:kirchhoff-wellposed} gives $\tilde u = u(\alpha)$.
  Then $\ip{\tilde q_n - \hat q_n, \clv}_{\partial\oh} = 0$ for all
  $\clv \in \Va$, which forces $\tilde q_n = \hat q_n$ by the duality
  underlying the quotient norm in~\eqref{eq:quotient-norms} (see
  also~\cite{CarstDemkoGopal16}).

  {\em Proof of {\rm(b)}.}  By~\eqref{eq:KH-CS} and
  {\rm(H1)}--{\rm(H2)}, exactly as in the first bound
  of~\eqref{eq:KH-coercive},
  \[
    | d\Aa(\clw)(\clz)(\clv) |
    \le C_K\, \|\nabla\zeta\|_{L_2} \|\clv\|_{\Va}
    + \|\hat\varsigma_n\|_{H^{-1/2}(\partial\oh)} \|\clv\|_{\Va}
    \le \CdB \|\clz\|_{\Ua} \|\clv\|_{\Va}
  \]
  with $\CdB = [C_K^2 + 1]^{1/2}$,
  proving~\eqref{linvc}.

  {\em Proof of {\rm(c)}.}  Write
  $b_w(\zeta, v) := (B_w \zeta, \nabla v)_{\oh}$ for the
  bilinear form on $\Ho^1 \times H^1(\oh)$ generated by the
  linearization~\eqref{eq:KH-frozen-form}, so that, by
  Lemma~\ref{lem:kirchhoff-frozen}, $b_w$ is bounded with constant
  $C_K$ and, when restricted to conforming $v$, coercive on
  $\Ho^1$ with constant $c_K$, both uniformly in $w$
  and~$\alpha$.  Comparing with~\eqref{eq:KH-dA},
  $d\Aa(\clw)(\clz)(\clv) = b_w(\zeta,\clv) -
  \ip{\hat\varsigma_n, \clv}_{\partial\oh}$.

  Now fix $\clz^0 = (w^0, \hat r_n^0)$ and $\clz^1 = (w^1, \hat r_n^1)$ in
  $\Ua$, arbitrary, and put $e := w^0 - w^1$,
  $\hat e_n := \hat r_n^0 - \hat r_n^1$, and
  $\clz^t := \clz^1 + t(\clz^0 - \clz^1)$ for $t \in [0,1]$.  Fix
  $\clv \in \Va$.  By~(b) and the continuity of $M_\alpha$ and
  $M_\alpha'$ from {\rm(H1)}, the function
  $t \mapsto a_\alpha(\clz^t; \clv)$ is continuously differentiable on
  $[0,1]$, with derivative $d\Aa(\clz^t)(\clz^0 - \clz^1)(\clv)$.
  Hence, by the FTC and~\eqref{eq:KH-dA},
  \begin{equation}
    \label{eq:KH-averaged}
    \ip{ \Aa(\clz^0) - \Aa(\clz^1), \clv}
    = \int_0^1 d\Aa(\clz^t)(\clz^0 - \clz^1)(\clv)\, dt
    = \bar b(e, \clv) - \ip{ \hat e_n, \clv}_{\partial\oh},
  \end{equation}
  where $\bar b(\zeta, v) := \int_0^1 b_{w^t}(\zeta, v)\, dt$ is the
  average of these forms along the segment, with $w^t$ denoting the
  $\Ho^1$ component of $\clz^t$.  By the last assertion of
  Lemma~\ref{lem:kirchhoff-frozen}, $\bar b$ is bounded and coercive
  with the {\em same} two constants.

  The right-hand side of~\eqref{eq:KH-averaged} is exactly the broken
  primal DPG form generated by $\bar b$ and the interface variable
  $\hat e_n$.  Therefore the theory of broken variational
  formulations~\cite[Theorem~3.3]{CarstDemkoGopal16} converts the
  Lax--Milgram well-posedness of $\bar b$ into the inf-sup estimate
  \begin{equation}
    \label{eq:inf-sup-kirchhoff}
    \sup_{0 \ne \clv \in \Va}
    \frac{ | \ip{ \Aa(\clz^0) - \Aa(\clz^1), \clv} |}{ \|\clv\|_{\Va}}
    \ge \cdB\, \| \clz^0 - \clz^1 \|_{\Ua},
  \end{equation}
  with $\cdB$ as in~\eqref{eq:KH-cd}, and in particular independent of
  $\clz^0, \clz^1$ and of $\alpha \in \cQ$.  That reference gives
  $\cdB^{-2} = c_{\rm ub}^{-2} + \hat c^{-2} ( \|\bar b\|/c_{\rm ub} + 1)^2$, where
  $c_{\rm ub}$ and $\|\bar b\|$ are the inf-sup and continuity constants of
  $\bar b$ on the {\em unbroken} pair $\Ho^1 \times \Ho^1$ and
  $\hat c$ is the stability constant of the interface term.  Here
  $\hat c = 1$, since the quotient norm~\eqref{eq:quotient-norms} is
  the dual norm of that term, while Lemma~\ref{lem:kirchhoff-frozen}
  and the Friedrichs inequality~\eqref{eq:Friedrichs-std}, which gives
  $\|\zeta\|_{H^1(\om)} \le (1 + c_P^{-1})^{1/2}\|\nabla\zeta\|_{L_2}$,
  yield $\|\bar b\| \le C_K$ and
  $c_{\rm ub} \ge c_K (1 + c_P^{-1})^{-1/2}$, whence~\eqref{eq:KH-cd}.  Taking $\clz^0 = \clu(\alpha)$
  gives~\eqref{eq:c-prime} for every $\clz = \clz^1 \in \Ua$, that is,
  {\rm(c)} with $c_* = \cdB$ and $\delta = +\infty$.  (Since the
  estimate holds for all pairs, it also re-derives the uniqueness
  asserted in {\rm(a)}.)
\end{proof}

\begin{remark}[The alternate {\rm(b$'$)}--{\rm(c$'$)} route]
  \label{rem:kirchhoff-bc-route}
  This is the only application in this paper where both sufficient
  conditions of Remark~\ref{rem:sufficient-c} hold, and it is worth
  recording what they give.  Condition {\rm(b$'$)} is immediate, since
  \eqref{eq:KH-dA} is the broken primal DPG form generated by the
  linearization~\eqref{eq:KH-frozen-form}, so
  that~\cite[Theorem~3.3]{CarstDemkoGopal16} and the
  bounds~\eqref{eq:KH-coercive} supply the same $\cdB$ as
  in~\eqref{eq:inf-sup-kirchhoff}; condition {\rm(c$'$)} holds
  because $M_\alpha$ sees $w$ only through the
  scalar $G_\alpha(w)$ and $G_\alpha$ is norm-continuous (no Nemytskii operator intervenes between $w$ and the
  coefficient as in
  \S\ref{sec:nonlinear-diffusion}, cf.\ Remark~\ref{rem:no-frechet}). One can show that
  $d\Aa$ is therefore Lipschitz on bounded subsets of $\Ua$, with
  $\alpha$-independent constants.  Applying
  Remark~\ref{rem:sufficient-c} at $\clu(\alpha)$, whose norm is
  bounded independently of $\alpha$ by
  Lemma~\ref{lem:kirchhoff-wellposed}, and choosing the neighborhood
  radius $\delta$ so that the linearization error absorbs half the
  inf-sup constant, yields~\eqref{eq:c-prime} with $c_* = \cdB/2$, but
  only for $\clz$ within $\delta$ of $\clu(\alpha)$.  This is weaker
  than Proposition~\ref{prop:kirchhoff-ndpg}{\rm(c)} and would only deliver local variational correctness.
\end{remark}

\begin{proposition}[Verification of Assumption~\ref{asm:fortin-a}]
  \label{prop:fortin-kirchhoff}
  In the setting of~\eqref{eq:KH-DPG} and~\eqref{eq:KH-discrete},
  suppose $\kappa_\alpha$ is piecewise constant on $\oh$.  Then
  Assumption~\ref{asm:fortin-a} holds with $D_h = \Uah$ and the
  Fortin operator $\fp := \Pi^{\rm grad}_{k+d}$ for $k\ge1$, from
  \cite[Lemma~3.2]{GopalQiu14}, and $\fp := \tilde\Pi_d$ for $k=0$,
  from~\eqref{eq:fortin-mean-correction}.  Neither the polynomial degree $k$
  nor the nonlinearity $M_\alpha$ is restricted beyond
  Hypothesis~\ref{hyp:M}.
\end{proposition}
\begin{proof}
  Let $\clwh = (w_h, \hat r_{n,h}) \in \Uah$ and $\clv \in \Va$.  The
  Fortin condition~\eqref{eq:Fortin} reads
  \[
    M_\alpha( G_\alpha(w_h))\,
    (\kappa_\alpha \nabla w_h, \nabla( \clv - \fp \clv))_{\oh}
    - \ip{ \hat r_{n,h}, \clv - \fp \clv}_{\partial\oh}
    = 0.
  \]
  By~\cite[Lemma~3.2]{GopalQiu14} for $k\ge1$ and
  \eqref{eq:fortin-preservation} for $k=0$, this uniformly bounded
  operator satisfies, on every $K \in \oh$,
  \begin{equation}
    \label{eq:KH-fortin-orth}
    (\clv - \fp \clv, y_h)_K = 0, \;\;
    \qquad
    \ip{ \clv - \fp \clv, \mu_h}_{\partial K} = 0
  \end{equation}
  for all $y_h\in   P_{k-1}(K)$ and $\mu_h \in P_k(\partial K).$
  Since $\hat r_{n,h}|_{\partial K}$ is in $P_k(\partial K)$, the
  second term of the Fortin condition vanishes by the second
  orthogonality.  For the first, $w_h \in U_{k+1}$ and the piecewise
  constancy of $\kappa_\alpha$ put $\kappa_\alpha \nabla w_h$ in
  $P_k^d(K)$, so elementwise integration by parts gives
  \[
    (\kappa_\alpha \nabla w_h, \nabla( \clv - \fp \clv))_K
    = - (\dive(\kappa_\alpha \nabla w_h), \clv - \fp \clv)_K
    + \ip{ \kappa_\alpha \nabla w_h \cdot n, \clv - \fp \clv}_{\partial K},
  \]
  where $\dive(\kappa_\alpha \nabla w_h) \in P_{k-1}(K)$ and
  $\kappa_\alpha \nabla w_h \cdot n \in P_k(\partial K)$, so both
  terms vanish by~\eqref{eq:KH-fortin-orth}.  Hence the two terms of
  the Fortin condition vanish separately---the first for any value of
  the scalar $M_\alpha(G_\alpha(w_h))$---so~\eqref{eq:Fortin} holds
  for every $k \ge 0$, with no restriction on the nonlinearity
  $M_\alpha$.
\end{proof}

\begin{theorem}[Variational correctness of the Kirchhoff NDPG loss]
  \label{thm:vc-NDPG-Kirchhoff}
  Let Hypothesis~\ref{hyp:M} hold, adopt the
  setting~\eqref{eq:KH-DPG} and~\eqref{eq:KH-discrete}, and let
  $\kappa_\alpha$ be piecewise constant on $\oh$.  Then the
  NDPG loss function $\Ja\cdot$ of~\eqref{eq:DPG-loss-defn} is
  {\em globally} variationally correct, with the $\alpha$- and
  mesh-independent constants $c = \cdB^2/2$ and $C = \CdB^2 = C_K^2 + 1$,
  where $\cdB$ is given by~\eqref{eq:KH-cd} in terms of $c_K$, $C_K$
  and $c_P$: the
  estimates~\eqref{eq:reliability}--\eqref{eq:efficiency} hold for
  every mesh and every $\clwh \in \Uah$.  If, moreover,
  $f \in P_{\max(k-1,\,0)}$, then $\osc(\ell) = 0$, so that
  \begin{equation}
    \label{eq:KH-vc}
    \frac{c}{\|\fp\|^2_{L(\Va)}}\,
    \| \clu - \clwh \|^2_{\Ua}
    \le \Ja{\clwh}
    \le C\, \| \clu - \clwh \|^2_{\Ua},
    \qquad \clwh \in \Uah.
  \end{equation}
\end{theorem}
\begin{proof}
  Propositions~\ref{prop:kirchhoff-ndpg}
  and~\ref{prop:fortin-kirchhoff} verify
  Assumptions~\ref{asm:NDPGdiff} and~\ref{asm:fortin-a},
  respectively.  Since $D = \Ua$ and $\delta = +\infty$ in
  Proposition~\ref{prop:kirchhoff-ndpg},
  Theorem~\ref{thm:NDPG-gen-nonlinear-dpg} holds globally, with
  $h_0 = +\infty$ and $\tilde D_{h_0} = \Ua$, and
  gives~\eqref{eq:reliability}--\eqref{eq:efficiency} with
  $c = c_*^2/2 = \cdB^2/2$ and $C = \CdB^2$; these constants are
  independent of $\alpha$ and of the mesh because $\cdB$ and $\CdB$
  are.  When $f \in P_{k-1}$, the first orthogonality
  in~\eqref{eq:KH-fortin-orth} gives
  $\ell(\clv - \fp \clv) = (f, \clv - \fp \clv)_\om = 0$ for all
  $\clv \in \Va$; the same holds when $f \in P_0$ and $k = 0$, since
  $\tilde\Pi_d$ preserves element averages
  by~\eqref{eq:fortin-preservation}.  Either way
  $\osc(\ell) = 0$,
  and~\eqref{eq:reliability} implies the lower inequality
  in~\eqref{eq:KH-vc}.
\end{proof}

\subsection{Discussion}

With both applications (of Sections~\ref{sec:nonlinear-diffusion} and~\ref{sec:kirchhoff-ndpg}) now in hand, their similarities and  contrasts are worth
noting.  In each, condition~\eqref{eq:c-prime} was verified
directly, by averaging the derivative along a segment, and in each
this yielded $\delta = +\infty$, hence global variational
correctness.  What separates them is the cost of the discrete
hypothesis, Assumption~\ref{asm:fortin-a}, in the DPG cases:
the Kirchhoff application is
the more favorable of the two on this count.
In the quasilinear
application, the NDPG result of Theorem~\ref{thm:vc-NDPG-example}
needed the parameter to be piecewise constant on $\oh$ and the trial
degree to be the lowest one, so that the nonlinearity evaluated at a
trial function stays piecewise polynomial.
In the Kirchhoff case,
its primal NDPG loss is
covered, by Theorem~\ref{thm:vc-NDPG-Kirchhoff}, under the same mild
condition on the coefficient $\kappa_\alpha$ as in the
quasilinear case, but with no condition on the nonlinearity and
at every polynomial degree: the Fortin property requires no
adaptation to $M_\alpha$ because $M_\alpha$ enters the form only
through a scalar multiplier.  Its NFOSLS loss, by
Theorem~\ref{thm:vc-NFOSLS-Kirchhoff}, applies even to elements of meshless hypothesis classes, and
is the cheapest loss in the paper to evaluate.
It is also the only place in this paper where the
sufficient conditions of Remark~\ref{rem:sufficient-c} are available.

\section{Application to neural approximations}
\label{sec:nn}
\definecolor{depthA}{RGB}{244,109,67}
\definecolor{depthB}{RGB}{215,48,39}
\definecolor{depthC}{RGB}{153,15,21}
\definecolor{depthD}{RGB}{75,0,10}
\definecolor{depthAb}{RGB}{116,169,207}
\definecolor{depthBb}{RGB}{43,140,190}
\definecolor{depthCb}{RGB}{4,90,141}
\definecolor{depthDb}{RGB}{2,40,80}
\definecolor{epA}{RGB}{166,196,222}
\definecolor{epB}{RGB}{107,153,208}
\definecolor{epC}{RGB}{49,110,180}
\definecolor{epD}{RGB}{12,70,140}
\definecolor{epE}{RGB}{5,30,80}
\definecolor{degA}{RGB}{31,119,180}
\definecolor{degB}{RGB}{214,96,20}
\definecolor{degC}{RGB}{35,110,45}

\pgfplotsset{
  s6ablation/.style={
    width=8.7cm, height=5.9cm,
    xmajorgrids=true, ymajorgrids=true,
    grid style={gray!22, very thin},
    tick label style={font=\scriptsize},
    label style={font=\footnotesize},
    title style={font=\footnotesize\bfseries, yshift=-2pt},
    legend cell align=left,
    legend style={font=\scriptsize, draw=none, fill=white,
                  fill opacity=0.93, text opacity=1, inner sep=1.5pt,
                  row sep=-1pt},
  },
  s6best/.style={thick, mark size=1.5pt},
}
\providecommand{\cfiveSlopeMeshKirchhoffR}{1.98}
\providecommand{\cfiveSlopeMeshKirchhoffRfem}{1.99}
\providecommand{\cfiveSlopeMeshNdpgR}{2.06}
\providecommand{\cfiveSlopeMeshNdpgRfem}{2.08}
\providecommand{\cfiveSlopeTrainsetKirchhoffE}{-0.23}
\providecommand{\cfiveSlopeTrainsetKirchhoffR}{-0.73}
\providecommand{\cfiveSlopeTrainsetNdpgE}{-0.52}
\providecommand{\cfiveSlopeTrainsetNdpgR}{-1.30}
We now train neural network surrogates of the parameter-to-solution map
with the NDPG losses, for the quasilinear flux model
of~\S\ref{ssec:nonlinear-diffus-flux-model} with the invertible profile of
Example~\ref{ex:invertible}, and for the stationary Kirchhoff model
of~\S\ref{sec:kirchhoff-ndpg}.  Experiments with analogs of NFOSLS losses
and similar networks are reported in~\cite{BDO,CDG,QCD}, so we focus on
NDPG.  Network predictions are compared with the finite element solution
in the space they predict into, with a high-fidelity solution in a richer
space, and through the ratio of loss to squared error, which variational
correctness bounds on both sides.

\subsection{The learning problem}
\label{ssec:learning-problem}
A loss that can be evaluated at an arbitrary input, and whose value is
equivalent to the error that input commits, can be used to approximate
$\alpha\mapsto\clu(\alpha)$ over a whole family of parameters at once.
Here ``learning'' means {\em regression} over a {\em hypothesis class}
$\mathcal H=\mathcal H(\Theta)$ of functions of $x\in\om$ and
$\alpha\in\cQ$, each determined by a set $\theta\in\Theta$ of trainable
weights and taking values in $\Ua$ at every $\alpha$.  Endowing $\cQ$ with
a probability measure $\mu$, the target is the minimizer of the population
risk
\begin{equation}
\label{minexp}
\clu =\argmin_{\clz \in D}\mathbb E_{\alpha\sim\mu}\big[\Ra\clz\big]= \argmin_{\clz \in D}\| \ell - \Aa(\clz)\|^2_{L_{2,\mu}(\cQ;(\Vap)_{\alpha\in\cQ})},
\end{equation}
\newcommand{\Y}{\mathbb{Y}}
\newcommand{\X}{\mathbb{X}}
where $\X:= L_{2,\mu}(\cQ;(\Ua)_{\alpha\in\cQ})$ and
$\Y'=L_{2,\mu}(\cQ;(\Vap)_{\alpha\in\cQ})$ are the direct integrals of the
(assumed measurable) families $(\Ua)_{\alpha\in\cQ}$ and
$(\Vap)_{\alpha\in\cQ}$, see~\cite{BDO}: $\X$ consists of the
$\mu$-equivalence classes of measurable $\clz$ with $\clz(\alpha)\in\Ua$
for $\mu$-a.e.\ $\alpha$ and
\begin{equation}
  \label{eq:Xnorm}
  \|\clz\|_{\X}^2
  :=\int_{\cQ}\|\clz(\alpha)\|_{\Ua}^2\,d\mu(\alpha)
  =\mathbb E_{\alpha\sim\mu}\big[\|\clz(\alpha)\|_{\Ua}^2\big]<\infty ,
\end{equation}
and likewise for $\Y'$.  Since the integrand depends on $\clz$ only
through $\clz(\alpha)$, \eqref{minexp} reduces to
\eqref{eq:exact-minimization} for $\mu$-a.e.\ $\alpha$.  We assume that
the solution operator
\begin{equation}
    \label{solop}
\mathcal F:   \alpha \mapsto \clu(\cdot;\alpha)
\end{equation}
maps into $\X$, and measure the accuracy of an approximation
$\alpha\mapsto\clu(\alpha;\theta^*)$ by $\|\clu-\clu(\theta^*)\|_{\X}$, its
root mean squared error over $\alpha\sim\mu$.

This error is not computable; the loss is minimized in its place.  For
both models below the loss is globally variationally correct, by
Theorems~\ref{thm:vc-NDPG-example} and~\ref{thm:vc-NDPG-Kirchhoff}, with
constants $c'$, $C'$ independent of $\alpha$ and of the mesh, and
$\osc(\ell)=0$ since $f\equiv1$.  Integrating in $\mu$ gives
\begin{equation}
  \label{vc}
    c'\, \| \clu - \clz \|_\X^2
    \;\le\;
    \bb E_{\alpha \sim \mu} \bigl[ \Jaa{\alpha}{\clz(\alpha)}\bigr]
    \;\le\;
    C'\, \| \clu - \clz \|_\X^2,
    \qquad \clz\in\X_h,
\end{equation}
where $\X_h:=\{\clz\in\X:\ \clz(\alpha)\in\Uah$ for $\mu$-a.e.\ $\alpha\}$.
Since $D=\Ua$ and $\delta=+\infty$ in Assumption~\ref{asm:NDPGdiff}(c) for
both models, the restriction to $\tilde D_{h_0}$ in
Theorem~\ref{thm:NDPG-gen-nonlinear-dpg} is inactive, and \eqref{vc}
applies to the output of a trained network, however far it is from $\clu$.
(For NFOSLS, Theorem~\ref{thm:vc-NFOSLS-example} gives the same estimate
for every $\clz\in\X$.)

Training minimizes the empirical counterpart of the expectation
in~\eqref{vc},
\begin{equation}
    \label{emprisk}
    \alpha\mapsto \clu(\alpha;\theta^*)\approx \clu(\alpha),\qquad \theta^*\in \argmin_{\theta\in\Theta}\frac{1}{\#\widehat\cQ}\sum_{\alpha\in\widehat\cQ}
    \Jaa{\alpha}{ \clu(\alpha;\theta)},
\end{equation}
over a finite set $\widehat\cQ\subset\cQ$ of $\mu$-i.i.d.\ draws.  Each
summand is computable from $\alpha$ and the network output alone, so no
solution of~\eqref{eq:Au=ell} is needed to evaluate~\eqref{emprisk} or its
gradient; the purely data-driven approach of operator learning instead
fits precomputed high-fidelity solutions, at the cost of a nonlinear solve
per training sample.  Estimate~\eqref{vc} bounds the error by the
population loss but does not control the gap between training and
population losses.  For a fixed trained network with finite loss
variance, an independent test sample of size $n_{\rm test}$ estimates its
population loss with root mean squared sampling error of order
$n_{\rm test}^{-1/2}$.

The parameter domain is finite-dimensional throughout: $\om$ is
partitioned into subdomains $\om_i$, $i=1,\dots,m_\alpha$, with indicator
functions $\chi_i$, and
\begin{equation}
  \label{eq:Qdefn}
  \mathcal{Q} = \left \{\sum_{i=1}^{m_\alpha} \alpha^{(i)} \chi_i(x) \; : \; \alpha^{(i)} > 0 \right \}.
\end{equation}
In every experiment the $\om_i$ are the cells of a uniform $n\times n$
tiling of $\om=(0,1)^2$,
\begin{equation}
  \label{eq:tiling}
  \om_{(i_1,i_2)}
  =\left(\tfrac{i_1-1}{n},\tfrac{i_1}{n}\right)\times
   \left(\tfrac{i_2-1}{n},\tfrac{i_2}{n}\right),
  \qquad (i_1,i_2)\in\{1,\dots,n\}^2 ,
\end{equation}
enumerated once and for all, so that $m_\alpha=n^2$, with $n=2,\dots,5$.
Identifying $\alpha$ with its vector of values
$(\alpha^{(1)},\dots,\alpha^{(m_\alpha)})$ makes $\cQ$ an open subset of
$\mathbb R^{m_\alpha}$.  An $\alpha$ taking distinct values on adjacent
cells jumps across their interface, so~\eqref{eq:Au=ell} is a nonlinear
interface problem.  Function-valued parameters could be treated by
adapting the network input, but since our concern is variational
correctness we stay with this finite-dimensional case.  Every mesh $\oh$
used below resolves the tiling, so every $\alpha\in\cQ$ is piecewise
constant on $\oh$, as Theorems~\ref{thm:vc-NDPG-example}
and~\ref{thm:vc-NDPG-Kirchhoff} require.  The range of the
$\alpha^{(i)}$ and the measure $\mu$ are specified
in~\eqref{eq:alphadists}.  Figure~\ref{fig:exp2-alpha} shows a sample
$\alpha^*$ at $n=5$ and, as an example of what follows, the trained DPG network's prediction at that single sample
$\alpha^*$ for the invertible
nonlinearity of~\S\ref{sssec:num-models}, next to the nonlinear DPG
solution in the same trial space: at a single parameter the two are
visually indistinguishable.

\begin{figure}[htbp]
  \centering
  \begingroup
  \setlength{\tabcolsep}{1pt}
  \newcommand{\ndpgpanel}[1]{\includegraphics[width=.2\textwidth]{figures/ndpg_showcase/panels/#1.pdf}}\newcommand{\ndpgbar}[1]{\includegraphics[height=.2\textwidth]{figures/ndpg_showcase/panels/colorbar_#1.pdf}}\newcommand{\ndpgrowlabel}[1]{\raisebox{.5\dimexpr.2\textwidth-\height+\depth\relax}{\shortstack{\footnotesize #1}}}\begin{tabular}{@{}cc@{}}
    \multicolumn{2}{c}{\scriptsize\textbf{$\alpha^*$}}\\[-1pt]
    \includegraphics[width=.3\textwidth]{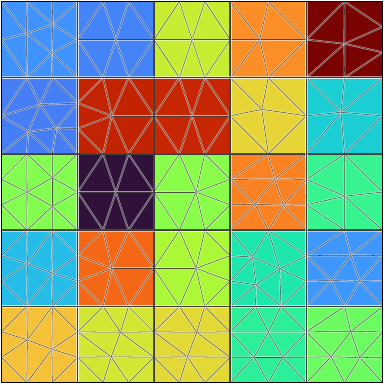}
    & \includegraphics[height=.3\textwidth]{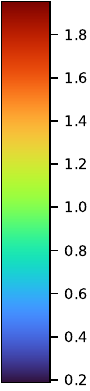}
  \end{tabular}\hspace{.03\textwidth}\begin{tabular}{@{}rccc@{}}
    & \scriptsize\textbf{NDPG-NN} & \scriptsize\textbf{NDPG-FEM} & \\[-1pt]
    \ndpgrowlabel{$\hat u$}
    & \ndpgpanel{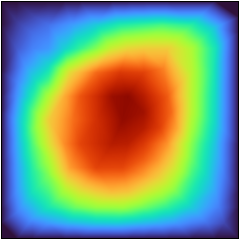} & \ndpgpanel{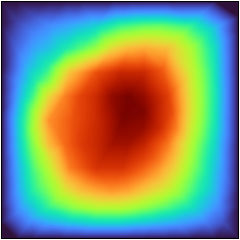} & \ndpgbar{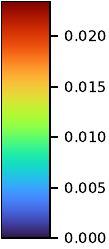} \\[1pt]
    \ndpgrowlabel{$|\hat q_n|$}
    & \ndpgpanel{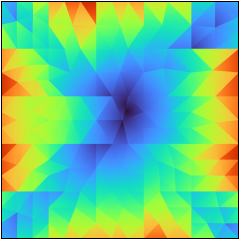} & \ndpgpanel{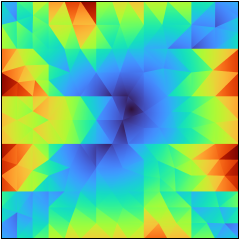}
    & \ndpgbar{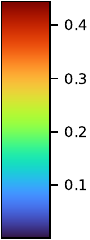}
  \end{tabular}
  \endgroup
  \caption{{\em Left.} An instance of the tiling~\eqref{eq:tiling} at
    $n=5$ (so $m_\alpha=25$) showing a coefficient $\alpha^*$ with
    contrast $\varkappa=20$, together with the aligned mesh of $298$
    triangles at $h=0.08$.  {\em Right.} For the invertible
    nonlinearity at that $\alpha^*$, the traces $\hat u$ (up) and
    $|\hat q_n|$ (down) are shown, through their minimal $H^1$ and
    $H(\dive)$ lifts ($k=0$), for two cases: for the prediction of a
    network trained with the NDPG loss (NDPG-NN) on $n_{\rm train}=4096$ samples for $20\,000$ epochs,
    and for the nonlinear DPG solution, i.e., the minimizer of the
    same loss over the same trial space $\Uah$ (NDPG-FEM). They are visually indistinguishable.}
  \label{fig:exp2-alpha}
  \label{fig:exp2-solutions}
\end{figure}

The hypothesis class combines the finite element trial space $\Uah$ with
a neural network producing coefficients in a fixed basis of it.  Let
$\Uah$ be~\eqref{eq:uw-discrete} for the ultraweak model
and~\eqref{eq:KH-discrete} for the Kirchhoff model, on a shape-regular
mesh $\oh$ of maximal element diameter~$h$.  In both instances $\Uah$ is
independent of $\alpha$, together with its norm, so
one basis $\psi_k$, $k=1,\dots,m_h$, serves all $\alpha\in\cQ$.
Identifying elements of $\Uah$ and $\cQ$ with their coefficient vectors
in $\bb R^{m_h}$ and $\bb R^{m_\alpha}$, an approximation $\cl F_h$ of the
solution operator~\eqref{solop} with values in $\Uah$ is determined by a
map $\mathbb{F}: \mathbb{R}^{m_\alpha} \to \mathbb{R}^{m_h}$ through
\begin{equation}
  \label{eq:fun-map}
  \cl F_h\Big( \sum_{i=1}^{m_\alpha} \alpha^{(i)} \chi_i(x)\Big)
  = \sum_{k=1}^{m_h} \,[\mathbb F(\alpha)]_k\, \psi_k(x),
  \qquad x \in \om.
\end{equation}
For $\mathbb{F}$ we use the low-rank residual architecture
of~\cite{CDG,BDO}.  Let $\varrho(y)_i =\max( y_i, 10^{-3} y_i)$ be the
leaky ReLU acting componentwise on $\bb R^N$, $N$ being the width, let $r$
be the rank (often $r\ll N$~\cite{BDO}) and $l$ the number of layers, and
let $A_{\text{out}} \in \bb R^{m_h \times N}$, $A_m \in \bb R^{N \times r}$,
$W_m \in \bb R^{r \times N}$, $A_{\text{in}} \in \bb{R}^{N \times m_\alpha}$,
$b_m \in \bb R^r$, $b_{\text{out}} \in \bb{R}^{m_h}$,
$b_{\text{in}} \in \bb{R}^{N}$.  With
\begin{gather*}
  \Phi_m (z) = z + A_m \varrho\big( W_m z + b_m \big),\quad m=1,\dots,l,
  \\
  L_{\text{out}}(z) = A_{\text{out}} z + b_{\text{out}},
  \qquad
  L_{\text{in}}(p) = A_{\text{in}} p + b_{\text{in}},
\end{gather*}
for $z \in \bb R^N$ and $p \in \bb R^{m_\alpha}$, the approximations of
$\mathbb F$ are
\begin{equation}
    \label{eq:NN-F}
  \mathbb{F}_\theta : \mathbb{R}^{m_\alpha} \to  \mathbb{R}^{m_h}, \qquad
  \mathbb{F}_\theta = L_{\text{out}} \circ \Phi_l \circ \dots \circ \Phi_1 \circ L_{\text{in}}
\end{equation}
where $\theta = \{A_{\text{in}}, b_{\text{in}}, A_m, W_m, b_m,
b_{\text{out}}, A_{\text{out}} : m = 1, \dots, l\}$ ranges over a set
$\Theta$.  Replacing $\bb F$ by $\bb F_\theta$ in~\eqref{eq:fun-map}
defines $\Fht : \cQ \to \Uah$, and
$\mathcal H(\Theta)=\{\Fht:\ \theta\in\Theta\}$.

The number of trainable weights in~\eqref{eq:NN-F} splits as
\begin{equation}
  \label{eq:ndpg-theta-split}
  |\theta| = \theta_{\rm fixed} + m_h\,(N+1),
  \qquad
  \theta_{\rm fixed} := N\,(m_\alpha+1) + l\,(2Nr+r),
\end{equation}
where $m_h(N+1)$ counts the weights of $L_{\rm out}$, one row of
$A_{\rm out}$ and one entry of $b_{\rm out}$ per degree of freedom of
$\Uah$, while $\theta_{\rm fixed}$ counts those of $L_{\rm in}$ and
$\Phi_1,\dots,\Phi_l$ and does not depend on the trial space.  In all
experiments $N=128$ and $r=32$, and from~\S\ref{ssec:ndpg-expressivity}
on also $l=5$, so $\theta_{\rm fixed}$ depends on $m_\alpha$ alone.
We keep  $N$, $r$ and $l$, and hence $\theta_{\rm fixed}$, fixed even when
$\Uah$ is enriched,  only to cap the computational cost.
The number of weights per degree of freedom,
$\theta_{\rm fixed}/m_h$, then shrinks as $h$ decreases, and
\S\ref{ssec:ndpg-ablation} checks whether this limits the attainable loss.

\subsection{General settings}
\label{ssec:num-protocol}

The theory asserts that the losses are reliable and efficient error
estimators.  The experiments quantify three things it leaves open: the
balance between spatial discretization and network size in the hybrid
hypothesis class; the dependence of the constants in~\eqref{vc} on the
parameter range and on the dimension of the parameter domain; and the
effect of minimizing over finite training sets.
Subsection~\ref{ssec:ndpg-expressivity} fixes the network size, while
\S\S\ref{ssec:num-ratio}--\ref{ssec:ndpg-tiling} measure the loss--error
relation, and \S\ref{ssec:ndpg-ablation} varies the training set and the
trial space.  Each table moves one quantity and holds the others at a
common configuration.  To this end only, we compute high-fidelity
reference solutions, which would not be available in practice; moderate
accuracy suffices.
\label{sssec:num-models}
\label{sssec:num-campaign}
The first model is the quasilinear flux law of
Example~\ref{ex:invertible},
$\sigma_\alpha(p)=-\alpha^{-1}(2+(1+|p|)^{-1})p$, with profile bounds
$\beta_1=2$, $\beta_2=3$.  Its inverse $\rho_\alpha=\sigma_\alpha^{-1}$ is
available in closed form, which makes the ultraweak form~\eqref{eq:DPG-A},
and hence the NDPG loss~\eqref{eq:DPG-loss-defn}, computable; its networks
produce the tuples $\clz_h=(q_h,u_h,\hat u_h,\hat q_{n,h})$
of~\eqref{eq:uw-discrete}.  The second is the Kirchhoff model with the
saturated Carrier law of Example~\ref{ex:mu-kirchhoff},
$M(s)=1+30s/(1+5s)$, so that $M_1=1$ and $M_2=7$; its networks produce the
pairs $\clz_h=(u_h,\hat q_{n,h})$ of~\eqref{eq:KH-discrete}.  There the
nonlinearity enters only through the scalar $m=M_\alpha(G_\alpha(u))$, so
by~\eqref{eq:KH-solution} a finite element solution costs one linear solve
and one scalar root-finding step, whereas the quasilinear model requires Newton's
method; it is the latter model that would make supervised training
expensive.  In both models $f=1$ on $\om=(0,1)^2$.  The parameter enters
as the resistivity $\alpha^{-1}$ in the flux model and as the diffusion
coefficient $\alpha$ in the Kirchhoff model, so the two are comparable
through the contrast defined below, not through individual values.

The sampling measure is the product uniform measure on
\begin{equation}
\label{eq:alphadists}
  \cQ=[\alpha_1,\alpha_2]^{m_\alpha},
  \qquad \alpha_1=0.1 \text{ fixed},\qquad
  \alpha_2\in\{0.1,\,0.5,\,1,\,2,\,3\},
\end{equation}
determined by the contrast $\varkappa=\alpha_2/\alpha_1\in\{1,5,10,20,30\}$.
At $\varkappa=1$, $\mu$ is the Dirac measure at $(0.1,\dots,0.1)$ and the
problem is not parametric; this case serves as a control.

Theorem~\ref{thm:vc-NDPG-example} requires the lowest order $k=0$, so that
the nonlinearity evaluated at a trial function is piecewise polynomial.
The Kirchhoff loss is covered for every degree by
Theorem~\ref{thm:vc-NDPG-Kirchhoff}, but we use $k=0$ there as well, so
that the two models are compared under identical conditions.  The test
norm is the product norm on $\Hdivoh\times H^1(\oh)$ in the ultraweak
model and the broken $H^1$ norm in the Kirchhoff model, neither depending
on $\alpha$.
In the Kirchhoff runs the broken test space is $P_{k+d+1}$, one degree
higher than the space $P_{k+d}$ the analysis requires, which the theory
permits.
Errors are measured in the trial norm $\|\cdot\|_{\Ua}$ of
each formulation: the $L_2$ norm for the ultraweak field variables, the
$H^1$ seminorm for the Kirchhoff field, and for the traces the norms
\[
  \|\hat v_h\|_{1/2,h}^2
  :=\|v_h\|_{L_2}^2+\|\nabla v_h\|_{L_2}^2,
  \qquad
  \|\hat r_{n,h}\|_{-1/2,h}^2
  :=\|r_h\|_{L_2}^2+\|\dive r_h\|_{L_2}^2,
\]
where $v_h$ and $r_h$ are extensions of minimal norm of $\hat v_h$ and
$\hat r_{n,h}$ within the Lagrange and Raviart--Thomas spaces of the mesh
and degree at hand.  Since the quotient norms~\eqref{eq:quotient-norms}
are infima of $\|\nabla v\|_{L_2}$ and $\|r\|_{\Hdiv}$ over all
extensions, these computed trace norms bound those of the theory from
above.  We write $\|\cdot\|_{\Ua}$ also for the resulting computed
product norm, in which all errors and ratios below are reported.

Following the ablation study for the linear counterpart in~\cite{CDG},
all networks use $N=128$ and $r=32$ in~\eqref{eq:NN-F} and are trained by
Adam with learning rate $10^{-4}$ and batch size $32$.  A further $500$
parameters form a validation set that monitors generalization but never
enters training.  Unless stated otherwise, networks are trained for
$10\,000$ epochs, so that differences between outcomes cannot be
attributed to the optimization effort; the studies
of~\S\ref{ssec:ndpg-expressivity} and~\S\ref{ssec:ndpg-ablation}, where
that effort is itself examined, run to $25\,000$ epochs, and the network
of Figure~\ref{fig:exp2-alpha} to $20\,000$.  All quantities are evaluated
on one test set
$\widehat\cQ_{\mathrm{test}}=\{\alpha_{[1]},\dots,\alpha_{[1024]}\}$ drawn
independently of the training and validation samples (the bracketed
subscript enumerates draws, the superscript in~\eqref{eq:Qdefn}
subdomains).
All experiments in this section use one fixed random seed for the training
samples and the network initialization, and one fixed random seed for
$\widehat\cQ_{\rm test}$; hence, for a given tiling and contrast, both models
are evaluated on the same parameters.
\label{sssec:num-heldout}
\label{sssec:num-reference}
\label{sssec:num-ratio}

\subsection{Network complexity}
\label{ssec:ndpg-expressivity}

Since it conditions every later experiment, we first fix the size of the
network.  Larger networks generally need more epochs to realize their
potential, so we vary the depth $l$ at fixed $N$ and $r$ and record the
cumulative minimum of the test loss over the epochs.  We do this in the
most demanding configuration considered: the $5\times5$ tiling with
contrast $\varkappa=30$, mesh size $h=0.1$ and $n_{\rm train}=4096$.
Precisely, we record
\begin{equation}
  \label{eq:cmin}
  \mathrm{cmin}(e) := \min_{e'\le e}\,
  \Big( \frac{1}{\#\widehat\cQ_{\rm test}}
  \sum_{\alpha\in\widehat\cQ_{\rm test}}
  \Jaa{\alpha}{\clz^{(e')}(\alpha)} \Big)^{1/2},
\end{equation}
where $\clz^{(e)}$ is the network at the checkpoint saved after $e$
epochs.  By construction $\mathrm{cmin}$ is
nonincreasing and becomes flat exactly where no later checkpoint improves on
the best value reached so far.
Figure~\ref{fig:ndpg-depth} shows that deeper networks reach a given level
in fewer epochs: the curve for $l=1$ remains constant over a large part of
the epoch axis and requires more training to reach new minima.
The depth $l=5$ is held fixed in all subsequent experiments.

\begin{figure}[htbp]
  \centering
  \begin{tikzpicture}
    \begin{groupplot}[s6ablation,
        log ticks with fixed point,
        ytick={0.01,0.02,0.03,0.05,0.1,0.2,0.3,0.5},
        group style={group size=2 by 1, horizontal sep=1.5cm,
                     ylabels at=edge left},
        xlabel={Epochs}, ylabel={$\widehat\cR^{1/2}$},
        xmode=normal, ymode=log,
        scaled x ticks=false, xtick distance=5000,
        xmin=0, xmax=11000,
        restrict x to domain=0:10000,
        x tick label style={font=\scriptsize},
        legend style={at={(0.97,0.97)}, anchor=north east},
        legend columns=1]
      \nextgroupplot[title={Invertible nonlinearity}]
      \addplot[s6best, depthA, mark=*] table[x=epoch, y=l1best] {data/ndpg_depth_ndpg.dat};
      \addplot[s6best, depthB, mark=square*] table[x=epoch, y=l3best] {data/ndpg_depth_ndpg.dat};
      \addplot[s6best, depthC, mark=triangle*] table[x=epoch, y=l5best] {data/ndpg_depth_ndpg.dat};
      \addplot[s6best, depthD, mark=diamond*] table[x=epoch, y=l7best] {data/ndpg_depth_ndpg.dat};
      \legend{$l=1$,$l=3$,$l=5$,$l=7$}
      \nextgroupplot[title={Kirchhoff model}]
      \addplot[s6best, depthA, mark=*] table[x=epoch, y=l1best] {data/ndpg_depth_kirchhoff.dat};
      \addplot[s6best, depthB, mark=square*] table[x=epoch, y=l3best] {data/ndpg_depth_kirchhoff.dat};
      \addplot[s6best, depthC, mark=triangle*] table[x=epoch, y=l5best] {data/ndpg_depth_kirchhoff.dat};
      \addplot[s6best, depthD, mark=diamond*] table[x=epoch, y=l7best] {data/ndpg_depth_kirchhoff.dat};
      \legend{$l=1$,$l=3$,$l=5$,$l=7$}
    \end{groupplot}
  \end{tikzpicture}
  \caption{The cumulative minimum~\eqref{eq:cmin} versus the number of training epochs, for
  depths $l\in\{1,3,5,7\}$ at the $5\times5$ tiling, contrast
  $\varkappa=30$, and first $10\,000$ epochs.  The widths $N=128$ and $r=32$, the training set
  $n_{\rm train}=4096$ and the mesh $h=0.1$ are held fixed.}
  \label{fig:ndpg-depth}
\end{figure}
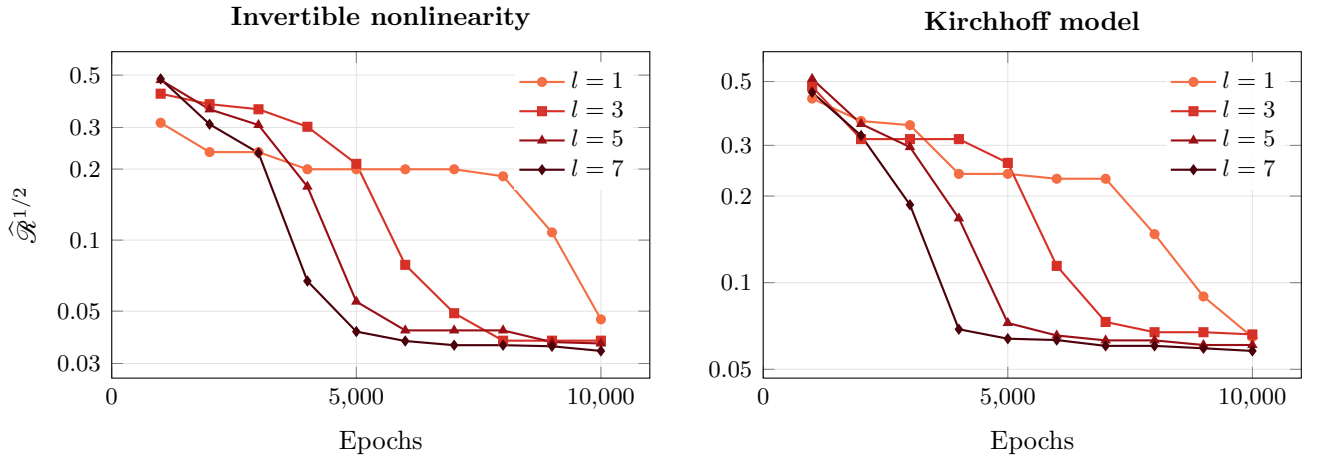

\subsection{Residual-error ratios}
\label{ssec:num-heldout}
\label{ssec:num-reference}
\label{ssec:num-ratio}

\begin{table}[htbp]
  \centering
  \footnotesize
  \setlength{\tabcolsep}{4pt}
  \renewcommand{\arraystretch}{0.95}
  \caption{Varying contrast for the $2\times2$ tiling; upper block: the invertible nonlinearity;
  lower block: the Kirchhoff model; mesh size $h=0.1$, degree $k=0$, number of epochs $10\,000$, $n_{\rm train}=4096$, $\#\widehat\cQ_{\rm test}=1024$, proportionality confidence parameter $K_{\rm est}$ defined in \eqref{eq:num-Kest}.}
  \label{tab:num-contrast}
  \begin{tabular}{l|ccccc}
    \hline
    & \multicolumn{5}{c}{$\varkappa$} \\
    \cline{2-6}
    & $1$ & $5$ & $10$ & $20$ & $30$ \\
    \hline
    \multicolumn{6}{@{}l}{\itshape Invertible nonlinearity}\\
    \noalign{\smallskip}
    $\bar{\mathsf r}_{\rm est}$ & $0.113$ & $0.127$ & $0.115$ & $0.166$ & $0.204$ \\
    $\widehat\cR^{\,\rm NN}$ & $2.41\times10^{-4}$ & $3.12\times10^{-4}$ & $3.64\times10^{-4}$ & $4.41\times10^{-4}$ & $5.44\times10^{-4}$ \\
    $\widehat\cR^{\,\rm FEM}$ & $2.41\times10^{-4}$ & $2.51\times10^{-4}$ & $2.79\times10^{-4}$ & $3.72\times10^{-4}$ & $4.85\times10^{-4}$ \\
    $K_{\rm est}$ & $53.0$ & $41.5$ & $28.3$ & $20.7$ & $16.7$ \\
    \hline
    \multicolumn{6}{@{}l}{\itshape Kirchhoff model}\\
    \noalign{\smallskip}
    $\bar{\mathsf r}_{\rm est}$ & $0.085$ & $0.302$ & $0.480$ & $0.647$ & $0.811$ \\
    $\widehat\cR^{\,\rm NN}$ & $1.33\times10^{-3}$ & $1.53\times10^{-3}$ & $1.62\times10^{-3}$ & $1.84\times10^{-3}$ & $1.84\times10^{-3}$ \\
    $\widehat\cR^{\,\rm FEM}$ & $1.33\times10^{-3}$ & $1.45\times10^{-3}$ & $1.53\times10^{-3}$ & $1.58\times10^{-3}$ & $1.60\times10^{-3}$ \\
    $K_{\rm est}$ & $43.9$ & $32.2$ & $24.6$ & $19.7$ & $17.9$ \\
    \hline
  \end{tabular}
\end{table}
For each test parameter $\alpha$ let $\clz_h^{\mathrm{NN}}(\alpha)=\Fht(\alpha)$
and $\clz_h^{\mathrm{FEM}}(\alpha)$ denote the network prediction and the
finite element solution in the same trial space $\Uah$.  The latter is
computed by Newton's method in the quasilinear model, converging for every
$\alpha$, and by scalar root finding in the Kirchhoff model.  We compare
the mean losses of the two states over the test set,
\begin{equation}
  \label{eq:num-heldout-residuals}
  \widehat\cR^{\,s}
  :=\frac1{\#\widehat\cQ_{\mathrm{test}}}
       \sum_{\alpha\in\widehat\cQ_{\mathrm{test}}}
       \Jaa{\alpha}{\clz_h^{s}(\alpha)},
  \qquad s\in\{\mathrm{NN},\mathrm{FEM}\},
\end{equation}
the empirical counterparts of
$\bb E_{\alpha \sim \mu} [ \Jaa{\alpha}{\clz(\alpha)}]$.  Since
$\clz_h^{\mathrm{FEM}}(\alpha)$ minimizes the loss over $\Uah$, we have
$\widehat\cR^{\,\rm NN}\ge\widehat\cR^{\,\rm FEM}$, and the gap measures
what the training leaves unrealized.

We wish to quantify the proportionality constants between losses and
errors, in particular their dependence on the parameter range.  Summing
the per-$\alpha$ estimates of Theorems~\ref{thm:vc-NDPG-example}
and~\ref{thm:vc-NDPG-Kirchhoff} over $\widehat\cQ_{\rm test}$ gives
$c'\le\bar{\mathsf r}\le C'$ for the ratio
$\bar{\mathsf r}:=\widehat\cR^{\,\rm NN}/\widehat{\cl E}(\clu-\clz_h^{\rm NN})$,
where
\begin{equation}
    \label{Ehat}
  \widehat{\cl E}(\clw):= \frac{1}{\#\widehat\cQ_{\rm test}}
  \sum_{\alpha\in\widehat\cQ_{\rm test}} \|\clw(\alpha)\|^2_\Ua .
\end{equation}
The constants $c'$, $C'$ depend on the endpoints $\alpha_1$, $\alpha_2$
of~\eqref{eq:alphadists}.  For the flux~\eqref{eq:5},
Corollary~\ref{cor:flux-instance} gives
$\lambda=C_\gamma=C_\gamma^-=0$, $\mu_1=c_0=\beta_1/\alpha_2$ and
$\mu_2=\beta_2/\alpha_1$.  With $\beta_1$, $\beta_2$, $\alpha_1$ fixed,
only $\mu_1$ varies, decreasing as $\varkappa$ grows, so $\CdB$
in~\eqref{eq:uw-Cd} grows like $\mu_1^{-1}$ and $\cdB$
in~\eqref{eq:uw-cd} shrinks.  For the Kirchhoff model, $\kappa_1=\alpha_1$
and $\kappa_2=\alpha_2$ in~\eqref{eq:KH-cK}, so $c_K$ is fixed while
$C_K$ grows linearly with $\alpha_2$.  In both models the range
guaranteed for $\bar{\mathsf r}$ therefore widens with the contrast.

Since $\clu(\alpha)$ is not available, we replace it by the solution
$\clu_{h_*}= \clu_{h_*}(\alpha)$ of degree $k=1$ on one uniform refinement of the mesh,
$h_*=0.05$, and report
\begin{equation}
  \label{eq:ndpg-ratio}
  \bar{\mathsf r}_{\rm est}
  :=\frac{\widehat\cR^{\,\rm NN}}
           {\widehat{\cl E}(\clu_{h_*}-\clz_h^{\rm NN})}.
\end{equation}
To assess this substitution, let $\clu_{\rm ref}$ be the solution of
degree $k=1$ on two uniform refinements of the mesh, let
$\widehat\cQ_{\rm sub}\subset\widehat\cQ_{\rm test}$ consist of the first
twelve test parameters, and set
\begin{equation}
  \label{eq:num-Kest}
  K_{\rm est}:=\min_{\alpha\in\widehat\cQ_{\rm sub}}
  \frac{\|\clu_{h_*}(\alpha)-\clz_h^{\rm NN}(\alpha)\|_{\Ua}}
       {\|\clu_{h_*}(\alpha)-\clu_{\rm ref}(\alpha)\|_{\Ua}}.
\end{equation}
(The meshes are nested and the degrees raised, so transfers between them
are exact to machine precision.)
If saturation holds, i.e.,
$\|\clu-\clu_{\rm ref}\|_{\Ua}\le\rho\,\|\clu-\clu_{h_*}\|_{\Ua}$ for some
$\rho<1$, the triangle inequality gives
$\|\clu-\clu_{h_*}\|_{\Ua}\le C_{\rm sat}\,\|\clu_{h_*}-\clu_{\rm ref}\|_{\Ua}$
with $C_{\rm sat}:=(1-\rho)^{-1}$, hence
\begin{equation}
  \label{eq:K-est-perturbation}
  \bigg|\frac{\|\clu(\alpha)-\clz_h^{\rm NN}(\alpha)\|_{\Ua}}
             {\|\clu_{h_*}(\alpha)-\clz_h^{\rm NN}(\alpha)\|_{\Ua}}-1\bigg|
  \le\delta:=\frac{C_{\rm sat}}{K_{\rm est}},
  \qquad \alpha\in\widehat\cQ_{\rm sub}.
\end{equation}
Were this true on all of $\widehat\cQ_{\rm test}$ with $\delta<1$, summing
squares would place $\bar{\mathsf r}$ between
$(1+\delta)^{-2}\,\bar{\mathsf r}_{\rm est}$ and
$(1-\delta)^{-2}\,\bar{\mathsf r}_{\rm est}$.  The larger $K_{\rm est}$,
the smaller the effect of replacing $\clu$ by $\clu_{h_*}$.  Since
$C_{\rm sat}$ is unknown and the check covers $\widehat\cQ_{\rm sub}$
only, this is a plausibility check, not a bound.

Table~\ref{tab:num-contrast} varies the contrast at the $2\times2$ tiling.
In the control column $\varkappa=1$, where $\alpha\equiv\alpha_1$, the two
losses agree in both models.  For $\varkappa>1$, $\widehat\cR^{\,\rm FEM}$
is consistently lower, but $\widehat\cR^{\,\rm NN}$ exceeds it by at most
$30\%$: measured through the loss, the training realizes most of the
accuracy the trial space affords.  Since $K_{\rm est}\ge 16$ throughout,
though decreasing as the contrast grows, $\bar{\mathsf r}_{\rm est}$ is a
reliable proxy for $\bar{\mathsf r}$.  It grows with the contrast, by less
than a factor of two in the flux model and monotonically, by a factor of
about ten, in the Kirchhoff model.  This is consistent with the widening
of the range permitted by Theorems~\ref{thm:vc-NDPG-example}
and~\ref{thm:vc-NDPG-Kirchhoff} and, in the Kirchhoff model, with the
growth of $C_K$ with $\alpha_2$.

\subsection{Dimension of the parameter domain}
\label{ssec:ndpg-tiling}

\begin{table}[htbp]
  \centering
  \footnotesize
  \setlength{\tabcolsep}{4pt}
  \renewcommand{\arraystretch}{0.95}
  \caption{Varying the tiling at $\varkappa=5$; upper block: invertible nonlinearity; lower block: Kirchhoff model;
  mesh size $h=0.1$, degree $k=0$, number of epochs $10\,000$, $n_{\rm train}=4096$, $\#\widehat\cQ_{\rm test}=1024$, $K_{\rm est}$ over the twelve parameters of $\widehat\cQ_{\rm sub}$.}
  \label{tab:num-tiling}
  \begin{tabular}{l|cccc}
    \hline
    & \multicolumn{4}{c}{tiling} \\
    \cline{2-5}
    & $2\times2$ & $3\times3$ & $4\times4$ & $5\times5$ \\
    \hline
    \multicolumn{5}{@{}l}{\itshape Invertible nonlinearity}\\
    \noalign{\smallskip}
    $\bar{\mathsf r}_{\rm est}$ & $0.127$ & $0.122$ & $0.104$ & $0.108$ \\
    $\widehat\cR^{\,\rm NN}$ & $3.12\times10^{-4}$ & $3.77\times10^{-4}$ & $3.57\times10^{-4}$ & $3.96\times10^{-4}$ \\
    $\widehat\cR^{\,\rm FEM}$ & $2.51\times10^{-4}$ & $2.78\times10^{-4}$ & $2.78\times10^{-4}$ & $2.88\times10^{-4}$ \\
    $K_{\rm est}$ & $41.5$ & $15.4$ & $15.5$ & $12.6$ \\
    \hline
    \multicolumn{5}{@{}l}{\itshape Kirchhoff model}\\
    \noalign{\smallskip}
    $\bar{\mathsf r}_{\rm est}$ & $0.302$ & $0.290$ & $0.299$ & $0.317$ \\
    $\widehat\cR^{\,\rm NN}$ & $1.53\times10^{-3}$ & $2.01\times10^{-3}$ & $2.02\times10^{-3}$ & $2.29\times10^{-3}$ \\
    $\widehat\cR^{\,\rm FEM}$ & $1.45\times10^{-3}$ & $1.72\times10^{-3}$ & $1.65\times10^{-3}$ & $1.98\times10^{-3}$ \\
    $K_{\rm est}$ & $32.2$ & $24.3$ & $11.1$ & $10.8$ \\
    \hline
  \end{tabular}
\end{table}

Table~\ref{tab:num-tiling} fixes $\varkappa=5$ and refines the tiling from
$2\times2$ to $5\times5$, i.e., $m_\alpha$ from $4$ to $25$, keeping
$h=0.1$ and $k=0$ (the mesh being adapted to resolve each tiling).
The ratio $\bar{\mathsf r}_{\rm est}$ stays within $[0.104, 0.127]$ in the
flux model and $[0.290, 0.317]$ in the Kirchhoff model, while
$K_{\rm est}\ge 10$: the loss--error relation is insensitive to the
dimension of the parameter domain, in agreement with the theorems, whose
constants depend on the parameter range but not on $m_\alpha$.  The
network loss $\widehat\cR^{\,\rm NN}$ grows mildly with $m_\alpha$, by
$27\%$ in the flux model and $50\%$ in the Kirchhoff model from the
coarsest to the finest tiling, and $\widehat\cR^{\,\rm FEM}$ grows less.
A finer tiling gives the coefficient more interfaces across which it
jumps, so the solutions are less regular and the trial space at fixed
$h$ and $k$ approximates them less well.  With the training set, the network and the
epochs held fixed, the more complex parameter dependence also widens the
gap between $\widehat\cR^{\,\rm NN}$ and $\widehat\cR^{\,\rm FEM}$, but
only moderately: their ratio stays below $1.4$, so the spatial and parametric
resolution are not yet out of balance.

\subsection{Balancing the training set and the trial space}
\label{ssec:ndpg-ablation}

We return to the $2\times2$ tiling and $\varkappa=5$, the configuration
common to Tables~\ref{tab:num-contrast} and~\ref{tab:num-tiling}, with
$k=0$, $l=5$ and $25\,000$ epochs.
By Theorems~\ref{thm:vc-NDPG-example}
and~\ref{thm:vc-NDPG-Kirchhoff} the constants in~\eqref{vc} depend on
neither $h$ nor $n_{\rm train}$, and
\S\S\ref{ssec:num-ratio}--\ref{ssec:ndpg-tiling} have quantified them, so
the losses can be read directly as error measures.  The training is
therefore assessed by comparing $\widehat\cR^{\,\rm NN}$ with the
training-independent $\widehat\cR^{\,\rm FEM}$ of the same trial space; no
reference solution is needed.  First, $h=0.1$ is fixed and $n_{\rm train}$
varied from $128$ to $4096$ (Table~\ref{tab:num-trainset}).

\begin{table}[htbp]
  \centering
  \footnotesize
  \setlength{\tabcolsep}{4pt}
  \renewcommand{\arraystretch}{0.95}
  \caption{Varying the training set size at the $2\times2$ tiling, $\varkappa=5$, $h=0.1$, $k=0$, $l=5$, number of epochs $25\,000$, $\#\widehat\cQ_{\rm test}=1024$.  The last row gives the loss of the finite element solution in the same trial space, which does not depend on training.}
  \label{tab:num-trainset}
  \begin{tabular}{r|cc|cc}
    \hline
    & \multicolumn{2}{c|}{Invertible nonlinearity} & \multicolumn{2}{c}{Kirchhoff model} \\
    $n_{\rm train}$ & $\widehat\cR^{\,\rm NN}$ & $\widehat\cR^{\,\rm NN}/\widehat\cR^{\,\rm FEM}$ & $\widehat\cR^{\,\rm NN}$ & $\widehat\cR^{\,\rm NN}/\widehat\cR^{\,\rm FEM}$ \\
    \hline
    $128$  & $2.29\times10^{-2}$ & $91.3$ & $8.67\times10^{-3}$ & $5.97$ \\
    $256$  & $5.17\times10^{-3}$ & $20.6$ & $2.67\times10^{-2}$ & $18.4$ \\
    $512$  & $8.74\times10^{-4}$ & $3.48$ & $2.69\times10^{-3}$ & $1.85$ \\
    $1024$ & $8.10\times10^{-4}$ & $3.22$ & $2.12\times10^{-3}$ & $1.46$ \\
    $2048$ & $2.61\times10^{-4}$ & $1.04$ & $1.53\times10^{-3}$ & $1.06$ \\
    $4096$ & $2.52\times10^{-4}$ & $1.00$ & $1.47\times10^{-3}$ & $1.01$ \\
    \hline
    $\widehat\cR^{\,\rm FEM}$ & $2.51\times10^{-4}$ & & $1.45\times10^{-3}$ & \\
    \hline
  \end{tabular}
\end{table}

As $n_{\rm train}$ grows, the ratio
$\widehat\cR^{\,\rm NN}/\widehat\cR^{\,\rm FEM}$ approaches $1$ in both
models: with enough samples the network attains, within the allotted
epochs, the loss of the discrete minimizer of its trial space.  A smaller
training set leaves the optimization short of that level, and the deficit
persists under further training, as the upper row of
Figure~\ref{fig:ndpg-c5-cmin} shows.  No rate of decay in $n_{\rm train}$ is to be expected:
$\widehat\cR^{\,\rm NN}$ approaches the expected loss of the trained
network, which is positive for any finitely parametrized approximation.

\begin{table}[htbp]
  \centering
  \footnotesize
  \setlength{\tabcolsep}{3pt}
  \renewcommand{\arraystretch}{0.95}
  \caption{Prediction and finite element losses for decreasing mesh sizes, fixed
  $n_{\rm train}=4096$, $k=0$, $l=5$ and number of epochs $25\,000$, $\#\widehat\cQ_{\rm test}=1024$.  The columns show the number of weights per spatial degree of freedom~\eqref{eq:ndpg-theta-split}, with $	\theta_{\rm fixed}=41\,760$.}
  \label{tab:num-mesh}
  \begin{tabular}{r|ccc|ccc}
    \hline
    & \multicolumn{3}{c|}{Invertible nonlinearity} & \multicolumn{3}{c}{Kirchhoff model} \\
    & $\theta_{\rm fixed}/m_h$ & $\widehat\cR^{\,\rm NN}$ & $\widehat\cR^{\,\rm FEM}$ & $\theta_{\rm fixed}/m_h$ & $\widehat\cR^{\,\rm NN}$ & $\widehat\cR^{\,\rm FEM}$ \\
    \hline
    \multicolumn{7}{c}{mesh size $h$, at $n_{\rm train}=4096$} \\
    \hline
    $0.2$ & $148.6$ & $1.05\times10^{-3}$ & $1.05\times10^{-3}$ & $369.6$ & $5.45\times10^{-3}$ & $5.44\times10^{-3}$ \\
    $0.12$ & $55.6$ & $3.72\times10^{-4}$ & $3.70\times10^{-4}$ & $138.7$ & $1.99\times10^{-3}$ & $1.99\times10^{-3}$ \\
    $0.1$ & $37.9$ & $2.52\times10^{-4}$ & $2.51\times10^{-4}$ & $94.7$ & $1.47\times10^{-3}$ & $1.45\times10^{-3}$ \\
    $0.08$ & $24.1$ & $1.58\times10^{-4}$ & $1.55\times10^{-4}$ & $60.3$ & $8.65\times10^{-4}$ & $8.54\times10^{-4}$ \\
    \hline
  \end{tabular}
\end{table}

Next, $n_{\rm train}=4096$ is fixed and $h$ decreased.  Since
$\theta_{\rm fixed}$ is fixed, the number of weights per degree of
freedom, $\theta_{\rm fixed}/m_h$, shrinks (Table~\ref{tab:num-mesh}).
Both losses decay like $h^2$, with fitted exponents
$\cfiveSlopeMeshNdpgR$ for $\widehat\cR^{\,\rm NN}$ and
$\cfiveSlopeMeshNdpgRfem$ for $\widehat\cR^{\,\rm FEM}$ in the flux model,
and $\cfiveSlopeMeshKirchhoffR$ and $\cfiveSlopeMeshKirchhoffRfem$ in the
Kirchhoff model, and $\widehat\cR^{\,\rm NN}$ exceeds
$\widehat\cR^{\,\rm FEM}$ by at most about $2\%$, at the finest mesh of
the flux model.  The network thus keeps pace with the finite element
solution down to the smallest $h$, even though the number of weights per
degree of freedom drops by a factor of about six over the table; how far
this persists under further refinement, before larger $N$, $r$ or $l$
become necessary, is not tested here.  The Kirchhoff trial space~\eqref{eq:KH-discrete} has two
components where the ultraweak one~\eqref{eq:uw-discrete} has four, so its
$\theta_{\rm fixed}/m_h$ is larger on every mesh.  The lower row of
Figure~\ref{fig:ndpg-c5-cmin} shows the training effort this costs: each
curve reaches the level $(\widehat\cR^{\,\rm FEM})^{1/2}$ of its own
trial space (dashed), and the finer the mesh, the later it does so.

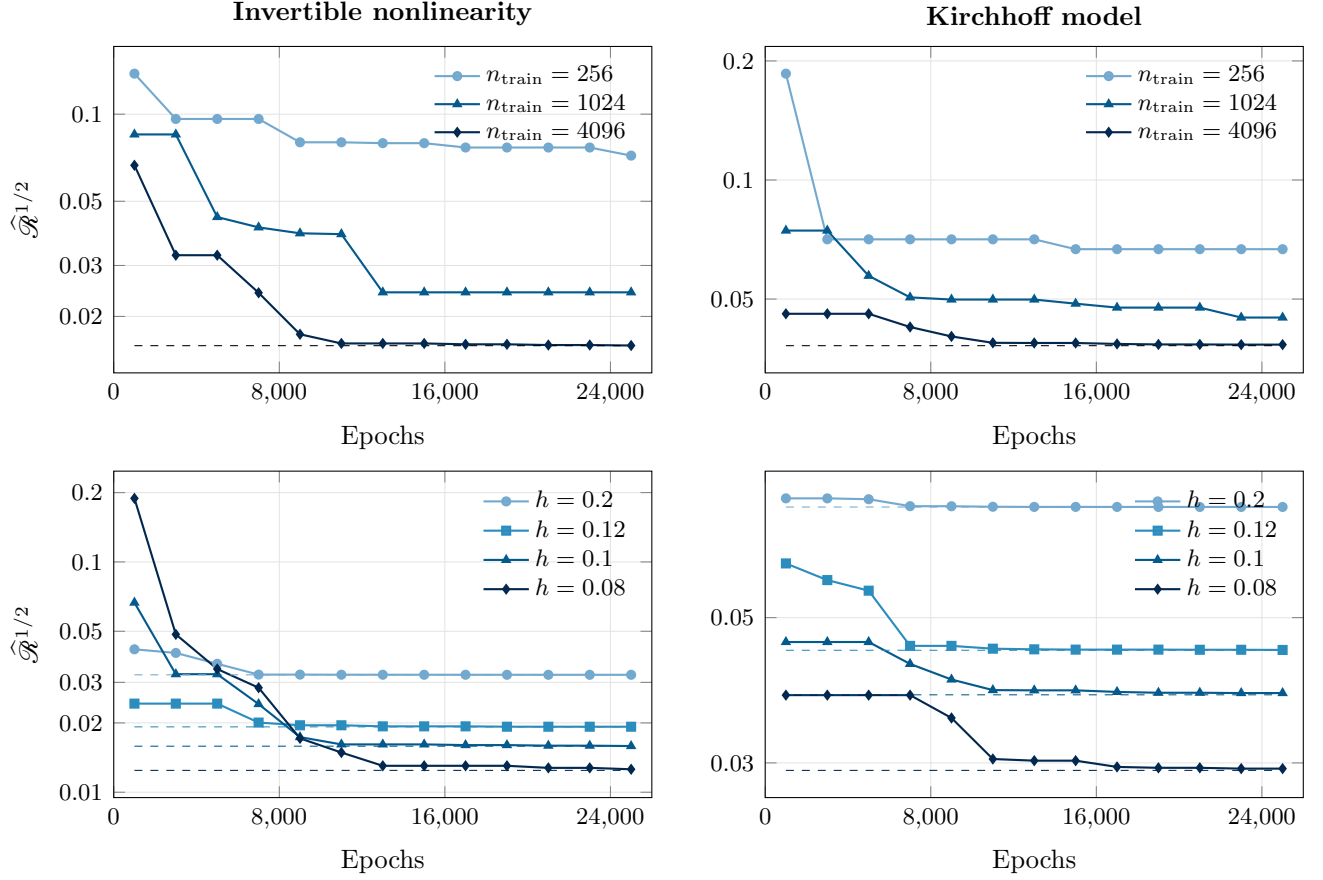
\begin{figure}[htbp]
  \centering
  \begin{tikzpicture}
    \begin{groupplot}[s6ablation,
        log ticks with fixed point,
        ytick={0.01,0.02,0.03,0.05,0.1,0.2,0.3,0.5},
        group style={group size=2 by 2, horizontal sep=1.5cm,
                     vertical sep=1.3cm,
                     ylabels at=edge left},
        legend style={font=\scriptsize, draw=none, fill=none},
        legend columns=1]
      \nextgroupplot[title={Invertible nonlinearity}, xlabel={Epochs}, ylabel={$\widehat\cR^{1/2}$},
        xmode=normal, ymode=log, scaled x ticks=false,
        xtick distance=8000, xmin=0, xmax=26000,
        x tick label style={font=\scriptsize},
        legend style={at={(0.97,0.97)}, anchor=north east}]
      \addplot[s6best, depthAb, mark=*] table[x=epoch, y=n256] {data/ndpg_c5_trainsetcmin_ndpg.dat};
      \addplot[s6best, depthCb, mark=triangle*] table[x=epoch, y=n1024] {data/ndpg_c5_trainsetcmin_ndpg.dat};
      \addplot[s6best, depthDb, mark=diamond*] table[x=epoch, y=n4096] {data/ndpg_c5_trainsetcmin_ndpg.dat};
      \legend{$n_{\rm train}=256$,$n_{\rm train}=1024$,$n_{\rm train}=4096$}
      \addplot[dashed, black, forget plot] table[x=epoch, y=n256fem] {data/ndpg_c5_trainsetcmin_ndpg.dat};
      \nextgroupplot[title={Kirchhoff model}, xlabel={Epochs},
        xmode=normal, ymode=log, scaled x ticks=false,
        xtick distance=8000, xmin=0, xmax=26000,
        x tick label style={font=\scriptsize},
        legend style={at={(0.97,0.97)}, anchor=north east}]
      \addplot[s6best, depthAb, mark=*] table[x=epoch, y=n256] {data/ndpg_c5_trainsetcmin_kirchhoff.dat};
      \addplot[s6best, depthCb, mark=triangle*] table[x=epoch, y=n1024] {data/ndpg_c5_trainsetcmin_kirchhoff.dat};
      \addplot[s6best, depthDb, mark=diamond*] table[x=epoch, y=n4096] {data/ndpg_c5_trainsetcmin_kirchhoff.dat};
      \legend{$n_{\rm train}=256$,$n_{\rm train}=1024$,$n_{\rm train}=4096$}
      \addplot[dashed, black, forget plot] table[x=epoch, y=n256fem] {data/ndpg_c5_trainsetcmin_kirchhoff.dat};
      \nextgroupplot[xlabel={Epochs}, ylabel={$\widehat\cR^{1/2}$},
        xmode=normal, ymode=log, scaled x ticks=false,
        xtick distance=8000, xmin=0, xmax=26000,
        x tick label style={font=\scriptsize},
        legend style={at={(0.97,0.97)}, anchor=north east}]
      \addplot[s6best, depthAb, mark=*] table[x=epoch, y=h02] {data/ndpg_c5_meshcmin_ndpg.dat};
      \addplot[s6best, depthBb, mark=square*] table[x=epoch, y=h012] {data/ndpg_c5_meshcmin_ndpg.dat};
      \addplot[s6best, depthCb, mark=triangle*] table[x=epoch, y=h01] {data/ndpg_c5_meshcmin_ndpg.dat};
      \addplot[s6best, depthDb, mark=diamond*] table[x=epoch, y=h008] {data/ndpg_c5_meshcmin_ndpg.dat};
      \legend{$h=0.2$,$h=0.12$,$h=0.1$,$h=0.08$}
      \addplot[dashed, depthAb, forget plot] table[x=epoch, y=h02fem] {data/ndpg_c5_meshcmin_ndpg.dat};
      \addplot[dashed, depthBb, forget plot] table[x=epoch, y=h012fem] {data/ndpg_c5_meshcmin_ndpg.dat};
      \addplot[dashed, depthCb, forget plot] table[x=epoch, y=h01fem] {data/ndpg_c5_meshcmin_ndpg.dat};
      \addplot[dashed, depthDb, forget plot] table[x=epoch, y=h008fem] {data/ndpg_c5_meshcmin_ndpg.dat};
      \nextgroupplot[xlabel={Epochs},
        xmode=normal, ymode=log, scaled x ticks=false,
        xtick distance=8000, xmin=0, xmax=26000,
        x tick label style={font=\scriptsize},
        legend style={at={(0.97,0.97)}, anchor=north east}]
      \addplot[s6best, depthAb, mark=*] table[x=epoch, y=h02] {data/ndpg_c5_meshcmin_kirchhoff.dat};
      \addplot[s6best, depthBb, mark=square*] table[x=epoch, y=h012] {data/ndpg_c5_meshcmin_kirchhoff.dat};
      \addplot[s6best, depthCb, mark=triangle*] table[x=epoch, y=h01] {data/ndpg_c5_meshcmin_kirchhoff.dat};
      \addplot[s6best, depthDb, mark=diamond*] table[x=epoch, y=h008] {data/ndpg_c5_meshcmin_kirchhoff.dat};
      \legend{$h=0.2$,$h=0.12$,$h=0.1$,$h=0.08$}
      \addplot[dashed, depthAb, forget plot] table[x=epoch, y=h02fem] {data/ndpg_c5_meshcmin_kirchhoff.dat};
      \addplot[dashed, depthBb, forget plot] table[x=epoch, y=h012fem] {data/ndpg_c5_meshcmin_kirchhoff.dat};
      \addplot[dashed, depthCb, forget plot] table[x=epoch, y=h01fem] {data/ndpg_c5_meshcmin_kirchhoff.dat};
      \addplot[dashed, depthDb, forget plot] table[x=epoch, y=h008fem] {data/ndpg_c5_meshcmin_kirchhoff.dat};
    \end{groupplot}
  \end{tikzpicture}
  \caption{The cumulative minimum~\eqref{eq:cmin} over the epochs at the
  $2\times2$ tiling, $\varkappa=5$, $k=0$, $l=5$.  Top: three training set
  sizes at $h=0.1$, against the dashed level
  $(\widehat\cR^{\,\mathrm{FEM}})^{1/2}$ of the shared trial space.  Bottom:
  four mesh sizes at $n_{\rm train}=4096$, each against the dashed level of
  its own trial space.}
  \label{fig:ndpg-c5-cmin}
\end{figure}

\section*{Acknowledgments}

This work was supported in part by NSF Grants DMS-2245077, DMS-2513112, DMS-2245097, DMS-2038080,   and by the German Research Foundation grant CRC 1481.

\bibliography{ndpg}
\bibliographystyle{siam}
\end{document}

%% file: preamble.tex
\usepackage{calc}
\usepackage[margin=2.5cm]{geometry}
\usepackage[bbgreekl]{mathbbol}
\usepackage[cal=boondox]{mathalpha}
\usepackage{hyperref}
\usepackage{enumitem}
\usepackage{amssymb}
\usepackage{pgfplots}
\usepgfplotslibrary{groupplots}
\pgfplotsset{compat=1.17}

\newtheorem{theorem}{Theorem}[section]
\newtheorem{lemma}[theorem]{Lemma}

\newtheorem{corollary}[theorem]{Corollary}
\newtheorem{proposition}[theorem]{Proposition}

\theoremstyle{definition}
\newtheorem{assumption}[theorem]{Assumption}
\newtheorem{hypothesis}[theorem]{Hypothesis}
\newtheorem{definition}[theorem]{Definition}
\newtheorem{example}[theorem]{Example}

\theoremstyle{remark}
\newtheorem{remark}[theorem]{Remark}

\newcommand{\bb}[1]{\mathbb{{#1}}}
\newcommand{\cl}[1]{\mathcal{{#1}}}
\let\div\relax
\DeclareMathOperator{\tr}{{\ensuremath{\text{tr}}}}
\newcommand{\trn}{\mathop{\mathrm{tr}_n}}
\DeclareMathOperator{\osc}{{\ensuremath{\text{osc}}}}
\DeclareMathOperator{\supp}{{\ensuremath{\text{supp}}}}
\DeclareMathOperator{\div}{{\ensuremath{\text{div}}}}

\DeclareMathOperator{\grad}{{\ensuremath{\text{grad}}}}

\DeclareMathOperator*{\argmin}{arg\,min}

\DeclareMathOperator*{\dive}{div}

\DeclareFontFamily{U}{mathx}{\hyphenchar\font45}
\DeclareFontShape{U}{mathx}{m}{n}{
<-6> mathx5 <6-7> mathx6 <7-8> matha7
<8-9> mathx8 <9-10> mathx9
<10-12> mathx10 <12-> mathx12
}{}
\DeclareSymbolFont{mathx}{U}{mathx}{m}{n}
\DeclareMathSymbol{\bigplus}{\mathop}{mathx}{"90}
\DeclareMathSymbol{\bigtimes}{\mathop}{mathx}{"91}

\newcommand{\veps}{\varepsilon}

\newcommand{\ip}[1]{\langle{#1}\rangle}
\newcommand{\fp}{\varPi^a}

\newcommand{\Aa}{\cl{A}_\alpha}

\newcommand{\Ra}[1]{\cl{R}(\alpha, {#1})}
\newcommand{\Rza}{\cl{G}_\alpha}
\newcommand{\Ja}[1]{\cl{R}_h(\alpha, {#1})}
\newcommand{\Jaa}[2]{\cl{R}_h({#1}, {#2})}

\newcommand{\Ea}[2]{\cl{e}_h^{#1}(\alpha, {#2})}

\newcommand{\clu}{{\cl{u}}}
\newcommand{\clv}{{\cl{v}}}
\newcommand{\clvh}{\clv_h}

\newcommand{\clw}{{\cl{w}}}
\newcommand{\clwh}{{\clw_h}}
\newcommand{\clz}{{\cl{z}}}
\newcommand{\clzh}{{\clz_h}}
\newcommand{\clph}{{\cl{p}_h}}

\newcommand{\Fht}{{\cl{F}_{h,\theta}}}

\newcommand{\Ua}{{\cl{U}_\alpha}}
\newcommand{\Uap}{{\cl{U}'_\alpha}}
\newcommand{\Uah}{{\cl{U}_{\alpha, h}}}

\newcommand{\Va}{{\cl{V}_\alpha}}
\newcommand{\Vap}{{\cl{V}'_\alpha}}
\newcommand{\Vah}{\cl{V}_{\alpha, h}}
\newcommand{\VK}{\cl{V}(K)}
\newcommand{\VhK}{\cl{V}_h(K)}
\newcommand{\cQ}{\cl{Q}}
\newcommand{\cR}{\cl{R}}

\newcommand{\om}{\varOmega}
\newcommand{\oh}{{\varOmega_h}}

\newcommand{\Frechet}{Fr{\'{e}}chet}
\newcommand{\RT}{RT}

\newcommand{\cdB}{c_{d}}
\newcommand{\CdB}{C_{d}}

\newcommand{\Ho}{\mathring{H}}

\newcommand{\Hdivoh}{H(\dive, \oh)}
\newcommand{\Hoh}{\Ho^{1/2}(\partial\oh)}
\newcommand{\Hmh}{H^{-1/2}(\partial\oh)}
\newcommand{\Hdiv}{H(\dive)}

\renewcommand{\Omega}{\om}